\documentclass{amsart}
\usepackage{tikz}
\usepackage{xcolor}
\usepackage{amssymb,latexsym,amsmath,extarrows,mathabx}
\usepackage{graphicx,mathrsfs}
\usepackage{hyperref,url}
\numberwithin{equation}{section}

\newtheorem{theorem}{Theorem}[section]
\newtheorem{lemma}[theorem]{Lemma}

\newtheorem{proposition}[theorem]{Proposition}
\newtheorem{remark}[theorem]{Remark}

\newtheorem{problem}[theorem]{Problem}
\newtheorem{definition}[theorem]{Definition}
\newtheorem{corollary}[theorem]{Corollary}

\newtheorem{conjecture}[theorem]{Conjecture}

\newcommand{\al}{\alpha}
\newcommand{\be}{\beta}
\newcommand{\ga}{\gamma}

\newcommand{\de}{\delta}

\newcommand{\De}{\Delta}
\newcommand{\e}{\varepsilon}

\newcommand{\la}{\lambda}

\newcommand{\La}{\Lambda}
\newcommand{\si}{\sigma}

\newcommand{\vp}{\varphi}
\newcommand{\om}{\omega}

\newcommand{\eps}{\eta}

\newcommand{\cq}{\mathcal Q}
\newcommand{\cp}{\mathcal P}
\newcommand{\ch}{\mathcal H}

\newcommand{\cb}{\mathcal B}
\newcommand{\ce}{\mathcal E}
\newcommand{\cd}{\mathcal D}
\newcommand{\ci}{\mathcal I}
\newcommand{\cj}{\mathcal J}

\newcommand{\wh}{\widehat}

\newcommand{\ZR}{\mathbb{R}}
\newcommand{\ZT}{\mathbb{T}}
\newcommand{\ZZ}{\mathbb{Z}}

\newcommand{\ZC}{\mathbb{C}}

\newcommand{\ZS}{\mathbb{S}}

\newcommand{\Id}{{\bf 1}}

\newcommand{\ck}{{\mathcal K}}
\newcommand{\cT}{{\mathcal T}}
\newcommand{\cQ}{{\mathcal Q}}

\newcommand{\cl}{{\mathcal L}}

\newcommand{\cv}{{\mathcal V}}

\newcommand{\ang}{\measuredangle}

\newcommand{\supp}{{\rm supp}}

\author{Hong Wang}

\address{Courant institute of mathematical sciences, New York University}

\email{hw3639@nyu.edu}

\author{Shukun Wu}

\address{Department of Mathematics, Indiana University Bloomington}

\email{shukwu@iu.edu }

\title[Restriction using decoupling and two-ends Furstenberg]{Restriction estimates using decoupling theorems and two-ends Furstenberg inequalities
}

\begin{document}

\begin{abstract}
We propose to study the restriction conjecture using decoupling theorems and two-ends Furstenberg inequalities. 
Specifically, we pose a two-ends Furstenberg conjecture, which implies the restriction conjecture.
As evidence, we prove this conjecture in the plane by using the Furstenberg set estimate.
Moreover, we use this planar result to prove a restriction estimate for $p>22/7$ in three dimensions, which implies Wolff's $5/2$-hairbrush bound for Kakeya sets in $\mathbb{R}^3$. 
Our approach also makes improvements for the restriction conjecture in higher dimensions.
\end{abstract}

\maketitle


In this paper, we propose to study the Fourier restriction conjecture via decoupling theorems and two-ends Furstenberg inequalities. For a $C^2$ hypersurface $S\subset\ZR^n$, $n\geq2$, the Fourier extension operator $E_S$ is defined as
\begin{equation}
    E_Sf(x):=\int_Se^{ix\cdot \xi}f(\xi)d\si_S(\xi).
\end{equation}
Here $f:S\to\ZC$ and $\si_S$ is the surface measure on $S$. 

Regarding the extension operator, Stein \cite{Stein-restriction} made the following conjecture.
\begin{conjecture}
\label{restriction-conj}
Suppose that $S\subset\ZR^n$ is a compact $C^2$ hypersurface (maybe with boundary) with a strictly positive second fundamental form. 
Then when $p>\frac{2n}{n-1}$,
\begin{equation}
\label{restriction-esti}
    \|E_Sf\|_p\leq C_p \|f\|_{L^p(d\si_S)}.
\end{equation}
\end{conjecture}
Our main theorem is the following:
\begin{theorem}
\label{main-3d}
Conjecture \ref{restriction-conj} is true when $n=3$ and $p>22/7$.
\end{theorem}
\noindent Our approach can also make progress for the restriction conjecture in higher dimensions. 
See Theorem~\ref{main-high-d} in Section \ref{higher-dimension-intro}.
\medskip

\subsection{Overview} 

Fefferman \cite{Fefferman-inequalities} solved Conjecture \ref{restriction-conj} in the plane and provided partial results in higher dimensions. 
After Fefferman, Tomas \cite{Tomas} showed that \eqref{restriction-esti} holds for $p>\frac{2(n+1)}{n-1}$ by a $TT^\ast$ method. 
In fact, Tomas proved the $L^2$ estimate $\|E_Sf\|_p\leq C_p\|f\|_2$ for $p>\frac{2(n+1)}{n-1}$, which is a quite complete result since this estimate fails if $p<\frac{2(n+1)}{n-1}$. 
The endpoint case $p=\frac{2(n+1)}{n-1}$ was later settled by Stein, so this $(L^2, L^{\frac{2(n+1)}{n-1}})$ estimate is known as the Stein-Tomas estimate. 

\smallskip

The modern era of the restriction conjecture started in 1991, when Bourgain published his first article \cite{Bourgain-Besicovitch} in this field. 
His idea is to study $E_Sf$ by decomposing it into wave packets: For a large $R$, we first decompose the function $f$ as $f=\sum_\theta f\Id_\theta$. 
Each $\theta\subset S$ is a cap of radius $R^{-1/2}$, and $\{\theta\}$ forms a covering of $S$. 
An important observation is that  $|E_Sf_{\theta}|$ is essentially constant in any tube $T$ of dimensions $R^{1/2}\times\cdots\times R^{1/2}\times R$, which we call an $R$-tube, pointing in direction $\theta$.
Given a $\theta$, let $\ZT_\theta=\{T\}$ be a family of $R$-tubes pointing in direction $\theta$ that forms a finite-overlapping covering of $B^n(0,R)$. 
We then break each function $f_\theta=\sum_{T\in\ZT_\theta}f_T$ for some functions $\{f_T\}$ so that inside $B^n(0,R)$, $E_Sf_T$ is essentially supported on $T$, and $|E_Sf_T|$ is essentially constant on $T$ as well. 
Each $f_T$ is called a ``wave packet", and this gives us the ``wave packet decomposition" of $f$: $f=\sum_{\ZT}f_T$ with $\ZT=\bigcup_\theta\ZT_\theta$. 

For each wave packet $f_T$, $E_Sf_T$ has two key attributes: its oscillation, which maintains roughly the same amplitude on $T$; its support, which is $T$ essentially. 
Thus, armed with the wave packet decomposition, the study of the function $E_Sf$ can be divided into two main parts: 
\begin{enumerate}
    \item For each $x$, study the oscillation among all $\{E_Sf_T(x):x\in T\}$.
    \item For a given collection of $R$-tubes $\ZT$, study the geometric interference among $T\in\ZT$, which is commonly referred to as Kakeya-type problems.
\end{enumerate}
In Bourgain's original work \cite{Bourgain-Besicovitch}, the interference between different wave packets is studied locally via the Stein-Tomas estimate, and the global geometric interaction among the thin tubes of wave packets is studied by using Kakeya-type inequalities.

\smallskip

An important milestone towards the restriction conjecture was the work of Wolff \cite{Wolff-bilinear} and Tao \cite{Tao-bilinear}, where a multi-scale method known as ``induction on scales" was initiated. 
Broadly speaking, this method provides the following dichotomy: 
Either we obtain the desired result by induction, or we can observe extra geometric information regarding the thin tubes associated to wave packets.

In 2014, there were two breakthroughs in the field of restriction theory: the application of polynomial partitioning to Fourier analysis by Guth \cite{Guth-R3}, and the proof of the $\ell^2$-decoupling theorem by Bourgain and Demeter \cite{Bourgain-Demeter-l2}. 
The polynomial method is an algebraic tool that is powerful in studying Kakeya-type inequalities, whereas the decoupling theorems are powerful in studying interference between wave packets. 

\smallskip

In this paper, the tools to study oscillation are the method of induction on scales and decoupling theorems, and the tools to study geometric interference between wave packets (incidence problems between balls and tubes) are two-ends Furstenberg inequalities. 

\medskip

\subsection{Tools for oscillation.} The decoupling phenomenon was first observed by Wolff \cite{Wolff-decoupling}. 
It culminated in Bourgain-Demeter's resolution of the $\ell^2$-decoupling theorem \cite{Bourgain-Demeter-l2}. 
What we are using here is an influential refinement known as the ``refined decoupling theorem". 
This result (stated below) was presented in \cite{GIOW} and observed independently by Du-Zhang.

\begin{theorem}
Let $E_S$ be the extension operator for a strictly convex $C^2$ hypersurface $S$ with a Gaussian curvature $\sim1$, and let $p=\frac{2(n+1)}{n-1}$. Suppose $f$ is a sum of wave packets $f=\sum_{T\in\ZT}f_T$ so that $\|E_Sf_T\|_{L^p(w_{B_R})}^2$ are  the same up to a constant multiple for all $T\in\ZT$. 
Let $X$ be a union of $R^{1/2}$-balls in $B_R$ such that each $R^{1/2}$-ball $Q\subset X$ intersects to at most $M$ tubes $\ZT$. 
Then 
\begin{equation}
    \|E_Sf\|_{L^p(X)}^p\lessapprox M^{\frac{2}{n-1}}\sum_{T\in\ZT}\|E_Sf_T\|_{L^p(w_{B_R})}^p.
\end{equation}
Here $w_{B_R}$ is a weight that is $\sim1$ on $B_R$ and decreases rapidly outside $B_R$. 
\end{theorem}

If each $\ZT_\theta$ contains $\sim m$ many $R$-tubes, then we have $\sum_{T\in\ZT_\theta}\|E_Sf_T\|_{L^p(w_{B_R})}^p\sim (mR^\frac{n+1}{2})^{1-\frac{p}{2}}\big(\sum_{T\in\ZT_\theta}\|E_Sf_T\|_{L^2(w_{B_R})}^2\big)^{\frac{p}{2}}$. This is because each $E_Sf_T$ is essentially constant on $T$ and because $|T|\sim R^\frac{n+1}{2}$.
Note that each $f_\theta$ is supported in a cap $\theta$ with measure $\sim R^{-\frac{n-1}{2}}$.
By Plancherel and H\"older's inequality, we get $\sum_{T\in\ZT_\theta}\|E_Sf_T\|_2^2\lesssim R \|f_\theta\|_2^2\lesssim R (R^{-\frac{n-1}{2}})^{1-\frac{2}{p}}\|f_\theta\|_p^2$.
These calculations show that when $p=\frac{2(n+1)}{n-1}$,
\begin{equation}
\label{decoupling-end}
    \|E_Sf\|_{L^p(X)}^p\lessapprox (Mm^{-1})^{\frac{2}{n-1}}R^{-1}\|f\|_p^p.
\end{equation}
Therefore, the refined decoupling theorem suggests that one can control $\|E_Sf\|_{L^p(X)}$ by the multiplicity factor $M$ and $m$. 
Estimating these factors is indeed a Kakeya-type problem, that is, a tube-ball incidence problem. 

\smallskip

To set up this problem, we first make the following simple but crucial observation: Suppose that each $R$-tube $T\in\ZT$ intersects $\sim \la R^{1/2}$ many $R^{1/2}$-balls in $X$, then via $L^2$-orthogonality on each $R^{1/2}$-balls in $X$,
\begin{equation}
\label{plancherel-end}
    \|E_Sf\|_{L^2(X)}^2\lesssim \la R\|f\|_2^2.
\end{equation}
In other words, the smaller the $\lambda$, the better estimate we have in the $L^2$-space.

Now we set up the incidence problem as follows:

\begin{problem}
\label{incidence-problem}
\rm
Let $M(Q)=\#\{T\in\ZT:  T\cap Q\not=\varnothing\}$ be the multiplicity on an $R^{1/2}$-ball $Q\subset X$. 
Suppose that each $\ZT_\theta$ contains $\sim m$ many $R$-tubes, and for each $R$-tube $T\in\ZT$, the shading $Y(T)=T\cap X$ contains $\sim \la R^{1/2}$ many $R^{1/2}-$balls.
We want an upper bound for the multiplicity factor $M(Q)$ for a generic $Q\subset X$ using the information of $\la$ and $m$.
\end{problem}

One strong tool to study this incidence problem is the Kakeya maximal inequality. 
Specifically, the Kakeya maximal conjecture asserts that for a generic $Q\subset X$, the multiplicity factor $M(Q)$ can be bounded by $m$ and the density $\la$ as
\begin{equation}
\label{multi-bound-1}
    M(Q)\lessapprox m\la^{1-n}.
\end{equation}
The loss of $m$ in \eqref{multi-bound-1} agrees with the gain in \eqref{decoupling-end} perfectly. 
However, even with the full strength of the Kakeya maximal inequality, the loss of $\la$ in \eqref{multi-bound-1} is too big, and cannot be compensated by the gain in \eqref{plancherel-end}. 

\smallskip

The tool that allows us to improve upon \eqref{multi-bound-1} is the method of induction on scales.
In our setting, roughly speaking, the induction on scales provides the following dichotomy: 
Either we prove the desired restriction estimate by induction, or the shading $Y(T)$ is distributed on both ends of the tube $T$. 
With this extra two-ends spacing information on the shading $Y$, the exponent on $\la$ in \eqref{multi-bound-1} is expected to improve significantly.
Similar observations were also made in \cite{Wolff-Kakeya,Wolff-X-ray,Katz-Tao-Kakeya-maximal}. 

For example, in Bourgain's work \cite{Bourgain-Besicovitch}, he proved a Kakeya maximal inequality $M(Q)\lessapprox m \la^{-4/3}R^{1/3}$ when $n=3$. 
With little effort, this bound can be improved to $M(Q)\lessapprox m \la^{-2/3}R^{1/3}$ if there is a two-ends assumption on the shading $Y$. 
Plugging the improved bound back to \eqref{decoupling-end} with $n=3$, we end up with
\begin{equation}
    \|E_Sf\|_{L^4(X)}^4\lessapprox (\la R)^{-2/3}\|f\|_4^4.
\end{equation} 
Interpolate this estimate with the $L^2$ estimate \eqref{plancherel-end}, we have $\|E_Sf\|_{L^p(X)}^p\lessapprox \|f\|_p^p$ when $p=3.2$.
Via a dyadic pigeonholing argument and a global-to-local reduction (known as an $\epsilon$-removal argument), this implies $\|E_Sf\|_{p}^p\leq C_p \|f\|_p^p$ when $p>3.2$.
In other words, we prove Conjecture \ref{restriction-conj} when $n=3$ and $p>3.2$. 

\begin{remark}
\rm

Combining Wolff's hairbrush structure with Bourgain's bush structure, one can improve $M(Q)\lessapprox m \la^{-2/3}R^{1/3}$ to $M(Q)\lessapprox m\cdot \min\{\la^{-1}R^{1/4}, R^{1/2}\}$ through a simple application of the two-ends assumption.
This bound still yields $p>3.2$ for Conjecture \ref{restriction-conj} when $n=3$.
\end{remark}

\medskip

\subsection{Tools for incidence}

The calculation in the last subsection gives a strong restriction inequality (it covers all prior findings in $\ZR^3$), and, surprisingly, the tool for the incidence problem is the ``bush" structure discovered by Bourgain back in 1991.
It is therefore intriguing to explore the strongest incidence estimates under a two-ends assumption.

\smallskip

Let us first take a look at the lowest dimension $n=2$. A powerful incidence result known as the Furstenberg set estimate was obtained recently in \cite{Orponen-Shmerkin-2,ren2023furstenberg}. 
Before stating this result, we quickly go through some necessary notations.
Suppose $\de\in(0,1)$ is a small number. 
For a set $E\subset [0,1]^n$, let $\cd_\de(E)$ be the smallest family of $\de$-balls that forms a covering of $E$. 
Given an $s\in(0,n]$, the set $E$ is called a $(\de,s)$-set if $\#\cd_\de(E\cap B)\lesssim r^s\#\cd_\de(E)$ for all $r\in[\de,1]$ and all $r$-balls $B\subset[0,1]$. 
Given a family $\cT$ of non-horizontal $\delta\times1$-tubes in the unit ball, by the point-line duality (Definition \ref{point-line-duality}), we can identify $\cT$ as a family $\cp_\cT$ of $\de$-balls inside the unit ball. 
We say $\cT$ is a $(\de,s)$-set if $\cup_{\cp_\cT}$ is a $(\de,s)$-set.

Now we state the Furstenberg set estimate.

\begin{theorem}[\cite{ren2023furstenberg}]
\label{furstenberg-thm}
Let $t\in(0,2]$, $s\in(0,1]$, and $\la\in[\de,1]$. 
Suppose that $\cT$ is $(\de,t)$-set of $\de\times1$-tubes in the unit ball, and for each $T\in\cT$, there is a shading $Y(T)\subset T$ so that $Y(T)$ is a $(\de,s)$-set and $|Y(T)|\sim\la |T|$. 
Then 
\begin{equation}
    \big|\bigcup_{T\in\cT} Y(T)\big|\gtrapprox \la\de\cdot\de^{-\min\{t,\frac{s+t}{2},1\}}.
\end{equation}
\end{theorem}

\noindent Note that when $\#\cT\lesssim\de^{-t}$, $t=1$, and $\la\sim\de^{1-s}$, Theorem \ref{furstenberg-thm} implies that a generic $\de$-ball contained in $\bigcup_{T\in\cT}Y(T)$ intersects $\lessapprox \de^{-\frac{1-s}{2}}\lesssim\la^{-1/2}$ tubes in $\cT$. 

\smallskip

Let us see how this theorem and the calculation can help Problem \ref{incidence-problem} when $n=2$.
Assume that
\begin{enumerate}
    \item For all $T\in\ZT$, the $R^{-1}$-dilate of $Y(T)$ is an $(R^{-1/2},s)$-set with $\la=R^{\frac{s-1}{2}}$;
    \item $\ZT_\theta$ is not empty for $\sim R^{1/2}$ caps $\theta$.
\end{enumerate}
Then, by the triangle inequality, the calculation above shows that for a generic $Q\subset X$, the multiplicity factor $M(Q)$ can be bounded above by $m\la^{-1/2}$. 
Plug this back to \eqref{decoupling-end} with $n=2$ so that
\begin{equation}
    \|E_Sf\|_{L^6(X)}^6\lessapprox (\la R)^{-1}\|f\|_6^6.
\end{equation} 
An interpolation with the $L^2$ estimate \eqref{plancherel-end} shows $\|E_Sf\|_{L^4(X)}^4\lessapprox \|f\|_4^4$ when $n=2$. 
That is, we prove the restriction conjecture in the plane using decoupling and the Furstenberg set estimate, however, with two additional assumptions.

\smallskip

The real challenge is the first assumption ``the $R^{-1}$-dilate of $Y(T)$ is an $(R^{-1/2},s)$-set with $\la=R^{\frac{s-1}{2}}$".
Being a $(R^{-1/2},s)$-set means that the shading $Y(T)$ possesses a strong spacing condition. 
In particular, it suggests that the $R^{-1}$-dilate of $Y(T)$ behaves like an $s$-dimensional set.
However, in Problem \ref{incidence-problem}, the spacing condition on $Y(T)$ is merely two-ends, which, roughly speaking, only gives ``the $R^{-1}$-dilate of $Y(T)$ is an $(R^{-1/2},s)$-set with some $s$ obeying $R^{-s}\approx 1$". 
This statement is much weaker than the first assumption, as generally $\la$ is much smaller than $1$, and it can be as small as $R^{-1/2}$. 

\smallskip

In \cite[Corollary 1.5]{Demeter-Wang}, the $s$-dimensional spacing assumption on $Y(T)$ was removed and replaced by a product spacing assumption on $\cT$.  
The main result in \cite{Demeter-Wang} does not quite give what we need here, but the arguments and proof scheme inspire us to make the following crucial and perhaps surprising observation:

\medskip

With merely the two-ends spacing condition, we can obtain $M(Q)\lessapprox m\la^{-1/2}$ for Problem \ref{incidence-problem} in the plane.

\begin{theorem}[c.f. Theorem \ref{two-ends-furstenberg}]
Suppose that $\ZT$ is a collection of $R^{1/2}\times R$-tubes with $\#\ZT_\theta\leq m$.
Suppose also that there exists a $\la\in[R^{-1/2},1]$ such that $|Y(T)|\sim\la |T|$, and $Y(T)$ is two-ends for each $T\in \ZT$.
Then
\begin{equation}
\label{two-end-Kakeya}
    \big|\bigcup_{T\in \ZT}Y(T)\big|\gtrapprox \la^{3/2}(R^{3/2}m^{-1}\#\ZT).
\end{equation}
In particular, for a typical $Q\subset\bigcup_{T\in \ZT}Y(T)$, $M(Q)\lessapprox m\la^{-1/2}$.
\end{theorem}

The bound $M(Q)\lessapprox m\la^{-1/2}$ matches what one can get from the Furstenberg set estimate. 
Especially, it shows that we indeed can prove the restriction conjecture in $\ZR^2$ using decoupling and \eqref{two-end-Kakeya}, which we call a two-ends Furstenberg inequality.

\medskip

\subsection{Two-ends Furstenberg inequality: consequence and conjecture}

By using Wolff's hairbrush structure and the two-ends Furstenberg inequality in the plane, we prove $M(Q)\lessapprox m\la^{-3/4}R^{1/4}$ for Problem \ref{incidence-problem} when $n=3$ (c.f. Lemma \ref{two-ends-hairbrush-lem}). Consequently, we get Theorem~\ref{main-3d}.

\smallskip

When $n=3$, the exponent $p>22/7$ corresponds to the exponent $5/2$ for the Kakeya conjecture. 
To compare, the exponent $p>3.2$ corresponds to the exponent $7/3$ for the Kakeya conjecture, which was proved by Bourgain \cite{Bourgain-Besicovitch} in 1991. We will discuss the numerologies in Section \ref{ending-remark}.

On the other hand, the exponent $5/2$ is almost the best-known result for the Kakeya conjecture. 
It was obtained by Wolff \cite{Wolff-Kakeya} in 1995, and improved to $5/2+\epsilon_0$ in 2017 by Katz-Zahl \cite{Katz-Zahl-1} (see also \cite{Wu-regulus-strip}).

\medskip

Let us return to general dimensions. 
As mentioned at the beginning of the previous subsection, we would like to find out the strongest two-ends Furstenberg inequality. 
Given that the two-ends Furstenberg inequality and the Furstenberg set estimate have the same numerology for the restriction problem in $\ZR^2$, we first explore examples and evidence for the higher-dimensional Furstenberg set estimate.

In \cite{GSW-19}, the authors studied the Furstenberg set estimate under a ``well-spaced" assumption in all dimensions. We refer to their paper for the precise definition of well-spaced tubes.

\begin{theorem}[\cite{GSW-19} Theorem 4.1]
\label{well-spaced-thm}
Suppose $\cT$ is a collection of ``well-spaced" $\de$-tubes with length 1 and cross-section radius $\de$.
Then for $n=2$ and $3$,
\begin{equation}
\label{well-spaced-esti}
    \#\cp_r(\cT)\lessapprox\frac{(\#\cT)^\frac{n}{n-1}}{r^\frac{n+1}{n-1}},
\end{equation}
where $\cp_r(\cT):=\{B\subset\cup_\cT:B\cap T\not=\varnothing\text{ for $\geq r$ tubes $T\in\cT$}\}$ denotes the set of $r$-rich $\de$-balls of $\cT$.
\end{theorem}

In our notation, Theorem \ref{well-spaced-thm} shows that if $\bigcup_{T\in \ZT}Y(T)$ is a union of well spaced $R^{1/2}$-balls, then when $n=2$ and $3$,
\begin{equation}
    \big|\bigcup_{T\in \ZT}Y(T)\big|\gtrapprox \la^{\frac{n+1}{2}}(R^{\frac{n+1}{2}}\#T).
\end{equation}
In particular, we have
\begin{equation}
\label{multi-bound-2}
    \#M(Q)\lessapprox m\la^{-\frac{n-1}{2}}
\end{equation}
for Problem \ref{incidence-problem} under the well-spaced assumption on balls when $n=2$ and $3$. 

\smallskip

Notice that \eqref{multi-bound-2} is stronger than \eqref{multi-bound-1}, which is predicted by the Kakeya maximal inequality. 
Moreover, the following discrete example shows that \eqref{multi-bound-2} is likely the best possible bound: 

Let $\cp$ be a set of $n$-dimensional lattice points $[0,N]\times[0,k_1N]\times\cdots [0,k_{n-1}N]$ for some $k_1,\ldots, k_{n-1}\geq1$. Let $\cl$ be the set of lines $\ell_{a,b}:(a,0)+\ZR (1,b):a\in[1,k_1N]\times\cdots\times[1,k_{n-1}N]\cap\ZZ^{n-1}, b\in[1,k_1]\times\cdots\times[1,k_{n-1}]\cap\ZZ^{n-1}$. 
Then $\#\cp\sim N^n\prod k_j$, $\#\cl=N^{n-1}\prod k_j^2$. Moreover, the incidence $\ci(\cp,\cl)=\sum_{p\in\cp,\ell\in\cl}\chi(p,\ell)
\sim N^n\prod k_j^2$. 
Therefore, we have
\begin{equation}
    \ci(\cp,\cl)\sim (\#\cp)^\frac{2}{n+1}(\#\cl)^\frac{n}{n+1},
\end{equation}
which matches the numerology in \eqref{well-spaced-esti} and hence in \eqref{multi-bound-2}.

\smallskip

In the absence of unanticipated phenomena, we might hazard a guess on the best possible estimate for the two-ends Furstenberg inequality. 
Here we use the notations in Definitions \ref{shading-def}, \ref{E-L-def}, \ref{lambda-dense-def}, and \ref{two-ends-def}.
\begin{conjecture}
\label{two-ends-kakeya-conj}
Let $\de\in(0,1)$.
Let $(L,Y)_\de$ be a set of directional $\de$-separated lines in $\ZR^n$ with an $(\e_1,\e_2)$ two-ends, $\lambda$-dense shading. 
Then for any $\e>0$,
\begin{equation}
    |E_L|\geq c_\e\de^{\e}\de^{O(\e_1)} \la^{\frac{n-1}{2}}\sum_{\ell\in L}|Y(\ell)|.
\end{equation}
\end{conjecture}

Apply an $R$-dilate version of Conjecture \ref{two-ends-kakeya-conj} to  Problem \ref{incidence-problem} to have $M(Q)\lessapprox m\la^{-\frac{n-1}{2}}$. 
Thus, plugging this back to \eqref{decoupling-end} and interpolating with \eqref{plancherel-end}, we get
\begin{theorem}
\label{proof-of-restriction}
Conjecture \ref{two-ends-kakeya-conj} implies Conjecture \ref{restriction-conj}.
\end{theorem}

\begin{remark}
\rm

There might be an analog of Conjecture~\ref{two-ends-kakeya-conj} if directional separation is replaced by a more general assumption. 
See Remark~\ref{rmk: two-ends-furstenberg}. 
We do not pursue it here due to the lack of knowledge in higher dimensions. 
\end{remark}

\medskip 

\subsection{Analogues of Furstenberg sets conjecture in \texorpdfstring{$\mathbb{R}^n$}{}}

Conjecture~\ref{two-ends-kakeya-conj} is the higher dimensional version of Theorem~\ref{two-ends-furstenberg}.
Here we include higher dimensional versions of the Furstenberg sets conjecture, which might serve as a (difficult) intermediate step towards Conjecture \ref{two-ends-kakeya-conj}. 

\smallskip

Let $\cT$ be a set of distinct $\delta$-tubes in $\mathbb{R}^n$. Let $U\subset \mathbb{R}^n$ be an open set, define
\begin{equation}
    \cT[U]:=\{ T \in \cT: T\subset U\}. 
\end{equation}

\begin{definition} \label{def: convexwolff}
Let $\de\in(0,1)$ and let $t\in (0, 2)$.
We say that a family of $\de$-tubes $\cT$ satisfies the {\bf $t$-Frostman Convex Wolff axiom with error $C$} if for any convex set $U\subset \mathbb{R}^n$, we have 
\begin{equation}
    \#\cT[U] \leq C |U|^t  \, \#\cT. 
\end{equation}
In particular, $\# \cT\geq C^{-1} \delta^{-(n-1)t}.$
\end{definition}

What motivates us to define $t$-Frostman Convex Wolff Axiom is the following: 
When $n=2$, a set of $\delta$-tubes satisfying the $t$-Frostman Convex Wolff Axiom with error $C$  is the same as a $(\delta, t, C)$-set of $\delta$-tubes. 
However, when $n\geq 3$, a $(\delta, t, C)$-set of $\delta$-tubes could be trapped in the $\delta$-neighborhood of a $k$-plane for some $2k\geq t$, which is considered to be a degenerate scenario. 
Thus, the notion ``$(\delta, t, C)$-set" is not an appropriate notion to study $\cup_{T\in \cT} Y(T)$ in higher dimensions. 

The idea of using convex sets to ``pack'' tubes was first introduced in  \cite{wang2024assouad} in the study of Kakeya sets in $\mathbb{R}^3$, where the authors studied a set of $\delta$-tubes satisfying the Convex Wolff Axiom: $\# \cT[U] \leq C|U| \, \#\cT$. 
This is precisely the $1$-Frostman Convex Wolff Axiom. 
In Definition \ref{def: convexwolff}, the factor $t$ no longer stands for dimension as it does in the definition of a $(\delta, t,C)$-set. 
Instead, it quantifies how non-concentrated $\cT$ is in any convex set. 

\smallskip

Let $E\subset \mathbb{R}^n$ be a set. 
Denote by $|E|_{\delta}$ the minimum number $\delta$-balls required to cover $E$. 
\begin{conjecture} \label{conj: Furstenberg3}
Let $t\in (0,2)$.
For any $\e>0$, there exists $\eta>0$ such that the following holds for $\delta>0$ sufficiently small. 
Suppose that $\cT$ is a set of $\delta$-tubes in $\mathbb{R}^3$ satisfying  $t$-Frostman Convex Wolff Axiom with error $\delta^{-\eta}$, and for each $T\in \cT$, $Y(T)$ is a $(\delta, s, \delta^{-\eta})$-set.  
Then 
\begin{equation}
     |\bigcup_{T\in \cT} Y(T)|_{\delta} \geq \delta^{-\min \{ s+ 2t, t+2s, 2+s\}+\e}.
\end{equation}
\end{conjecture}

Conjecture~\ref{conj: Furstenberg3} is based on the $n$-dimensional lattice points example discussed in the previous section (take $n=3$), the bush example (for $s>t$),  and the example given by the product of a $(n-1)$-dimensional hyperplane (take $n=3$) with an $s$-dimensional set (for $s+t>2$).

\smallskip

For $n\geq 4$, we need the stronger Polynomial Wolff Axiom instead of Convex Wolff Axiom because of the following example in $\mathbb{R}^4$: 
Let $\cT$ be a set of $\delta$-tubes contained in the $\delta$-neighborhood of  $\{ xy-zw=1 \} \cap B(0, 10)$. 
Notice that $\cT$ satisfies $1$-Frostman Convex Wolff Axiom.
But if we let $Y(T)=T$, then $|\cup_{T\in \cT} Y(T)| \sim \delta$, which is a contradiction to estimate \eqref{eq: furstenbergn}, a natural generalization of Conjecture~\ref{conj: Furstenberg3} in higher dimensions.

\begin{definition}
Let $\de\in(0,1)$ and let $t\in (0, 2)$.
We say a family of $\de$-tubes $\cT$ satisfies { \bf $t$-Frostman Polynomial Wolff axiom with error $C$ } if for any semialgebraic set $U$ of complexity $O(1)$ (independent of $\delta$), we have 
\begin{equation}
    \#\cT[U] \leq C |U|^t  \, \#\cT. 
\end{equation}
In particular, $\# \cT\geq C^{-1} \delta^{-(n-1)t}.$
\end{definition}


\begin{conjecture} \label{conj: Furstenbergn}
Let $t\in (0,2)$ and $n\geq 4$. 
For any $\e>0$, there exists $\eta>0$ such that the following holds for $\delta>0$ sufficiently small. 
Suppose $\cT$ that is a set of $\delta$-tubes in $\mathbb{R}^n$ satisfying  $t$-Frostman Polynomial Wolff Axiom with error $\delta^{-\eta}$, and for each $T\in \cT$, $Y(T)$ is a $(\delta, s, \delta^{-\eta})$-set.  Then 
    \begin{equation}\label{eq: furstenbergn}
    |\bigcup_{T\in \cT} Y(T)|_{\delta} \geq \delta^{-\min \{ s+ (n-1) t, \frac{(n-1)t}{2}+\frac{n+1}{2}s, n-1+s\}+\e}.
    \end{equation}
\end{conjecture}

Note that the two-dimensional analog of Conejecture~\ref{conj: Furstenberg3} (for numerology, see Conjecture~\ref{conj: Furstenbergn} with $n=2$)  is precisely the Furstenberg sets conjecture in the plane, proved in Theorem~\ref{furstenberg-thm}.

\begin{remark}
\rm

Various interesting special cases emerged in the study of the Kakeya conjecture in $\ZR^3$. 
For example, the $\text{SL}_2$ example discovered in \cite{Katz-Zahl-1} and studied in \cite{Wu-regulus-strip}, the sticky Kakeya-set studied in \cite{WZ22}, and the $\text{SL}_2$-Kakeya set studied in \cite{FO-23, KWZ}, where the Kakeya conjecture was explored with an extra assumption on the set lines ($\de$-tubes).
Therefore, it will be interesting to first consider Conjecture \ref{conj: Furstenberg3} with similar additional assumptions on $\de$-tubes in $\cT$.

\end{remark}

\medskip

\subsection{Results in higher dimensions}
\label{higher-dimension-intro}

Although the two-dimensional two-ends Furstenberg inequality in Theorem \ref{two-ends-furstenberg} does not appear to be particularly useful in higher dimensions, with the help of another incidence estimate (a two-ends inequality proved in \cite{Katz-Tao-Kakeya-maximal}), the approach developed in this paper can still obtain an improvement for Conjecture \ref{restriction-conj} for general $n$.

\begin{theorem}
\label{main-high-d}
Conjecture \ref{restriction-conj} is true when $p>p(n)$, where
\begin{equation}
\label{p(n)}
    p(n)=\frac{154n+6}{77n-95}=2+\frac{28}{11n}+O(n^{-2}).
\end{equation}
\end{theorem}

Notice that $28/11=2.545454...\,$. Therefore, it improves the best-known result for Conjecture \ref{restriction-conj} obtained in \cite{GuoWZ} for general large $n$.
\begin{remark}
\rm

We do not attempt to optimize the best $p(n)$ our approach can obtain, since the optimal exponent is expected to have the same (leading term) asymptotic behavior as $p(n)$. 

\end{remark}

\medskip

\subsection{Outline of the paper and notations}

In Section~\ref{section: incidence geometry}, we review the background related to incidence geometry, including the multi-scale decomposition and some (standard) notations and lemmas.
In Section~\ref{section: two-ends Furstenberg2d}, we prove the two-ends Furstenberg inequality in the plane.
Section~\ref{section: twoends Kakeyand} contains some higher-dimensional two-ends Furstenberg inequalities, which serve as geometric tools for our theorems in the restriction conjecture.
Section~\ref{section: fourier analysis} contains a review of Fourier analysis.
Finally, in Section~\ref{section: restriction estimates}, we prove the restriction estimates.

\bigskip

\noindent {\bf Notations:} Throughout the paper, we use $\# E$ to denote the cardinality of a finite set.
If $\ce$ is a family of sets in $\ZR^n$, we use $\cup_\ce$ to denote $\cup_{E\in\ce}E$.
For $A,B\geq 0$, we use $A\lesssim B$ to mean $A\leq CB$ for an absolute (big) constant $C$, and use $A\sim B$ to mean $A\lesssim B$ and $B\lesssim A$.
For a given $\de<1$, we use $ A \lessapprox B$ to denote $A\leq c_\e\de^{-\e} B$ for all $\e>0$ (same notation applies to a given $R>1$ by taking $\de=R^{-1}$).

\bigskip

\noindent
{\bf Acknowledgment.} We are grateful to Ciprian Demeter and Xiumin Du for their careful review of the early draft and valuable comments.

\bigskip


\section{Preliminaries in incidence geometry}\label{section: incidence geometry}

In this section, we introduce notations and basic tools in incidence geometry.

\begin{definition}[Refinement]
For two finite sets $E,F\subset\ZR^n$, we say $E$ is a  $\gtrsim c$-refinement of $F$, if $E\subset F$ and $\#E\gtrsim c\#F$; we say $E$ is a $\gtrapprox c$-refinement of $F$, or simply a {\bf refinement} of $F$, if $E\subset F$ and  $\#E\gtrapprox c\#F$.
A similar definition applies when $E,F$ are two finite unions of $\de$-balls.
\end{definition}

\begin{definition}
\label{E-rho}
Let $\delta\in(0,1)$ be a small number.
Let $E$ be a finite union of $\de$-balls in $\ZR^n$. 
For another $\rho\geq\de$, let
\begin{equation}
\label{rho-covering}
    |E|_\rho=\min\{\#\cd_\rho:\cd_\rho\text{ is a covering of $E$ by $\rho$-balls}\}.
\end{equation}
Let $\cd_\rho(E)$ be a family of $\rho$-balls that attains the minimum in \eqref{rho-covering}.
Denote by 
\begin{equation}
\label{E-sub-rho}
    (E)_\rho=\cup_{\cd_\rho(E)}.
\end{equation}
The same definitions apply to $\cp$ when it is a finite set of disjoint $\de$-balls in $\ZR^n$ by considering $\cup_{\cp}$.
\end{definition}

\begin{definition}[$(\de,s,C)$-set]
\label{delta-s}
Let $\delta\in(0,1)$ be a small number.
For $s\in(0,n]$, a non-empty set $E\subset \ZR^n$ is called a {\bf $(\de,s,C)$-set} (or simply a {\bf $(\de,s)$-set} if $C$ is not important in the context) if 
\begin{equation}
    |E\cap B(x,r)|_\de\leq Cr^s|E|_\de, \hspace{.3cm}\forall x \in\ZR^n, \,r\in[\de,1].
\end{equation}
\end{definition}

\begin{lemma}
\label{refinement-delta-s-set}
Let $\delta\in(0,1)$ and let $C_1, C_2\geq1$.
If a union of $\de$-balls $E$ is a $(\de,s, C_1)$-set and $E'$ is a $\geq C_2^{-1}$-refinement of $E$. Then $E'$ is a $(\de,s, C_1C_2)$-set. 
\end{lemma}
\begin{proof}
Just note that $|E|_\de\leq C_2|E'|_\de$.
\end{proof}

\begin{definition}[Katz-Tao $(\de,s,C)$-set]
Let $\delta\in(0,1)$ be a small number.
For $s\in(0,n]$, a finite set $E\subset \ZR^n$ is called a {\bf Katz-Tao $(\de,s,C)$-set} (or simply a {\bf Katz-Tao $(\de,s)$-set} if $C$ is not important in the context) if 
\begin{equation}
    \#(E\cap B(x,r))\leq C(r/\de)^s, \hspace{.3cm}\forall x \in\ZR^n, \,r\in[\de,1].
\end{equation}
\end{definition}

\begin{lemma}
\label{katz-tao-set-lem}
Let $0<\de<\rho<1$. 
Suppose $E$ is a Katz-Tao $(\de,s,C)$-set, where $1\leq C\lessapprox1$.
Then there exists $E'\subset E$ such that $\rho^s\#E'\gtrapprox\de^s\#E$, and $E'$ is a Katz-Tao  $(\rho,s,C')$-set for some $C'\lesssim1$.
\end{lemma}
\begin{proof}
Choose a probability $p=(\de/\rho)^s|\log\de|^{-1}C^{-1}$.
Let $E_1$ be a uniform random sample of $E$ with probability $p$.
Then with high probability, $\rho^s\#E_1\gtrapprox\de^s\#E$, and
\begin{equation}
\label{after-probability}
\Big\{
\begin{array}{cc}
    \#(E_1\cap B(x,r))\leq C_1(r/\rho)^s, & \hspace{.3cm}\forall x \in\ZR^n, \,r\in[\rho|\log\de|^2,1], \\[1ex]
    \#(E_1\cap B(x,r))\leq C_2(r/\rho)^s, & \hspace{.3cm}\forall x \in\ZR^n, \,r\in[\rho, \rho|\log\de|^2], 
\end{array}
\end{equation}
for some $C_1\lesssim1$ and $C_2\lessapprox1$.
Let $E_2\subset E_1$ be a maximal $\rho|\log\de|^2$-separated set.
By the second line of \eqref{after-probability}, we have $\#E_2\gtrapprox\# E_1$.
Note that the first line of \eqref{after-probability} still holds when $E_1$ is replaced by $E_2$.
Let $E'=E_2$ and $C'=C_1$. \qedhere

\end{proof}

\smallskip

\begin{definition}
\label{m-uniform}
Let $\delta\in(0,1)$ be a small number.
Let $M=|\log\de|$ and let $\rho_j = M^{j}$, $j=1,\dots, \lceil\log_M\de^{-1}\rceil$.
Given $E$ a union of $\de$-balls in $\ZR^n$, we say $E$ is {\bf uniform  with error $C$}  if $|E\cap D_{\rho_j}|$ are  the same up to a multiple of $C$ for all $D_{\rho_j}\subset(E)_{\rho_j}$. 
When $C$ is not important in the context, we say $E$ is {\bf uniform}. 
The same definitions apply to $\cp$ when it is a finite set of disjoint $\de$-balls in $\ZR^n$ by considering $\cup_{\cp}$.

\end{definition}

\begin{lemma}[Uniformization]
\label{uniformization}
Let $\delta\in(0,1)$ be a small number.
If $E\subset\ZR^n$ is a union of dyadic $\de$-balls, then there is a $\gtrsim (\log|\log\de|)^{-\frac{|\log\de|}{\log|\log\de|}}$-refinement $E'$ of $E$ such that $E'$ is uniform.
\end{lemma}

This lemma is standard, which follows from dyadic pigeonholing on $(E)_\rho$ at each scale $\rho=\rho_j$. 
It shows that for an arbitrary set, we can always find a refinement that is also uniform. 
The next definition serves as a preliminary tool to study a uniform set.

\begin{definition}[Branching function]
Let $\delta\in(0,1)$ and let $M=|\log\de|$.
Let $\rho_j =M^{-j},\, j=1, \dots,  \lceil \log_M\de^{-1}\rceil=: N $.
Let $E\subset\ZR^n$ be a union of $\de$-balls that is uniform.
Define a {\bf branching function} $\be_E:[0,N]\to[0,n]$ as
\begin{equation}
    \be_E(j)=\frac{\log(|E|_{\rho_j})}{ |\log \delta|}, \hspace{.3cm}0\leq j\leq N
\end{equation}
and interpolate linearly between $\beta_E( j)$ and $\beta_E(j+1)$. 
\end{definition}
\noindent Note that the branching function $\be_E$  characterizes the distribution of $\de$-balls in $E$.

\begin{lemma}
\label{uniform-sets-branching-lem}
Let $\delta\in(0,1)$, $M=|\log\de|$, and  $N=\lceil \log_M \delta^{-1} \rceil$.
Suppose $\ce$ is a finite family of uniform sets of $\de$-balls in $\ZR^n$.
Then there is subset $\ce'\subset\ce$ with $\# \ce'\gtrsim (n \log M)^{-N}\#\ce $ 
and a {\bf uniform branching function $\be_{\ce'}$ of $\ce'$} such that for any $E\in\ce'$, $|\be_{E}-\be_{\ce'}|\leq  |\log \delta|^{-1}$.    
In particular, $\#\ce'\gtrapprox\#\ce$.

\end{lemma}
\begin{proof} 
For any branching function $\beta: [0, N]\rightarrow [0,n]$, define an ``$\e$-ball" $D(\beta, \e)=\{ \beta': [0, N]\rightarrow [0,n]$, $|\beta(j)- \beta'(j)| \leq \e \text{ for all }0\leq j\leq N\}$.

Cover $\{\beta_E \}_{E\in \ce}$ by $D(\beta_E, \e)$. Since $\be_E(j+1)-\be_E(j)\in[0,n\frac{\log M}{|\log \delta|}]$ for all $0\leq j\leq N$ and since $\be_E(0)=0$, by pigeonholing, there exists $E\in \ce$ such that $\ce':=\{ E'\in \ce: \beta_{E'}\in D(\beta_E, \e) \}$ has cardinality $\#\ce' \geq  (n \frac{\log M}{|\log \delta|} \e^{-1})^{-N} \#\ce$. Take $\e= |\log \delta|^{-1}$.
\end{proof}

After rescaling, the branching function of a uniform set becomes a $1$-Lipschitz function. 
What follows is a powerful tool to analyze a $1$-Lipschitz function.
It was first introduced by Shmerkin \cite{shmerkin2023non} and has been developed into various forms.
The precise statement below can be found in \cite{Demeter-Wang}. 

\begin{lemma}
\label{lip-decomposition}
Let $\eta>0$ be a small number and let $\eta_0=\eta_0(\eta)=\eta^{2\eta^{-1}}$. Then for a non-decreasing 1-Lipschitz function $\be:[0,1]\to[0,1]$, there exists a partition
\begin{equation}
    0=A_1<A_2<\cdots<A_{H+1}=1
\end{equation}
and a sequence 
\begin{equation}
    0\leq s_1< s_2<\cdots< s_{H}\leq 1
\end{equation}
such that for each $1\leq h\leq H$, we have the following:
\begin{equation}
    A_{h+1}-A_h\geq\eta_0\eta^{-1};
\end{equation}
\begin{align}
    f(x)\geq& f(A_h)+s_l(x-A_h)-\eta(A_{h+1}-A_h), \text{ for all }x\in[A_h,A_{h+1}],\\
    &f(A_{h+1})\leq f(A_h)+(s_h+3\eta)(A_{h+1}-A_h);
\end{align}
\begin{equation}
    s_H\geq f(1)-f(0)-\eta.
\end{equation}
\end{lemma}

Suppose $E$ is a union of $\de$-balls that is uniform.
Apply this lemma to the branching function $\be_E$, we have a powerful characterization of the distribution of the set $E$.
Thus, together with Lemma \ref{uniform-sets-branching-lem}, we have

\begin{proposition}[Multi-scale decomposition]
\label{multiscale-prop}
Let $\eta>0$ be a small number and let $\eta_0(\eta)=\eta^{2\eta^{-1}}$. 
Let $\delta>0$ be sufficiently small with  $|\log \delta| \eta_0(\eta) >2$.

Suppose $\ce$ is a family of uniform unions of $\de$-balls and $\be_\ce$ is a uniform branching function such that $|\be_{E}-\be_{\ce}|\leq |\log \delta|^{-1}$ for all $E\in \ce$. 
Then there exists a partition
\begin{equation}
    0=A_1<A_2<\cdots<A_{H+1}=1
\end{equation}
and a sequence 
\begin{equation}
    0\leq s_1< s_2<\cdots< s_{H}\leq 1
\end{equation}
such that for each $1\leq h\leq H$ and all $E\in\ce$, we have the following:
\begin{enumerate}
    \item $A_{h+1}-A_h\geq\eta_0\eta^{-1}$.
    \item $\log_{1/\de}\big(\frac{|E|_{\de^{A_{h+1}}}}{|E|_{\de^{A_h}}}\big)\leq (s_h+4\eta)(A_{h+1}-A_h)$.
    \item For each $\de^{A_{h+1}}$-ball $D\subset (E)_{\de^{A_{h}}}$, the $\de^{-A_h}$-dilate of the set $ (E)_{\de^{A_{h+1}}}\cap D$ is a $(\de^{A_{h+1}-A_h}, s_h, O(\de^{4\eta(A_{h+1}-A_h)}))$-set.
    \item $s_H\geq \log_{1/\de}|E|_\de-\eta$.
\end{enumerate}
\end{proposition}

\medskip

If $L$ is a family of non-horizontal lines in $\ZR^2$, then each $\ell\in L$ can be parameterized as $\ell_{(a,b)}:(a,0)+\ZR(b,1)$. Let $P=\{(a,b):\ell_{(a,b)}\in L\}\subset\ZR^2$. Therefore, the map $F:\{\text{points in $\ZR^2$}\}\to\{\text{non-horizontal lines in $\ZR^2$}\}$, $F((a,b))=\ell_{(a,b)}$ is one-to-one. Moreover, if $P$ is contained in a compact set, then $F$ and $F^{-1}$ are both Lipschitz.

\begin{definition}[point-line duality]
\label{point-line-duality}
The map $F$ and its inverse $F^{-1}$ are both referred to as {\bf{point-line duality}}. 
\end{definition}

The point-line duality has the following crucial property: If $x\in\ZR^2$ is a point and $\ell\subset\ZR^2$ is a non-horizontal line, then $x\in \ell$ if and only if $F^{-1}(\ell)\in F(x)$. 

We can extend the point-line duality to a family of $\de$-balls in the unit ball and a family of $\de\times1$-tubes contained in the unit ball by noting the following: 
If $p\subset B(0,1)$ is a $\de$-ball and $T$ is a  $\de\times 1$-tube contained in the unit ball that is quantitatively transverse to the $x$-axis, then $p\cap T\not=\varnothing$ if and only if $F^{-1}(T)\cap F(p)\not=\varnothing$. 
Here $F^{-1}(T)$ is the union of points $F^{-1}(\ell)$ where $\ell\cap T\subset T$. 
Thus, a family of $\de\times1$-tubes can be identified as a family of $\de$-balls, and we can make the following definition.

\begin{definition}
Let $\cT$ be a family of $\de\times1$-tubes contained in the unit ball. We say $\cT$ is {\bf uniform} if $\{F^{-1}(T):T\in\cT\}$, as a set of $\de$-balls, is uniform. 
Similarly, we say $\cT$ is a  {\bf $(\de,s)$-set} (equiv. {\bf Katz-Tao $(\de,s)$-set}) if $\{F^{-1}(T):T\in\cT\}$ is a $(\de,s)$-set (equiv. Katz-Tao $(\de,s)$-set).
Same definitions apply to a family $\de$-separated lines by considering $(F^{-1}(L))_\de$.
\end{definition}

\medskip

We adopt the language from the study of the Kakeya conjecture to present our incidence results.
Such language can be found in, for example, \cite{KWZ}.

\begin{definition}[Shading]
\label{shading-def}
Let $L$ be a family of lines in $\ZR^n$ and let $\de\in(0,1)$.
A {\bf shading} $Y:L\to B^n(0,1)$ is an assignment such that $Y(\ell)\subset N_\de(\ell)\cap B^n(0,1)$ is a union of $\de$-balls in $\ZR^n$ for all $\ell\in L$.
We writes $(L,Y)_\de$ to emphasize the dependence on $\de$. 

Similarly, given a family of $\de$-tubes $\cT$, a {\bf shading} $Y:\cT\to\ZR^n$ is an assignment such that $Y(T)\subset T$ is a union of $\de$-balls in $\ZR^n$ for all $T\in \cT$.
\end{definition}

\begin{definition}[$E_{L,Y}$]
\label{E-L-def}
Let $\de\in(0,1)$ and let $(L,Y)_\delta$ be a set of lines and shading.
Define $E_{L,Y}:=\bigcup_{\ell\in L}Y(\ell)$, which can be identified as a union of $\de$-balls. If the shading $Y$ is apparent from the context, we will use $E_{L}$ to denote $E_{L,Y}$.
Moreover, for each $x\in E_{L,Y}$, define 
\begin{equation}
    L_Y(x)=\{\ell\in L:x\in Y(\ell)\}.
\end{equation}
Again, if the shading $Y$ is apparent from the context, we use $L(x)$ to denote $L_Y(x)$.
\end{definition}

\begin{definition}
\label{lambda-dense-def}
Let $\de\in(0,1)$ and let $(L,Y)_\delta$ be a set of lines and shading.
We say $Y$ is {\bf $\la$-dense}, if $|Y(\ell)|\geq \la |N_\de(\ell)|$.
\end{definition}

\begin{definition}
Let $\de\in(0,1)$ and let $(L,Y)_\delta$ be a family of lines and shading.
We say $(L',Y')_\de$ is a {\bf refinement} of $(L,Y)_\delta$, if $L'\subset L$, $Y'(\ell)\subset Y(\ell)$ for all $\ell\in L'$, and if
\begin{equation}
    \sum_{\ell'\in L'}|Y'(\ell')|\gtrapprox\sum_{\ell\in L}|Y(\ell)|.
\end{equation}
Similarly, for a family of $\de$-tubes and shadings $(\cT,Y)$, we say $(\cT',Y')$ a {\bf refinement} of $(\cT,Y)$, if $\cT'\subset \cT$, $Y'(T)\subset Y(T)$ for all $T\in \cT'$, and if $\sum_{T'\in \cT'}|Y'(T')|\gtrapprox\sum_{T\in \cT}|Y(T)|$.
\end{definition}

\smallskip

Although a shading $Y(\ell)$ is a union of $n$-dimensional balls, it is indeed a one-dimensional object.
Thus, aligning with \eqref{E-sub-rho}, we make the following definition.

\begin{definition}
\label{tube-segment}
Let $0<\delta<r < 1$. 
Let $\ell$ be a line, and let $Y(\ell)$ be a shading by $\de$-balls. 
We define $(Y(\ell))_r$ as follows:
Let $\cj(\ell)$ be a minimal covering of $Y(\ell)$ by $\de\times r$-tubes contained in $N_\de(\ell)$. 
Now define $(Y(\ell))_r=\cup_{\cj(\ell)}$. 
\end{definition}

\begin{definition}
\label{two-ends-reduction}
Let $v, C>0$, and let $\de\in(0,1)$.
Let $\ell$ be a line and $Y(\ell)$ be a uniform shading by $\de$-balls.
Define $\rho=\rho(\ell;v,C)\in[\de,1]$ as
\begin{equation}
\label{two-ends-reduction-1-section-1}
    \rho:=\min\{r\in[\de,1]:|Y(\ell)|_r< C^{-1}r^{-v}\}.
\end{equation}
Consequently, since $Y(\ell)$ is uniform, for all $r\in[\de,\rho]$ and all $J\subset (Y(\ell))_\rho$, 
\begin{equation}
\label{two-ends-reduction-2-section-1}
    |Y(\ell)\cap J|_r\gtrapprox C^{-1}(r/\rho)^v.
\end{equation}
\end{definition}

Definition \ref{two-ends-reduction} is the standard two-ends reduction on a shading $Y(\ell)$.
Note that, with $\rho=\rho(\ell;v,C)$, 
the $\rho^{-1}$-dilate of $Y(\ell)\cap J$ is a $(\de/\rho,v, CC')$-set for all $\de\times\rho$-tubes $J\subset (Y(\ell))_\rho$, for some $C'\lessapprox1$.

Next, we introduce a quantitative ``two-ends" condition on a shading $Y(\ell)$. 
As Definition \ref{two-ends-reduction} suggests, it is essentially the weakest possible spacing assumption that can be imposed on $Y(\ell)$.

\begin{definition}[Two-ends]
\label{two-ends-def}
Let $\de\in(0,1)$ and let $(L,Y)_\delta$ be a set of lines and shading.
Let $0<\e_2<\e_1<1$.
We say $Y$ is {\bf $(\e_1 ,\e_2, C)$-two-ends} if for all $\ell\in L$ and all $\de\times\de^{\e_1}$-tubes $J\subset N_\de(\ell)$, 
\begin{equation}
    |Y(\ell)\cap J|\leq C\de^{\e_2} |Y(\ell)|.
\end{equation}
When the constant $C$ is not important in the context, we say $Y$ is {\bf $(\e_1 ,\e_2)$-two-ends}, or simply {\bf two-ends}. 
A similar definition applies to a single shading $Y(\ell)$.
\end{definition}

\begin{lemma}
\label{delta-1-rho}
Let $\de\in(0,1)$, let $\ell$ be a line, and $Y(\ell)$ be a uniform shading by $\de$-balls.
Let $0<\e_2<\e_1<1$, and let $v<\e_2$.
Suppose $Y(\ell)$ is $(\e_1,\e_2, C)$-two-ends, and let $\rho=\rho(\ell;v,C)$ be the scale given by Definition \ref{two-ends-reduction}.
Then $\rho\geq \de^{\e_1}$.
\end{lemma}
\begin{proof}
By \eqref{two-ends-reduction-1-section-1}, $|Y(\ell)|_{\rho}< C^{-1}\rho^{-v}< C^{-1}\de^{-\e_2}$.
If $\rho < \delta^{\e_1}$, then since $Y(\ell)$ is $(\e_1,\e_2, C)$-two-ends, we have $|Y(\ell)|_{\rho}\geq |Y(\ell)|_{\de^{\e_1}}\geq C^{-1}\de^{-\e_2}$, a contradiction.
\end{proof}

\begin{lemma}
\label{two-ends-shading-lem}
Let $\de\in(0,1)$, let $\ell$ be a line, and $Y(\ell)$ be a shading by $\de$-balls.
If $Y(\ell)$ is $(\e_1,\e_2,C)$-two-ends and $Y'(\ell)$ is a refinement of $Y(\ell)$, then there exists $C'\lessapprox1$ such that $Y'(\ell)$ is $(\e_1,\e_2,CC')$-two-ends.
\end{lemma}
\begin{proof}
Just note that $|Y'(\ell)|\gtrapprox|Y(\ell)|$.
\end{proof}

\smallskip

\begin{definition}
\label{L[T]-def}
Let $v\in(0,1]$.
For a tube $T\in\ZR^n$ with cross-section radius $v$ and length $1$, let $\cl[T]$ be the family of lines interesting $T$ that make an angle $\leq v$ with the coreline of $T$.
Given a set of lines $L$, define
\begin{equation}
    L[T]:=L\cap \cl[T].
\end{equation}
\end{definition}

For each line $\ell$ in $\ZR^n$, denote by $V(\ell)$ the direction of $\ell$, which can be identified as a point on $\ZS^{n-1}$.
Given a set of lines $L$, denote by $V(L)=\{V(\ell), \ell\in L\}$, counted with multiplicity (that is, even when $V(\ell_1)=V(\ell_2)$, they are counted as different points in $V(L)$).

\begin{definition}
We say a set of $\de$-separated lines $L$ is {\bf $m$-parallel} if every $\de$-ball on $\ZS^{n-1}$ contains $\leq m$ points from the direction set $V(L)$.
\end{definition}

\smallskip

We end this section with two standard lemmas.

\begin{lemma}
\label{rich-point-refinement}
Let $\de\in(0,1)$ and let $(L,Y)_\delta$ be a set of lines in $\ZR^n$ and shading.
There exists a $\mu\geq1$, a set $E^\mu\subset E_L$, and a  refinement $(L',Y')_\de$ of $(L,Y)_\de$ so that 
\begin{enumerate}
    \item $Y'(\ell)$ is a refinement of $Y(\ell)$ for all $\ell\in L'$.
    \item $\#L_{Y'}(x)\sim\mu$ for all $x\in E_{L,Y'}$.
    \item $Y'(\ell)= E^\mu\cap N_\de(\ell)$ for all $\ell\in L'$.
    \item $\mu\approx |E_{L,Y'}|^{-1}\sum_{\ell\in L'}|Y(\ell')|$.
\end{enumerate}
\end{lemma}
\begin{proof}
By dyadic pigeonholing, there is a number $\mu\geq1$ and a set $E^\mu\subset E_L$ so that
\begin{enumerate}
    \item For any $x\in E^\mu$, $\#L(x)\sim \mu$.
    \item We have $\int_{E^\mu}\# L(x)\gtrapprox \int_{E_L}\# L(x)=\sum_{\ell\in L}Y(\ell)$.
\end{enumerate}
Let $Y'$ be a new shading defined as $Y'(\ell)=E^\mu\cap Y(\ell)$. Then $\#L_{Y'}(x)\sim\mu$ for any $x\in E_{L,Y'}$. 
By pigeonholing, we can find a refinement $L'$ of $L$ such that $Y'(\ell)$ is a refinement of $Y(\ell)$ for all $\ell\in L'$.
The objects $\mu$, $E^\mu$, and $(L',Y')_\de$ are what we need for the lemma.
\end{proof}

\begin{lemma}
\label{broad-narrow-lem}
Let $\de\in(0,1)$, and let $(L,Y)_\delta$ be a set of lines in $\ZR^n$ and shading. 
For all $x\in E_L$, there exists a $\rho=\rho(x)\in[10\de,1]$ so that the following is true.

\begin{enumerate}
    \item There exists a refinement of $L'(x)\subset L(x)$ such that $\ang(\ell,\ell')+\de\leq 2\rho$ for any $\ell,\ell'\in L'(x)$.
    \item There are two disjoint subsets $L_1,L_2\subset L'(x)$ of lines such that $\# L_1, \# L_2\gtrapprox \# L'(x)$, and $\rho\geq\ang(\ell_1,\ell_2)\gtrapprox \rho$ for all $\ell_1\in L_1,\ell_2\in L_2$.
\end{enumerate}
\end{lemma}

\begin{proof}
Let $A=10\cdot 2^n$. 
Let $\rho_{j}=|\log\de|^j$ with $\rho_n\approx \de$. 
Then $n\lesssim |\log\de|/\log|\log\de|$. 
Consider the following algorithm, starting at the scale $\rho_0$:

At scale $\rho_j$, there is a set of lines $L_j$ with $\# L_j\geq (10A)^{-1} \# L_{j-1}$ such that the directions of the lines in $L_j$ is contained in a $\rho_j$-cap $\theta_j\subset\ZS^{n-1}$. 
Partition $\theta_j$ into $\sim |\log\de|$ finite-overlapping $\rho_{j+1}$-caps $\{\theta_{j+1}\}$. 
For each $\theta_{j+1}$, let 
\begin{equation}
    L_{\theta_{j+1}}:=\{\ell\in L_j:\text{$V(\ell)$ is contained in }\theta_{j+1}\}.
\end{equation}

If there is a $\rho_{j+1}$-cap $\theta_{j+1}$ such that $\# L_{\theta_{j+1}}\geq  (10A)^{-1} \# L_{j}$, let $L_{j+1}=L_{\theta_{j+1}}$ and continue the algorithm to the next scale $\rho_{j+1}$. 
Otherwise, the algorithm stops. 
In this case, note that there are $\geq 9A$-caps $\theta_{j+1}$ such that $ \#L_{\theta_{j+1}}\gtrapprox \# L_j$. 
Thus, there are $\geq 9A$-caps $\theta_{j+1}$ such that $N_{\rho_{j+1}}(\theta_j)$ are disjoint and $ \#L_{\theta_{j+1}}\gtrapprox \# L_j$. 
 
Since $n$ is a finite number, the algorithm stops within finite steps. 
Suppose it stops at a scale $\rho_j$. 
Let $\rho=\rho_j$ and $L'(x)=L_j$ to conclude the lemma. \qedhere

\end{proof}

\bigskip


\section{A two-ends Furstenberg inequality in the plane} \label{section: two-ends Furstenberg2d}

\smallskip

The main goal of this section is to establish the two-ends Furstenberg inequality.
\begin{theorem}
\label{two-ends-furstenberg}
Let $\de\in(0,1)$.
Let $(L,Y)_\de$ be a set of $1$-parallel, $\de$-separated lines in $\ZR^2$ with an $(\e_1, \e_2)$-two-ends, $\lambda$-dense shading. Then for any $\e>0$,
\begin{equation}
\label{eq:thm2.1}
    |E_L|\geq c_\e\de^{\e}\de^{\e_1/2} \la^{1/2}\sum_{\ell\in L}|Y(\ell)|.
\end{equation}
\end{theorem}

\begin{remark}
\rm \label{rmk: two-ends-furstenberg}
If a two-ends assumption is added to \cite[Theorem 1.1]{Demeter-Wang} and \cite[Proposition 4.2]{Demeter-Wang}, the same proof might work for a larger range $s\in (0,1]$ instead of $s\in(0,1/2]$, which would give Theorem~\ref{two-ends-furstenberg} directly with $s=1$ by double counting.  
The assumption $s\leq 1/2$ was only used in \cite[Section 8, Step 3]{Demeter-Wang}, where the authors there showed $r\leq \delta^{-1/2}$. 
However, the condition $r\lessapprox \delta^{-1/2}$ can also be achieved for all $s\in (0,1]$ by utilizing the bush structure and the two-ends assumption. 
To ensure that the proof in \cite{Demeter-Wang} applies, one needs to verify that the two-ends condition is maintained after tube-thickening and pigeonholing.  
\end{remark}

\smallskip

The proof of Theorem \ref{two-ends-furstenberg} relies on several incidence estimates.

\begin{lemma}
\label{furstenberg-upper-range-pre}
Let $0<t<u\leq 1$ and let $\eta\in(0, (u-t)/2 )$.
Let $\cb$ be a collection of $\de$-balls in $\ZR^2$, and for each $B\in\cb$, let $\cT(B)$ be a family of $\de\times1$-tubes intersecting $B$.
Suppose $\cb$ is a $(\de,2-t,\de^{-\eta})$-set,
and suppose that for each $B\in\cb$, $\cT(B)$ is a $(\de,u,\de^{-\eta})$-set.
Let $r$ be such that $\#\cT(B)\geq r$ for all $B\in\cb$.
Then for any $\e>0$, 
\begin{equation}
    \#\bigcup_{B\in\cb}\cT(B)\geq c_\e\de^\e \de^{-1}r,
\end{equation}
where $ \#\bigcup_{B\in\cb}\cT(B)$ means the maximum number of distinct tubes contained in $\bigcup_{B\in\cb}\cT(B)$  (two $\delta\times 1$-tubes $T_1$ and $T_2$ are distinct if $|T_1\cap T_2|\leq |T_1|/2$). 
\end{lemma}

Lemma \ref{furstenberg-upper-range-pre} was stated in \cite[Theorem 4.1]{ren2023furstenberg} with an additional dependence between $\eta$ and $\e$.
However, this dependence can be removed since we are only considering the upper range $t+s>2$ in the Furstenberg set estimate. 
We refer to \cite[Proposition 4.2]{ren2023furstenberg} for its proof.

\begin{lemma}
\label{furstenberg-upper-range}
Let $0<t<u\leq 1$ and let $\eta\in(0,(u-t)/2)$. 
Let $\De\in[\de^{1-\eta},1]$, and let $\cb$ be a set of $\de$-balls in $\ZR^2$ contained in some   $\De$-ball $B_\De$.
For each $B\in\cb$, let $\cT(B)$ be a set of $\de\times1$-tubes intersecting $B$.
Suppose that the $\De^{-1}$-dilate of $\cb$ is a $((\de/\De),2-t,(\de/\De)^{-\eta})$-set,
and for each $B\in\cb$, $\cT(B)$ is a $(\de,u,(\de/\De)^{-\eta})$-set.
Let $r\geq1$ be such that $\#\cT(B)\geq r$ for all $B\in\cb$.
Then for any $\e>0$, 
\begin{equation}
    \#\bigcup_{B\in\cb}\cT(B)\geq c_\e\de^\e (\De/\de)r.
\end{equation}
\end{lemma}
\begin{proof}
For each $B\in\cb$, let $\cT_{\de/\De}(B)$ be the family of distinct $\de/\De\times1$-tubes containing at least one $\de\times1$-tube in $\cT(B)$.
For each $T_{\de/\De}\in\cT_{\de/\De}(B)$, define $\cT(T_{\de/\De}):=\{T\in\cT(B):T\subset T_{\de/\De}\}$.
By dyadic pigeonholing on $\{\#\cT(T_{\de/\De}):T_{\de/\De}\in\cT_{\de/\De}(B)\}$, there exist a set $\cT_{\de/\De}'(B)\subset\cT_{\de/\De}(B)$ and a number $\mu_{B}$ such that $\#\cT(T_{\de/\De})\sim\mu_B$ for all $T_{\de/\De}\in\cT_{\de/\De}'(B)$ and $\#\cT_{{\de/\De}}'(B)\cdot\mu_B\gtrapprox\#\cT(B)\geq r$.
Since $\cT(B)$ is a $(\de,u, (\de/\De)^{-\eta})$-set, $\cT_{\de/\De}(B)$ is a $(\de/\De, u, (\de/\De)^{-\eta})$-set, yielding that $\cT_{\de/\De}'(B)$ is a $(\de/\De,u,(\de/\De)^{-2\eta})$-set.

For each $\de/\De\times1$-tube $T_{\de/\De}\in\cT_{\de/\De}'(B)$, there is a unique $\de\times \De$-tube $\bar T$ intersecting $B$ that is parallel to $T_{\de/\De}$.
Let $\bar\cT'(B)$ be this set of $\de\times\De$-tubes,
so $\#\bar\cT'(B)=\#\cT_{{\de/\De}}'(B)\gtrapprox r/\mu_B$.
Note that $\bar\cT'(B)$ is indeed the $\De$-dilate of $\cT_{\de/\De}'(B)$. 
In particular, the $\De^{-1}$-dilate of $\bar\cT'(B)$ is a $(\de/\De,u,(\de/\De)^{-2\eta})$-set.
By dyadic pigeonholing on $\{\mu_B:B\in\cb\}$, there exists a refinement $\cb'$ of $\cb$ and a uniform $\mu$ such that $\mu_B\sim\mu$ for all $B\in\cb'$.
Hence, the $\De^{-1}$-dilate of $\cb'$ is a $((\de/\De),2-t,(\de/\De)^{-2\eta})$-set.
Moreover, $\#\cT(T_{\de/\De})\sim\mu$ for all $T_{\de/\De}\in\cT_{\de/\De}'(B)$ and all $B\in\cb'$. 
This shows that there are $\gtrsim \mu$ many $\de\times 1$-tubes in $\cT(B)$ intersecting $\bar T$ for all $\bar T\in\bar\cT'(B)$ and all $B\in\cb'$.
Therefore, $\#\bigcup_{B\in\cb}\cT(B)\gtrsim\mu \#\bigcup_{B\in\cb'}\bar\cT'(B)$.

Apply a $\De$-dilate of Lemma \ref{furstenberg-upper-range-pre} to the $\de\times\De$-tubes $\{\bar \cT'_B,B\in\cb'\}$ so that
\begin{equation}
    \#\bigcup_{B\in\cb}\cT(B)\gtrsim\mu \#\bigcup_{B\in\cb'}\bar\cT'(B)\gtrapprox \mu(\de/\De) (r/\mu)\gtrapprox (\de/\De)r. \qedhere
\end{equation}

\end{proof}

\smallskip

The next theorem was proved in \cite{Demeter-Wang}. We state a more quantitative version of it here.

\begin{theorem}[\cite{Demeter-Wang}, Theorem 5.4]
\label{DW-24}
For any $\nu>0$, there exists an $\eta>0$, which is much smaller than $\nu^{2\nu^{-1}}$, such that the following is true for sufficiently small $\de$: 

Let $\cT$ be a family of $\de\times 1$-tubes in the plane, and let $\cp$ be a family of uniform $\de$-balls such that for all $\rho\in[\de,1]$,
\begin{equation}
    |\cup_\cp|_\rho\gtrsim\rho^{-s}\de^{\eta}.
\end{equation}
For each $\de$-ball $P\in\cp$, let $\cT_p\subset\cT$ be a $(\de,s,\de^{-\eta})$-set 
(here we identify each $T\in\cT_p$ as a $\de$-arc on $\ZS^1$) of $\de\times1$-tubes passing through $p$ with $\#\cT_p\sim r$. 
Moreover, as a union of $\de$-arcs of $\ZS^1$, $\cup_{\cT_p}$ is uniform, and there is a uniform branching function of the family of unions of $\de$-arcs $\{\cup_{\cT_p}\}_{p\in\cp}$.

Then one of the following must be true:
\begin{enumerate}
    \item We have
    \begin{equation}
        \#\cp \lesssim\de^{s-\nu}\frac{(\#\cT)^2}{r^2}.
    \end{equation}
    \item There exists a scale $\De\gtrsim \de^{1-\sqrt{\eta}}$ such that for each $\De$-ball $B\subset (\cup_\cp)_\De$, the $\De^{-1}$-dilate of $\cup_\cp\cap B$ is a uniform, $(\De/\de, 2-s+\eta^{1/4}, \de^{-\eta})$-set.
\end{enumerate}

\end{theorem}

\begin{remark}
\rm
 We have two remarks on the statement of Theorem \ref{DW-24}:

(1) In \cite[Theorem 5.4]{Demeter-Wang}, the set of $\de$-balls $\cp$ is required to be ``$\epsilon$-uniform", which is stronger than our assumption here that $\cp$ is uniform.
However, this difference only affects the multi-scale decomposition (Proposition~\ref{multiscale-prop}) of $\cp$.
With the current weaker assumption, we can get the same multi-scale decomposition (at a cost of an acceptable loss on some factors).
Thus, the conclusions of \cite[Theorem 5.4]{Demeter-Wang} are still true under the assumption of Theorem \ref{DW-24}.

\smallskip

(2) The proof of Theorem~\ref{DW-24} (\cite[Theorem 5.4]{Demeter-Wang})  uses multi-scale decomposition of $\cp$, followed by  the Furstenberg set estimate  \cite{ren2023furstenberg} on each scale. 
Item (2) of Theorem~\ref{DW-24} happens when $\cp$ has dimension $>2-s$ in smaller scales. 
In  \cite[Theorem 5.4]{Demeter-Wang}, item (2) is stated in a weaker form. 
However, their proof indeed gives the current stronger version of item (2) stated here. 

\end{remark}

We will use Theorem \ref{DW-24} to prove the following proposition, which is a combination of Lemma \ref{furstenberg-upper-range} and the dual version of Theorem \ref{DW-24}.

\begin{proposition}
\label{prop-induction}
For any $\nu>0$, there exists an $\eta>0$, which is much smaller than $\nu^{2\nu^{-1}}$, such that the following is true for sufficiently small $\de\in(0,1)$:

Let $(L,Y)_\de$ be a set of uniform, $\de$-separated lines and shading such that:
\begin{enumerate}
    \item For each $\ell\in L$, the shading $Y(\ell)$ is uniform, and there is a uniform branching function for the family of shadings $\{Y(\ell):\ell\in L\}$.
    \item For all $\ell\in L$, $|Y(\ell)|\lesssim\de^{1-s-\eta}|N_\de(\ell)|$, and $Y(\ell)$ is a $(\de,s, \de^{-\eta})$-set, which yields $|Y(\ell)|\gtrsim\de^{1-s+\eta}|N_\de(\ell)|$.
    \item For all $\rho\in[\de,1]$, $\#\{T\in\cT_\rho:L[T]\not=\varnothing\}\gtrsim \rho^{-s}\de^\eta$ (recall Definition \ref{L[T]-def}), where $\cT_\rho$ is a maximal set of distinct $\rho\times1$-tubes in $B^2(0,1)$.
    \item $\#L(x)\lessapprox |E_L|^{-1}\sum_{\ell\in L}|Y(\ell)|$ for all $x\in E_L$.
\end{enumerate}
Then one of the following must be true:
\begin{enumerate}
    \item $\# L(x)\lessapprox\de^{(s-1)/2}$ for all $x\in E_L$.
    \item We have
    \begin{equation}
        |E_L|\gtrapprox \de^{\nu+3\eta}\de^{3(1-s)/2}.
    \end{equation}
    \item There exists $\De\gtrsim \de^{1-\sqrt{\eta}}$ that $|E_L\cap N_{\De}(\ell)|\gtrsim \de^{1-s+2\eta}|N_{\De}(\ell)|$ for all $\ell\in L$.
\end{enumerate}

\end{proposition}

\begin{proof}

By the point-line duality, we can identify $L$ as a set of $\de$-balls $\cp\subset[0,1]^2$ with the following correspondences:
\begin{enumerate}
    \item Each $\ell\in L$ corresponds to $\de$-ball $p\in \cp$, and the shading $Y(\ell)$ corresponds to a family of $\de\times1$-tubes $\cT_p$ passing through the $\de$-ball $p\in \cp$.
    \item The union of $\de$-balls $E_{L}$ corresponds to the union of $\de$-tubes $\cT=\bigcup_{p\in \cp} \cT_p$.
\end{enumerate}  
It is straightforward to check that the configuration $(\cp,\bigcup_{p\in\cp}\cT_p)$ obeys the hypothesis of Theorem \ref{DW-24} with $r\gtrsim\de^{-s+\eta}$. 
Thus, one of the following must be true:
\begin{enumerate}
    \item[{\bf A.}] We have
    \begin{equation}
        \#\cp \lesssim\de^{s-\nu}\frac{(\#\cT)^2}{r^2}.
    \end{equation}
    \item[{\bf B.}] There exists a scale $\De\gtrsim \de^{1-\sqrt{\eta}}$ such that for each $\De$-ball $B\subset (\cup_\cp)_\De$, the $\De^{-1}$-dilate of $\cup_\cp\cap B$ is a uniform, $(\De/\de, 2-s+\eta^{1/4}, \de^{-\eta})$-set.
\end{enumerate}

\smallskip

Suppose Case {\bf A} happens. 
By reversing the point-line duality, we get 
\begin{equation}
\label{DW-case-1}
    |E_L|\gtrsim \de^{\nu/2+\eta}\de^{2-3s/2}(\# L)^{1/2}.
\end{equation}
Recall that $\#L(x)\lessapprox |E_L|^{-1}\sum_{\ell\in L}|Y(\ell)|$ for all $x\in E_L$ and $|Y(\ell)|\lesssim\de^{1-s-\eta}|N_\de(\ell)|$ for all $\ell\in L$.
Therefore, $\#L(x)\lessapprox |E_L|^{-1}\de^{2-s-\eta}(\# L)\lessapprox\de^{-\nu/2-2\eta}\de^{s/2}(\# L)^{1/2}$. 

If $\# L\leq \de^{-1+4\eta+\nu}$, then $\# L(x)\lessapprox \de^{(s-1)/2}$ for all $x\in E_L$, giving item (1) in Proposition \ref{prop-induction}. 
If $\# L\geq \de^{-1+4\eta+\nu}$, then $|E_L|\gtrsim \de^{\nu+3\eta}\de^{3(1-s)/2}$, giving item (2) in Proposition \ref{prop-induction}.

\smallskip

Suppose Case {\bf B} happens. Consider each $\De$-ball $B\subset(\cup_\cp)_\De$ (i.e. $B$ is the $\Delta$-ball contained in $\Delta$-neighborhood of the union of $\delta$-balls in $\cp$) and let $\cp_B$ be the set of $\delta$-balls in $\cp$ contained in $B$. 
For each $p\in \cp_B$, consider the family of tubes $\cT_p$, which is a $(\de,s, \de^{-\eta})$-set.
Since $\De\gtrsim\de^{1-\sqrt{\eta}}$, 
$\cT_p$ is a $(\de,s, (\de/\De)^{-\sqrt{\eta}})$-set for all $p\in\cp_B$.
Similarly, since the $\De^{-1}$-dilate of the $\cp_B$ is a $(\de/\De, 2-s+\eta^{1/4}, \de^{-\eta})$-set, it is also a $(\de/\De, 2-s+\eta^{1/4}, (\de/\De)^{-\sqrt{\eta}})$-set.
Applying Lemma \ref{furstenberg-upper-range} to the configuration $(\cp_B, \{\cT_p\}_{p\in \cp_B})$ with $(u,t,\eta,r)=(s, s-\eta^{1/4}, \eta^2,\de^{-s+\eta})$, we have $\#\bigcup_{p\in \cp_B}\cT_p\gtrapprox (\De/\de)\de^{-s+\eta}$.
That is, the maximal number of distinct $\de\times1$-tubes in $\bigcup_{p\in \cp_B}\cT_p$ is $\gtrapprox (\De/\de)\de^{-s+\eta}$.

By reversing the point-line duality, each $\De$-ball $B$ corresponds to a $\De\times1$-tube $T_{\De}$, and the $\de\times1$-tubes $\bigcup_{p\in \cp_B}\cT_p$ corresponds to the $\de$-balls $\bigcup_{\ell\subset T_{\De}}Y(\ell)$.
Thus, what we had above can be stated as follows:
There is a partition of $L$ into $\{L[T_{\De}], T_{\De}\in\cT_{\De}\}$, where $\cT_{\De}$ is a family of distinct $\De\times1$-tubes and $L[T_{\De}]=\{\ell\in L: \ell\subset T_\De\}$, so that 
\begin{equation}
\label{big-at-Delta}
    \big|\bigcup_{\ell\in L[T_{\De}]}Y(\ell)\big|\sim\de^2\#\bigcup_{p\in \cp_B}\cT_p\gtrapprox(\De/\de)\de^{2-s+\eta}. 
\end{equation}
This gives item (3) of Proposition \ref{prop-induction}. \qedhere



\end{proof}

\begin{corollary}
\label{cor-induction}
For any $\nu>0$, there exists an $\eta>0$, which is much smaller than $\nu^{2\nu^{-1}}$, such that the following is true for sufficiently small $\de\in(0,1)$:

Let $(L,Y)_\de$ be a set of uniform, $\de$-separated lines obeying the assumptions (1)-(3) in Proposition \ref{prop-induction}.
If there is an additional assumption that  $T_\rho$ contains $\leq (\rho/\de)^{2-s+\eta^{1/4}/2}$ many $\de$-tubes in $\cT$ for all $\rho$-tube $T_\rho$ and all $\rho\in[\de,1]$, then
\begin{equation}
\label{additional-case}
    |E_L|\gtrsim \de^{\nu/2+\eta}\de^{2-3s/2}(\# L)^{1/2}.
\end{equation}
\end{corollary}
\begin{proof}

Proceed as we did in the proof of Proposition \ref{prop-induction}, and note that Case {\bf B} never happens by our assumption on $(L,Y)_\de$.
Thus \eqref{additional-case} follows from \eqref{DW-case-1} (note that we do not need  Assumption (4) of Proposition \ref{prop-induction} to obtain \eqref{DW-case-1}).
\end{proof}

\smallskip

\begin{definition}[$(\de,s,C; \De)$-set]
\label{delta-s-Delta}
Let $\de\in(0,1)$ and let $\De\in[\de,1]$.
We say a union of $\de$-balls $E$ is a {\bf $(\de,s,C; \De)$-set}, if 
\begin{equation}
    |E\cap B(x,r)|_\de\leq Cr^s|E|_\de, \hspace{.3cm}\forall x \in\ZR^n, \,r\in[\De,1].
\end{equation}
\end{definition}

Recall Definition \ref{delta-s}.
Note that a $(\de,s,C)$-set is a $(\de,s,C;\De)$-set for all $\De\in[\de,1]$.
When $\De=\de$, a $(\de,s,C;\De)$-set is just a $(\de,s,C)$-set.
Also, at the coarser resolution $\De$, a $(\de,s,C;\De)$-set is a $(\De,s,C)$-set.
Moreover, similar to Lemma \ref{refinement-delta-s-set}, we have
\begin{lemma}
\label{refinement-delta-s-Delta-set}
Let $0<\de\leq\De\leq1$ and let $C_1, C_2\geq1$.
If a union of $\de$-balls $E$ is a $(\de,s, C_1;\De)$-set and $E'$ is a $\gtrsim C_2^{-1}$-refinement of $E$, then $E'$ is a $(\de,s, C_1C_2;\De)$-set. 
\end{lemma}

\smallskip

By the two-ends reduction in Definition \ref{two-ends-reduction}, Theorem \ref{two-ends-furstenberg} reduces to the following proposition (note that if a set of lines $L$ is directional $\de$-separated, then it automatically is a Katz-Tao $(\de,1)$-set).

\begin{proposition}
\label{two-ends-furstenberg-prop} 
Let $\e>0$ be given.
Let $\eta>0$ be such that Proposition \ref{prop-induction} and Corollary \ref{cor-induction} are true with $(\nu,\eta)=(\e^2,2\eta)$.
Then there exists $c_\e>0$ such that the followings is true for all $\delta \in (0, 1)$:

Let $(L, Y)_{\delta}$ be a set of $\delta$-separated lines in $\mathbb{R}^2$ with a uniform, $\lambda$-dense shading.
Suppose $L$ is a Katz-Tao $(\de,1)$-set, and suppose $Y(\ell)$ is a $(\de,\e^2,C;\rho)$-set for some $\rho \in [\delta, \delta^{\eta}]$ and for all $\ell\in L$.
Then
\begin{equation}
\label{main-prop-esti}
    |E_L|\geq c_{\e} \delta^{\e} C^{-\eta^{-2}} \lambda^{1/2} \sum_{\ell\in L} |Y(\ell)|. 
\end{equation}
\end{proposition}
\begin{remark}
\rm

Theorem \ref{two-ends-furstenberg} is also true under the slightly weaker assumption that $L$ is a Katz-Tao $(\de,1)$-set.
It was stated with a stronger assumption to align with the setup in the restriction conjecture.
\end{remark}

\smallskip

\begin{proof}[Proof of Theorem \ref{two-ends-furstenberg} via Proposition \ref{two-ends-furstenberg-prop}]

It suffices to prove Theorem \ref{two-ends-furstenberg} for all $\e\in(0, \sqrt{\e_2})$.
By dyadic pigeonholing, there exists a refinement $L_1$ of $L$ such that for all $\ell\in L_1$, $|Y(\ell)|$ are the same up to a constant multiple. 
Let $Y_1=Y$.
Assume $|Y_1(\ell)|\sim\la$ for all $\ell\in L_1$, without loss of generality.

\smallskip

Apply Lemma \ref{uniformization} and Lemma \ref{uniform-sets-branching-lem} to the set of shadings $\{Y_1(\ell):\ell\in L_1\}$, we know that there is a refinement $(L_2,Y_2)_\de$ of $(L_1,Y_1)_\de$ such that 
\begin{enumerate}
    \item For all $\ell\in L_2$, $Y_2(\ell)$ is uniform, and $Y_2(\ell)$ is a refinement of $Y_1(\ell)$. 
    \item There is a uniform branching function $\be_{L_2}$ for the shadings $\{Y_2(\ell):\ell\in L_2\}$. 
\end{enumerate}
By Lemma \ref{two-ends-shading-lem}, there exists $K\lessapprox1$ such that $Y_2$ is $(\e_1, \e_2,K)$-two-ends.
For each $\ell\in L_2$, by Definition \ref{two-ends-reduction}, there exists a scale $\rho(\ell)=\rho(\ell;\e^2,K)\in[0,1]$ such that 
\begin{enumerate}
    \item $|Y_2(\ell)|_{\rho(\ell)}< K^{-1}\rho(\ell)^{-\e^2}$.
    \item For all $r\in[\de,\rho(\ell)]$ and all $J\subset (Y_2(\ell))_{\rho(\ell)}$, $|Y_2(\ell)\cap J|_r\gtrapprox K^{-1}(r/\rho)^{\e^2}$.
\end{enumerate}
Since $\e^2< \e_2$, by Lemma \ref{delta-1-rho}, $\rho(\ell)\geq\de^{\e_1}$ for all $\ell\in L_2$.
Apply dyadic pigeonholing to the set $\{\rho(\ell):\ell\in L_2\}$, we can find a uniform $\rho$ and a refinement $L_3$ of $L_2$ such that $\rho(\ell)=\rho$ for all $\ell\in L_3$.
Let $Y_3=Y_2$.
Then the set of lines and shading $(L_3,Y_3)_\de$ is a refinement of $(L,Y)_\de$ with a uniform, $\la$-dense shading.
Moreover, for all $\ell\in L_3$, the following is true:
\begin{enumerate}
    \item $|Y_3(\ell)|_\rho< K^{-1}\rho^{-\e^2} \lessapprox \rho^{-\e^2}$.
    \item For all $r\in[\de,\rho]$ and all $J\subset (Y_3(\ell))_{\rho}$, $|Y_3(\ell)\cap J|_r\gtrapprox K^{-1}(r/\rho)^{\e^2}$.
    \item $\rho\geq\de^{\e_1}$.
\end{enumerate}

\smallskip

Let $\cb_\rho$ be a family of finite-overlapping $\rho$-balls that covers $E_{L_3}$.
For each $B\in\cb_\rho$, let $\cT_B^0$ be the family of distinct $\de\times\rho$-tubes in $B$. 
Define $\cT_B$ as 
\begin{equation}
    \cT_B=\{J\in\cT_B^0:J\subset (Y_3(\ell))_\rho \text{ for some $\ell\in L_3$} \}.
\end{equation}
For each $J\in\cT_B$, let $L_3(J)=\{\ell\in L_3:J\subset(Y_3(\ell))_\rho \}$.
Pick one $\ell\in L_3(J)$ to define a shading $Y_B$ on $J$ as
\begin{equation}
    Y_B(J)=Y_3(\ell)\cap J.
\end{equation}
Since $Y_3$ is uniform, $Y_B(J)$ is uniform, and $|Y_B(J)|$ are the same up to a constant multiple for all $J\in\cT_B$.
Since $Y_3$ is $\gtrapprox\la$-dense and since $|Y_3(\ell)|_\rho  \lessapprox \rho^{-\e^2} $,  $Y_B(J)$ is $\gtrapprox\la  \rho^{-1 +\e^2}$-dense.
Recall that for all $r\in[\de,\rho]$ and all $T\subset (Y_3(\ell))_{\rho}$, $|Y_3(\ell)\cap T|_r\gtrapprox K^{-1}(r/\rho)^{\e^2}$.
Thus, the $\rho^{-1}$-dilate of $Y_B(J)$ is a $(\de/\rho,\e^2, K')$-set for some $K'\lessapprox1$.

\smallskip

We want to apply Proposition \ref{two-ends-furstenberg-prop} to the $\rho^{-1}$-dilate of $(\cT_B,Y_B)$. 
However,  the $\rho^{-1}$-dilate of $\cT_B$ may not be a Katz-Tao $(\de/\rho,1)$-set of tubes.
To overcome this issue, we will randomly pick a $\cT_B''\subset\cT_B$ that obeys that desired assumption.
After that, we will apply Proposition \ref{two-ends-furstenberg-prop} to $(\cT_B'',Y_B)$. Let's turn to details. 

\smallskip

Let $L_3(B)=\{\ell\in L_3: Y_3(\ell)\cap B\not=\varnothing\}$, so $L_3(B)=\bigcup_{J\in\cT_B}L_3(J)$.
Let $\mu_J=\#L_3(J)$.
By dyadic pigeonholing on $\{\mu_J:J\in\cT_B\}$, there is a subset $\cT_B'\subset\cT_B$ and a uniform number $\mu_B$ such that the following is true:
\begin{enumerate}
    \item $\mu_J\sim\mu_B$ for all $J\in\cT_B'$.
    \item $\mu_B\cdot \#\cT_B'\gtrapprox\sum_{J\in\cT_B}\mu_J\#L_3(J)\sim \#L_3(B)$.
\end{enumerate}

For each $v\in[\de,\rho]$, let $\cT_v(B)$ be  the the family of distinct $v\times\rho$-tubes in $B$ that containing at least one $J\in\cT_B'$.
Define a multiplicity factor $\si$ as follows:
\begin{equation}
    \si:=\sup_{v\in[\de,\rho]} \{ (\de/v)\cdot\sup_{\tilde{J}\in\cT_v(B)}\#\{J\subset \tilde{J}: J\in\cT_B\} \} \geq1.
\end{equation}
Note that for each $v\in[\de,\rho]$ and each $\tilde{J}\in\cT_v(B)$, the set of lines $L(\tilde{J})=\{\ell\in L:|N_\delta(\ell)\cap J| \gtrsim |J|\text{ for some }J\subset \tilde{J},\,J\in\cT_B'\}$ is contained in a $  (v/\rho)\times 1$-tube.
Since $L$ is a Katz-Tao $(\de,1)$-set, $\# L(\tilde{J})\lesssim (v/\de \rho)$.
Recall that there are $\gtrsim\#L_3(J)\sim\mu_B$ many $\ell\in L$ such that $|N_\delta(\ell)\cap J| \gtrsim |J|$ for each $J\in\cT_B'$.
Thus, for all $v\in[\de,\rho]$ and all $\tilde{J}\in\cT_v(B)$, $\mu_B^{-1}(v/\de \rho)\gtrsim\mu_B^{-1}\#L(\tilde{J})\gtrsim\#\{J\subset \tilde{J}: J\in\cT_B\}$. 
This implies
\begin{equation}
\label{si-mu-B}
    \si \cdot \mu_B\lesssim \rho^{-1}.  
\end{equation}
Let $\cT_B''\subset\cT_B'$ be a uniform random sample of probability $\si^{-1}$.
Thus, with high probability, $\cT_B''\gtrapprox \si^{-1}\cT_B'$.
Moreover, since $\sup_{\tilde{J}\in\cT_v(B)}\#\{J\subset \tilde{J}: J\in\cT_B'\}\lesssim(v/\de)^2$  for all $v\in[\de,\rho]$, with high probability, $\tilde{J}$ contains $\lessapprox(v/\de)$ many $\de\times \rho$-tubes in $\cT_B''$ for all $\tilde{J}\in\cT_v(B)$ and all $v\in[\de,\rho]$.
This shows that the $\rho^{-1}$-dilate of $\cT_B''$ is a Katz-Tao $(\de/\rho,1)$-set.

\smallskip

Recall that $Y_B$ is a uniform, $\gtrapprox\la\rho^{-1+\e^2}$-dense shading, and the $\rho^{-1}$-dilate of $Y_B(J)$ is a $(\de/\rho,\e^2, K')$-set for all $J\in\cT_B$ and some $K'\lessapprox1$.
Apply the $\rho$-dilate version of Proposition \ref{two-ends-furstenberg-prop} to $(\cT_B'',Y_B)$ (with $\e/2$ in place of $\e$) so that
\begin{equation}
\label{scale-rho-esti}
    |E_{\cT_B'',Y_B}|\gtrapprox c_{\e}(\delta/\rho)^{\e/2}\cdot (\lambda \rho^{-1+\e^2})^{1/2}\sum_{J\in\cT_B''}|Y_B(J)|. 
\end{equation}
Recall that $\cT_B''\gtrapprox\si^{-1}\cT_B'$, $\#L_3(J)\sim\mu_B$ for all $J\in\cT_B'$, and $\mu_B\cdot \#\cT_B'\gtrapprox \#L_3(B)$.
Since $|Y_B(J)|$ are the same up to a constant multiple, by \eqref{si-mu-B},
\begin{align}
    \sum_{J\in\cT_B''}|Y_B(J)|&\gtrapprox\si^{-1}\sum_{J\in\cT_B'}|Y_B(J)|\gtrsim(\si\mu_B)^{-1}\sum_{\ell\in L_3(B)}|Y_3(\ell)\cap B|\\
    &\gtrapprox\rho\int_B\#L_3(x).
\end{align}
Since $\rho\geq\de^{\e_1}$ and since $E_{\cT_B'',Y_B}\subset E_L\cap B$, plug this back to \eqref{scale-rho-esti} so that
\begin{align}
\label{ELcapB-two-ends-reduction}
    |E_L\cap B|&\gtrapprox  c_{\e}(\delta/\rho)^{\e/2} \cdot (\lambda \rho^{-1 +\e^2})^{1/2}\cdot\rho\int_B\#L_3(x)\\
    &\geq c_{\e}(\delta/\rho)^{\e/2}   \cdot    \rho^{\e^2/2}\de^{\e_1/2}\la^{1/2}\int_B\#L_3(x).
\end{align}
Since $\cb_\rho$ covers $E_{L_3}$ and since $(L_3,Y_3)_\de$ is a refinement of $(L,Y)_\de$, we sum up all $B\in\cb_\rho$ in \eqref{ELcapB-two-ends-reduction} to get
\begin{equation}
    |E_L|\gtrapprox  c_{\e}(\delta/\rho)^{\e/2}   \cdot    \rho^{\e^2/2} \de^{\e_1/2}\la^{1/2}\int\#L_3(x)\geq c_\e\de^{\e}\de^{\e_1/2}\la^{1/2}\sum_{\ell\in L}|Y(\ell)|. \qedhere
\end{equation}

\end{proof}

\smallskip

Finally, let us see how Proposition \ref{prop-induction} and Corollary \ref{cor-induction} imply Proposition \ref{two-ends-furstenberg-prop}. 

\bigskip

\begin{proof}[Proof of Proposition \ref{two-ends-furstenberg-prop}]

We will prove \eqref{main-prop-esti} by a backward induction on $\de$.
In the base case when  $\delta\gtrsim_{\e} 1$, we choose $c_{\e}$ sufficiently small such that \eqref{main-prop-esti} is true.
By dyadic pigeonholing, there exists a refinement $L_1$ of $L$ such that for all $\ell\in L_1$, $|Y(\ell)|$ are the same up to a constant multiple. 
Without loss of generality, assume
\begin{equation}
    |Y(\ell)|\sim\la
\end{equation}
for all $\ell\in L_1$.
Denote by $Y_1=Y$, so that $(L_1,Y_1)_\de$ is a refinement of $(L,Y)_\de$ and a family of $\de$-separated lines with a uniform,  $\lambda$-dense shading.
Moreover, there exists a $\rho\in[\de,\de^{\eta}]$ such that $Y_1(\ell)$ is a $(\de,\e^2,C;\rho)$-set for all $\ell\in L_1$.
Note that when $C\geq \delta^{-2\eta^2}$, $|E_L|\geq\sup_{\ell\in L}|Y(\ell)|\geq\de^2\geq C^{-\eta^{-2}}$, which implies \eqref{main-prop-esti} directly.
Thus, we can assume $C\leq\de^{-2\eta^2}$, in which case the spacing condition on $Y_1$ is not vacuous, as $\rho\leq\de^{\eta}$.

\medskip

\noindent {\bf Step 1: Finding the multi-scale decomposition and the correct scale.}

Applying Lemma \ref{uniform-sets-branching-lem} to the set of shadings $\{Y_1(\ell): \ell\in L_1\}$, there is a refinement $(L_2,Y_2)_\de$ of  $(L_1,Y_1)_\de$ such that there is a uniform branching function $\be_{L_2}$ for the family of shadings $\{Y_2(\ell): \ell\in L_2\}$. 
Apply Proposition \ref{multiscale-prop} to  $\{Y_2(\ell):\ell\in L_2\}$ with 
\begin{equation}
\label{epsilon}
    \eta_0=\eta_0(\eps\e^3)
\end{equation}
to have the following multi-scale decomposition:  There exists a partition 
\begin{equation}
    0=A_1<A_2<\cdots<A_{H+1}=1
\end{equation}
and a sequence 
\begin{equation}
    0\leq s_1< s_2<\cdots< s_{H}\leq 1
\end{equation}
such that for all $1\leq h\leq H$ and all $\ell\in L_2$, the following is true:
\begin{enumerate}
    \item $A_{h+1}-A_h\geq \eta_0\eta^{-1}\e^{-3}$.
    \item $\log_{1/\de}\big(\frac{|Y_2(\ell)|_{\de^{A_{h+1}}}}{|Y_2(\ell)|_{\de^{A_h}}}\big)\leq (s_h+4\eps\e^3)(A_{h+1}-A_h)$. 
    \item  For each $\de\times\de^{A_{h}}$-tube $J\subset (Y_2(\ell))_{\de^{A_{h}}}$ defined in Definition \ref{tube-segment}, the $\de^{-A_h}$-dilate of $(Y_2(\ell))_{\de^{A_{h+1}}}\cap J$ along $\ell$ is a $(\de^{A_{h+1}-A_h}, s_h, \de^{-4\eps\e^3(A_{h+1}-A_h)})$-set.
    \item $s_H\geq \log_{1/\de}|Y_2(\ell)|_\de-\eps\e^3$.
\end{enumerate}

We are only going to use the information in the range $[\de,\de^{A_H}]$.
Since $Y_2(\ell)$ is a refinement of $Y(\ell)$, by Lemma \ref{refinement-delta-s-Delta-set}, $Y_2(\ell)$ is a $(\de,\e^2,CC_2;\rho)$-set for some $C_2\lessapprox1$.
Since $\rho\leq\de^{\eta}$, this gives $|Y_2(\ell)|_\de\gtrapprox\de^{-\eta\e^2}\geq \de^{-2\eps\e^3}$, which implies $s_H>0$.

\smallskip

To ease notations, define
\begin{equation}
\label{parameters}
    r:=\de^{A_H}, \,\,\,\,\,s:=s_H,\,\,\,\,\,\eta_1:=\eps\e^3.
\end{equation}

\medskip

\noindent {\bf Step 2: Uniformization of $E_L$ inside each $r$-ball.}

We will establish a uniformization on $E_{L}\cap B$ for each $r$-ball $B\subset(E_L)_r$. 
Let $\cT_{\de/r}$ be a maximal collection of distinct $(\de/r)\times1$-tubes. 
For each $T\in\cT_{\de/r}$, recall Definition \ref{L[T]-def} for $L_2[T]$. 
Define
\begin{equation}
\label{tilde-lambda}
    \tilde\la:=\la (r/\de)^{1-s-\eta_1}.
\end{equation}

For each $\ell\in L_2$, define a new shading  $\tilde Y_2(\ell) = (Y_2(\ell))_r$ by $\delta\times r$-tubes as in Definition \ref{tube-segment}  and let $\tilde L_2[T]=L_2[T]$. 
For each $\de\times r$-tube $J\subset \tilde Y_2(\ell)$, note that $|Y_2(\ell)\cap J|\lesssim (\de/r)^{1-s-\eta_1}|N_\de(\ell) \cap J|$.  
Consequently, $\tilde Y_2(\ell)$ is $\gtrapprox \tilde\la$-dense, as $Y_2(\ell)$ is $\gtrapprox \lambda$-dense. 
Apply Lemma \ref{rich-point-refinement} to $(\tilde L_2[T], \tilde Y_2)_\de$ to obtain a refinement $(\tilde L_2'[T], \tilde Y_2')_\de$ of $(\tilde L_2[T], \tilde Y_2)_\de$ satisfying 
\begin{equation}
    \#(\tilde L_2[T])_{\tilde Y_2'}(x)\sim\frac{\sum_{\ell\in \tilde L_2[T]}|\tilde Y_2(\ell)|}{|E_{\tilde L_2[T], \tilde Y_2'}|}\approx\frac{\sum_{\ell\in \tilde L_2'[T]}|\tilde Y_2'(\ell)|}{|E_{\tilde L_2'[T]}|}.
\end{equation}
Note that $\tilde Y_2'(\ell)$ is $\gtrapprox\tilde\la$-dense and is still a union of $\delta\times r$-tubes.

Recall that $Y_2(\ell)$ is a $(\de,\e^2,CC_2;\rho)$-set for some $C_2\lessapprox1$.
Since $Y_2$ is uniform, $\tilde Y_2(\ell)$ is a $(r,\e^2,C\tilde C_2;\max\{\rho,r\})$-set for some $\tilde C_2\lessapprox1$.
Thus, the non-isotropic $r^{-1}$-dilate of $(\tilde L_2'[T], \tilde Y_2')_\de$ (which becomes a set of $r\times1$-tubes with a shading by $r$-balls) 
is a $(r,\e^2,C\tilde C_2;\max\{\rho,r\})$-set.
Clearly, $\max\{\rho,r\}\in[r,r^\eta]$, since $\rho\in[\de,\de^\eta]$.
Apply Proposition \ref{two-ends-furstenberg-prop} at scale $r$ to this non-isotropic $r^{-1}$-dilate to bound $E_{\tilde L_2'[T]}$ so that
\begin{align}
\label{mu-T-step2}
    \mu_T:=\frac{\sum_{\ell\in \tilde L_2'[T]}|\tilde Y_2'(\ell)|}{|E_{\tilde L_2'[T]}|}\lessapprox c_\e^{-1} r^{-\e}(C\tilde C_2)^{\eta^2} \tilde\la^{-1/2}.
\end{align}
For any two $\ell_1,\ell_2\in \tilde L_2[T]$, if $\tilde Y_2'(\ell_1)\cap \tilde Y_2'(\ell_2)\neq \emptyset$, it contains at least one $\de\times r$-tube. 
Thus, \eqref{mu-T-step2} shows that for each $\de\times r$-tube $J\subset E_{\tilde L_2( T), \tilde Y_2'}$ parallel to $T$, 
\begin{equation} \label{it: muT}
    \#\{\ell\in \tilde L_2[T], J\subset \tilde Y_2'(\ell)\}\approx\mu_T\lessapprox c_\e^{-1} r^{-\e}(C\tilde C_2)^{\eta^2} \tilde\la^{-1/2}.
\end{equation}

Let $\tilde L_2=\bigcup_{T\in\cT_{\delta/r}}\tilde L_2[T]$ and let $Y_3'(\ell)=\tilde Y_2'(\ell)\cap Y_2(\ell)$ for any $\ell\in \tilde L_2$. 
Since $|Y_2(\ell)\cap J|$ are the same up to a constant multiple, $(\tilde L_2,Y_3')_\de$ is a refinement of  $(L_2,Y_2)_\de$.
By dyadic pigeonholing on the multiplicity factors $\{\mu_T: T\in\cT_{\de/r}\}$,  there exists a subset $L_3[T]\subset\tilde L_2[T]$ for each $T\in\cT_{\de/r}$ and a uniform
\begin{equation}
\label{small-multi-assumption}
    \mu_3\lessapprox c_\e^{-1} r^{-\e}(C\tilde C_2)^{\eta^2} \tilde\la^{-1/2}
\end{equation}
such that the following is true:
\begin{enumerate}
    \item If $\mu_T\sim\mu_3$, $L_3[T]=\tilde L_2[T]$; otherwise, $L_3[T]=\varnothing$. 
    \item Let $L_3=\bigcup_{T\in\cT_{\de/r}}L_3[T]$. Then $(L_3,Y_3')_\de$ is a refinement of $(\tilde L_2,Y_3')_\de$.
\end{enumerate}
It follows from \eqref{it: muT} that for all $\de/\rho\times1$-tube $T\in\cT_{\de/r}$ the following is true: 
For all $\delta \times r$-tube $J\subset T$ that is parallel to  $T$ and $ E_{L_3[T], Y_3'}\cap J\neq \varnothing$, we have  
\begin{equation}
\label{uniform-multi-assumption-pre}
    \#\{\ell\in L_3[T]:  Y_3'(\ell)\cap J\not=\varnothing\}\approx\mu_3.
\end{equation}

\smallskip

Note that for each $r$-ball $B$, $\tilde Y_2'(\ell)\cap B$ is either an empty set, or is essentially an $\de\times r$-tube.
Let $\cb_{r,3}$ be the set of finite-overlapping $r$-balls contained in $(E_{L_3, Y_3'})_r$.  
For each $B\in\cb_{r,3}$, consider the family of distinct  $\de\times r$-tubes 
\begin{equation}
    \cT_B^0:=\{J: \tilde Y_2'(\ell)\cap B\subset J\text{ for some $\ell\in L_3$}\}.
\end{equation}
Thus, if $Y_3'(\ell)\cap B\not=\varnothing$, then there exists a $\de\times r$-tube $J\in\cT_B^0$ such that $Y_3'(\ell)\cap B\subset J$.
By Lemma \ref{uniformization}, there is a uniform refinement $\cT_B$ of $\cT_B^0$.

Now for each $\ell\in L_3$, define a new shading $Y_3(\ell)\subset Y_3'(\ell)$ as follows: 
First, for each $r$-ball $B\in\cb_{r,3}$ such that $Y_3'(\ell)\cap B\not=\varnothing$, define
\begin{equation}
    Y_3(\ell)\cap B=\left\{\begin{array}{cc}
    Y_3'(\ell)\cap B,  & \text{ if $\exists J\in\cT_B$ such that $\tilde Y_2'(\ell)\cap B\subset J$;} \\
    \varnothing,   & \text{otherwise}.
    \end{array}\right.
\end{equation}
Then, define $Y_3(\ell)=\bigcup_{B\in\cb_{r,3}}Y_3(\ell)\cap B$.

Since $\cT_B$ is a refinement of $\cT_B^0$, it follows from  \eqref{uniform-multi-assumption-pre} that for each $B\in\cb_{r,3}$,
\begin{equation}
    \int_B \#(L_3)_{Y_3'}(x)\lessapprox\int_B \#(L_3)_{Y_3}(x).
\end{equation}
This shows that $(L_3, Y_3)_\de$ is a refinement of $(L_3, Y_3')_\de$. 

\smallskip

As a remark, we remind the reader that the following is true:
For all $\ell\in L_3$ and all $\de\times r$-tubes $J\subset (Y_3(\ell))_r$, we have
\begin{enumerate}
    \item $Y_3(\ell)\cap J = Y_2(\ell)\cap J$.
    \item The $r^{-1}$-dilate of $Y_3(\ell)\cap J$ along $\ell$ is an $(\de/r, s, (\de/r)^{-\eta_1})$-set.
    \item $|Y_3(\ell)\cap J|\lesssim (\de/r)^{1-s-\eta_1}|J|$.
    \item $|Y_3(\ell)\cap J|$ are about the same.
\end{enumerate}
Moreover, for each $B\in\cb_{r,3}$, there is a uniform set of $\de\times r$-tubes $\cT_B$ such that for all $\ell\in L_3$, $Y_3(\ell)\cap B\subset J$ for some $J\in\cT_B$.
Also, by \eqref{uniform-multi-assumption-pre}, for all $J\in\cT_B$, we have
\begin{equation}
\label{uniform-multi-assumption}
    \#\{\ell\in L_3:  J\subset (Y(\ell_3))_r\}\approx\mu_3.
\end{equation}

\medskip

\noindent {\bf Step 3: A broad-narrow argument.}

Apply Lemma \ref{broad-narrow-lem} to $(L_3, Y_3)_\de$ (with $\al$ in place of $\rho$), there is an $\al\in[\de,1]$ and a set $E_\al\subset E_{L_3}$ such that
\begin{enumerate}
    \item For each $x\in E_\al$, there exists a $10\al$-cap $\si_x\subset\ZS^1$ and a refinement of $L_3'(x)$ of $L_3(x)$ such that the direction $V(\ell)\in \si_x$ for all $\ell\in L_3'(x)$.
    \item There are two subsets of lines $L', L''\subset L_3'(x)$ such that $\# L', \# L''\gtrapprox \# L_3'(x)$, and $\al\geq\ang(\ell',\ell'')\gtrapprox \al$ for all $\ell'\in L',\ell''\in L''$.
\end{enumerate}
Define for each $\ell\in L_3$ a new shading $Y_4(\ell)=E_\al\cap Y_3(\ell)$, so $(L_3, Y_4)_\de$ is a refinement of $(L_3, Y_3)_\de$. 
Since $(L_3, Y_3)_\de$ is a refinement of $(L_2, Y_2)_\de$, by dyadic pigeonholing, there exists a refinement $(L_4, Y_4)_\de$ of $(L_3, Y_4)_\de$ such that $Y_4(\ell)$ is a refinement of $Y_2(\ell)$ for all $\ell\in L_4$.

\smallskip

Suppose $\al\leq (\de/r)^{\eta_1}$. 
Recall that $Y_2(\ell)$ is uniform for all $\ell\in L_2\supset L_4$.
Since for each $\ell\in L_4$, $Y_4(\ell)$ is a refinement of $Y_2(\ell)$, by Lemma \ref{uniformization}, there exists a new shading $Y_4'$ such that $Y_4'(\ell)$ is uniform and is a refinement of $Y_4(\ell)$ for all $\ell\in L_4$. 
For each $\ell\in L_4$, consider a new shading $\tilde Y_4'(\ell)=(Y_4'(\ell))_{\de/\al}$.
Since $Y_4'(\ell)$ is uniform, there exists an $\kappa(\ell)<1$ such that the following is true:
\begin{enumerate}
    \item $|Y_4'(\ell)|\sim \kappa(\ell)|\tilde Y_4'(\ell)|$.
    \item $|Y_4'(\ell)\cap J|\sim \kappa(\ell)$ for each $\de\times (\de/\al)$-tube $J\subset (Y_4'(\ell))_{\de/\al}$.
\end{enumerate}
By dyadic pigeonholing on $\{\kappa(\ell): \ell\in L_4\}$, there exists a uniform $\kappa<1$ and a set of lines $L_4'\subset L_4$ such that the following is true. 
\begin{enumerate}
    \item $(L_4',Y_4')_\de$ is a refinement of $(L_4,Y_4)_\de$.
    \item $\kappa(\ell)\sim \kappa \leq 1$ for all $\ell\in L_4'$.
\end{enumerate}

\smallskip

Let $\cT$ be a collection of $\alpha$-separated $\alpha\times 1$-tubes, and for each $T\in\cT_{\alpha}$, recall Definition \ref{L[T]-def} for $L_4'[T]$.
Then
\begin{equation}
\label{narrow}
    |E_{L_4}|\gtrsim\sum_{T\in\cT_{\alpha}}|E_{L_4'[T]}|
\end{equation}
For each $T\in \cT_{\alpha}$, let $\phi_T$ be the non-isotropic $\alpha^{-1}$-dilate which maps  $T$ to the unit ball, so  $|\phi_T(E_{L_4'[T]})| = (\alpha/\delta) |E_{L_4[T]}|$.
For all $\ell\in L_4'\subset L_2$, recall that $Y_2(\ell)$ is a $(\de,\e^2,CC_2;\rho)$-set for some $C_2\lessapprox1$.
Since $Y'_4(\ell)$ is a refinement of $Y_2(\ell)$, $Y_4'(\ell)$ is a $(\de,\e^2,CC_4;\rho)$-set for some $C_4\lessapprox1$.
Thus, $\tilde Y_4'(\ell)$ is a $(\de/\al,\e^2,CC_4;\max\{\rho,\de/\al\})$-set, 
implying that $\phi_T(\tilde Y_{4}')$ is a $(\de/\al,\e^2,CC_4;\max\{\rho,\de/\al\})$-set.
Clearly, $\max\{\rho,\de/\al\}\in[\de/\al,(\de/\al)^\eta]$, since $\rho\in[\de,\de^\eta]$.
Moreover, since $Y_4'(\ell)$ is $\gtrapprox\la$-dense, $\tilde Y_{4}'(\ell)$ is $\gtrapprox\la \kappa^{-1} $-dense, which implies $\phi_T(\tilde Y_{4}')$ is $\gtrapprox\la \kappa^{-1} $-dense.
Apply induction at scale $\de/\alpha$ to $(\phi_{T}(L_{4}'[T]),\phi_T(\tilde Y_{4}'))_{\de/\alpha}$ so that
\begin{align}
    |E_{L_4'[T]}| \gtrsim \kappa |E_{L_{4}'[T], \tilde Y_{4}'}|&\gtrapprox c_\e(\de/\alpha)^\e(CC_4)^{-\eta^{-2}}\la^{1/2}\kappa^{1/2}\sum_{\ell\in L_{4}'[T]}|\tilde Y_{4}'(\ell)|\\
    &\gtrapprox c_\e(\de/\alpha)^\e(CC_4)^{-\eta^{-2}}\la^{1/2}\kappa^{-1/2}\sum_{\ell\in L_{4}'[T]}|Y_4'(\ell)|.
\end{align}
Finally, by \eqref{narrow} and since $(L_4', Y_4')_\de$ is a refinement of $(L,Y)_\de$, 
\begin{equation}
    |E_{L}|\geq|E_{L_4}|\geq  c_\e\de^{\e}C^{-\eta^{-2}} \la^{1/2}\sum_{\ell\in L}|Y(\ell)|.
\end{equation}

\smallskip

Suppose $\al\geq(\de/r)^{\eta_1}$. We proceed to the next step.

\medskip

\noindent {\bf Step 4: Conclude \eqref{main-prop-esti} directly when $r\geq\de^\eta$.} 
In this case, we will show that $(L_3,Y_3)_\de$ obeys the assumption of Corollary \ref{cor-induction} with $2\eta$, and then apply it to conclude \eqref{main-prop-esti}.

From the last step, we obtain a refinement $(L_4,Y_4)_\de$ of $(L_3,Y_3)_\de$ such that for all $x\in E_{L_4}$, $(L_{3})_{Y_3}(x)$  is ``two-broad" (i.e. $\al \geq (\de/r)^{\eta_1}$), as $ E_{L_4}\subset E_\al$.
Since $Y_3(\ell)\cap J$ is a $(\de/r, s, (\de/r)^{-\eta_1})$-set and $|Y_3(\ell)\cap J|\lesssim (\de/r)^{1-s-\eta_1}|J|$ for all $J\subset(Y_3(\ell))_r$ (see the end of Step 2) and all $\ell\in L_3\supset L_4$, as $r\geq\de^\eta$, the following is true:
\begin{enumerate}
    \item $Y_3(\ell)$ is a $(\de,s,\de^{-\eta_1-\eta})$-set for all $\ell\in L_3$.
    \item $|Y_3(\ell)|\lesssim \de^{1-s-\eta_1-\eta}|N_\de(\ell)|$.
    \item As a consequence of item (1) above, if $Y_4(\ell)$ is a refinement of $Y_3(\ell)$, then by Lemma \ref{refinement-delta-s-set}, $Y_4(\ell)$ is a $(\de,s, \de^{-2\eta_1-\eta})$-set.
\end{enumerate}

Since $(L_4,Y_4)_\de$ is a refinement of $(L_3, Y_3)_\de$, there exists an $\ell'\in L_4$ such that $Y_4(\ell')$ is a refinement of $Y_3(\ell')$, yielding that $Y_4(\ell')$ is a $(\de,s, \de^{-2\eta_1-\eta})$-set.
Since $(L_{3})_{Y_3}(x)$  is two-broad for all $x\in Y_4(\ell')$, $\#\{T\in\cT_v:L_3[T]\not=\varnothing\}\gtrsim (\al/v)^{s-2\eta_1-\eta}\geq v^{-s}\de^{2\eta_1+\eta}$ (recall Definition \ref{L[T]-def}) for all $v\in[\de,1]$, where $\cT_v$ is a maximal family of distinct $v\times1$-tubes in $B^2(0,1)$.
Moreover, since $(L_3,Y_3)_\de$ is a refinement of $(L_2, Y_2)_\de$, there exists an $\ell\in L_3$ such that $|Y_3(\ell)|\approx|Y_2(\ell)|\approx\la$, yielding 
\begin{equation}
\label{lambda-upper-bound}
    \la\lessapprox \de^{1-s-\eta_1-\eta}.
\end{equation}

Recall that $\eta$ is taken so that Corollary \ref{cor-induction} is true with $(\nu,\eta)=(\e^2,2\eta)$.
Since $\#L_3\gtrapprox\#L$ and since $\#L\lesssim\de^{-1}$, apply Corollary \ref{cor-induction} to $(L_3,Y_3)_\de$ with $(\nu,\eta)=(\e^2,2\eta)$ so that
\begin{align}
    |E_L|\gtrsim \de^{\e^2/2+2\eta}\de^{2-3s/2}(\# L_3)^{1/2}\gtrapprox\de^{\e^2/2+2\eta}\de^{3/2-3s/2}(\de\# L).
\end{align}
Recall \eqref{lambda-upper-bound}.
Since $\eta_1\leq\eta\leq\e^2$ and since $|Y_1(\ell)|\sim \la$ for all $\ell\in L_1$,
\begin{equation}
    |E_L|\gtrsim \de^{\e^2/2+O(\eta)}\la^{3/2}(\de\#L)\gtrsim\de^{O(\e^2)}\la^{1/2}\sum_{\ell\in L_1}|Y_1(\ell)|.
\end{equation}
This shows \eqref{main-prop-esti}, as $(L_1,Y_1)_\de$ is a refinement of $(L,Y)_\de$.

\smallskip

When $r\leq\de^\eta$ (we are only going to use this assumption in Step 7, Case {\bf C}), we proceed to the next step.

\medskip

\noindent {\bf Step 5: Setting up an incidence problem inside an $r$-ball.}

For $k=3,4$ and each $\ell\in L_k$,  let $\cj_k(\ell)=\{J\subset(Y_k(\ell))_r\}$.
Recall that $Y_4(\ell)$ is a $\gtrapprox 1 $-refinement of $Y_2(\ell)$. 
Since $Y_4(\ell)\subset Y_3(\ell)\subset Y_2(\ell)$ and since $Y_2(\ell)$ is uniform, we have $\#\cj_4(\ell)\leq\#\cj_3 (\ell)\lessapprox\#\cj_4(\ell)$.

Let $\cj_5(\ell)=\{J\in\cj_4(\ell): |Y_4(\ell)\cap J|\gtrapprox|Y_3(\ell)\cap J|\}$. 
Recall that  $|Y_3(\ell)\cap J|=|Y_2(\ell)\cap J|$ are  the same up to a constant multiple for all $J\in\cj_3(\ell)$. 
Thus, since $Y_4(\ell)$ is a $\gtrapprox 1 $-refinement of $Y_3(\ell)$, we have $\#\cj_5(\ell)\gtrapprox\#\cj_4(\ell)$. 
Consider a new shading  $Y_5(\ell)=\bigcup_{J\in\cj_5(\ell)}J\cap Y_4(\ell)$, so $Y_5(\ell)$ is a $\gtrapprox 1 $-refinement of $Y_4(\ell)$. 
Consequently, by taking $L_5=L_4$, $(L_5,Y_5)_\de$ is a refinement of $(L_4,Y_4)_\de$.

\smallskip

We remark that the set of lines and shading $(L_5,Y_5)_\de$ has the following properties:
\begin{enumerate}
    \item For all $x\in E_{L_5}$, $(L_{3})_{Y_3}(x)$  is ``two-broad" (i.e. $\al \geq (\de/r)^{\eta_1}$, as assumed at the end of Step 3), since $E_{L_5}\subset E_{L_4}\subset E_\al$.
    \item For all $\ell\in L_5$ and all $\de\times r$-tube $J\in\cj_5(\ell)$, the (1-dimensional) $r^{-1}$-dilate of $ Y_5(\ell)\cap J$ along the line $\ell$ is a $(\de/r, s, (\de/r)^{-2\eta_1})$-set. 
    This is a consequence of Lemma \ref{refinement-delta-s-set} and the fact that the $r^{-1}$-dilate of  $Y_3(\ell)\cap J$ is a $(\de/r, s, (\de/r)^{-\eta_1})$-set (see the end of Step 2).
\end{enumerate}

Let $\cb_{r,5}\subset\cb_{r,3}$ be the family of $r$-balls contained in $(E_{L_5, Y_5})_r$.
In other words, for any $B\in\cb_{r,5}$, $Y_5(\ell)\cap B\not=\varnothing$ for some $\ell\in L_5$.  
For each $r$-ball $B\in\cb_{r,5} \subset \cb_{r, 3}$, let $\cT_B$ be the set of distinct $\delta\times r$-tube segments defined in the end of Step 2.  
We restate the properties of $\cT_B$ here:
\begin{enumerate}
    \item For each $\ell \in L_3$ with $Y_3(\ell)\cap B\neq \varnothing$, there exists $J\in \cT_B$ that $J\subset (Y_3(\ell))_r$.
    \item For any two $J, J'\in \cT_B$, $|J\cap J'|\leq |J|/2$.
    \item $\cT_B$ is uniform, and for each $J\in\cT_B$, \eqref{uniform-multi-assumption} is true.
\end{enumerate}
We are going to analyze the incidences between $\delta\times r$-tubes in $\cT_B$ and $\delta$-balls in $B\cap E_{L_3, Y_3}$ (here we choose $E_{L_3, Y_3}$ over $E_{L_5, Y_5}$ because of the following: Each $x\in E_{L_5, Y_5}$ is ``two-broad" with respect to $(L_3,Y_3)_\de$. 
Such property does not hold if $(L_3,Y_3)_\de$ is replaced by $(L_5,Y_5)_\de$).

\smallskip
First, we study the structure of $\cT_B$. 
Pick $\ell\in L_5$ such that $Y_5(\ell)\cap B\not=\varnothing$, so the $r^{-1}$-dilate of $ Y_5(\ell)\cap B$ along the line $\ell$ is a $(\de/r, s, (\de/r)^{-2\eta_1})$-set. 
Recall the ``two-broad" property obtained in Step 3: 
For each $x\in E_{L_5} \subset E_{\al}$, there exists $\ell'\in L_3$ such that $ x\in Y_3(\ell')$ and $\angle (\ell, \ell')\gtrapprox \al\geq(\de/r)^{\eta_1} $.
Hence, for each $v\in[\de,r]$, the number of distinct $v\times r$-tubes required to cover $\cT_B$ is $\gtrsim (r\al/v)^{s-2\eta_1}\geq (r/v)^{s}(\de/r)^{3\eta_1}$. 
Let $\cT_v(B)$ denote this set of $v\times r$-tubes.

Next, for each $J\in\cT_B$, define
\begin{equation}
    L_3(J):=\{\ell\in L_3: J\subset(Y_3(\ell))_r\},
\end{equation}
and let $L_3(B)=\bigsqcup_{J\in\cT_B}L_3(J)$. 
Define $K_J=\#L_3(J)$, so by \eqref{uniform-multi-assumption}, $K_J\approx \mu_3$. 
Let $\ell_1(J),\ldots, \ell_{K_J}(J)$ be an enumeration of the lines in $L_3(J)$.
Such enumeration on each $L_3(J)$ gives a natural disjoint partition of $L_3(B)$: 
Let $L_{3,k}(B)=\{\ell_{k}(J):J\in\cT_B\}$, where $\ell_{k}(J):=\varnothing$ if $k\geq K_J$.
For each $k$ and each $J\in\cT_B$, define a shading
\begin{equation}
\label{shading-B-k}
    Y_{B,k}(J)=\left\{\begin{array}{cc}
    Y_3(\ell_k(J)) \cap J,  & \text{ if } k \leq K_J \\
    \varnothing,   & \text{otherwise}.
    \end{array}\right.
\end{equation}
Let $K_B = \min_{ J\in \cT_B} K_J\approx \mu_3$, so when $k\leq K_B$, $Y_{B,k}(J)\not=\varnothing$ for all $J\in\cT_B$.
Since $|Y_3(\ell)\cap J|$ are about the same for all $\ell\in L_3(J)$ and all $J\subset(Y_3(\ell))_r$, we know that $\sum_{J\in\cT_B}|Y_{B,k}(J)|\approx\sum_{J\in\cT_B}|Y_{B,k'}(J)|$ when $k,k'\leq K_B$ and $\sum_{J\in\cT_B}|Y_{B,k}(J)|\gtrapprox\sum_{J\in\cT_B}|Y_{B,k'}(J)|$ when $k\leq K_B\leq k'$.
Notice that $Y_{B,k}=\varnothing$ whenever $k\geq\max_{J\in\cT_B}K_J\approx\mu_3\approx K_B$.
Consequently,
\begin{equation}
\label{a-lot-of-J}
    \int_{B}\#L_3(x)=\sum_k\sum_{p\subset E_{\cT_B, Y_{B,k}}}\#\cT_B(p)\lessapprox\sum_{k\leq K_B}\sum_{p\subset E_{\cT_{B}, Y_{B,k}}}\#\cT_{B}(p).
\end{equation}

\medskip

\noindent {\bf Step 6.1: Analyze $E_L\cap B$ for each $r$-ball $B$ in the broad case (I).}

Recall that for all $B\in\cb_{r,5}$ and all $k\leq K_B$,  the configuration $(\cT_B, Y_{B,k})$ satisfies the following properties:
\begin{enumerate}
    \item $\cT_B$ is uniform.
    \item For each $v\in[\de,r]$, there are $\gtrsim (r/v)^{s}(\de/r)^{3\eta_1}$ distinct $v\times r$-tubes containing at least one $J\in\cT_B$. 
    \item For each $J\in\cT_B$, $|Y_{B,k}(J)|\lesssim(\de/r)^{1-s-\eta}|J|$, and the one-dimensional $r^{-1}$-dilate of $Y_{B,k}(J)$ along $J$ is a $(\de/r, s, (\de/r)^{-\eta_1})$-set.
\end{enumerate}

\smallskip

We want to apply Proposition \ref{prop-induction} to the $r^{-1}$-dilate of $(\cT_B, Y_{B,k})$. 
However, the $r^{-1}$-dilate of $(\cT_B, Y_{B,k})$ may not obey its Assumption (4).   
To get around this issue, we apply (a variant of) Lemma \ref{rich-point-refinement} to $(\cT_B, Y_{B,k})$ to obtain a refinement $(\cT_{B,k}', Y_{B,k}')$ of $(\cT_B,Y_{B,k})$ such that
\begin{enumerate}
    \item For any $J\in\cT_{B,k}'$, the one-dimensional $r^{-1}$-dilate of $Y_{B,k}'(J)$ along $J$ is a $(\de/r, s, (\de/r)^{-2\eta_1})$-set. 
    This can be achieved with the help of Lemma \ref{refinement-delta-s-set}.
    \item Let $E_{\cT_{B,k}',Y_{B,k}'}=\bigcup_{T\in\cT_{B,k}'}Y_{B,k}'(J)$. For any $\de$-ball $p\subset E_{\cT_{B, k}', Y_{B, k}'}$, $\cT_{B,k}'(p):=\{J\in\cT_{B,k}':p\subset Y_{B,k}'(J) \}$ satisfies 
    \begin{equation}
        \#\cT_{B,k}'(p) \lessapprox|E_{\cT_{B,k}',Y_{B,k}'}|^{-1}\sum_{J\in\cT_{B,k}'}|Y_{B,k}'(J)|.
    \end{equation}
\end{enumerate}
Recall that $|Y_{B,k}(J)|$ are about the same for all $J\in\cT_B$, which yields $\#\cT_{B,k}'\gtrapprox\#\cT_B$. Apply Lemma \ref{uniformization} to $\cT_{B,k}'$ so that there is a uniform refinement $\cT_{B,k}''$ of $\cT_{B,k}'$.
For each $J\in\cT_B$, let  $Y_{B,k}''(J)=Y_{B,k}'(J)$ if $J\in \cT_{B,k}''$;  otherwise, let $Y_{B,k}''(J)=\varnothing$. 

\smallskip

In order to apply Proposition \ref{prop-induction}, let us check the requirement for the parameters $\nu$ and $\eta$. 
Recall \eqref{epsilon} that the parameter $\eta$ is chosen such that Proposition \ref{prop-induction} is true when $(\nu,\eta)=(\e^2,2\eta)$. 
Since $(\cT_{B,k}'', Y_{B,k}'')$ is a refinement of $(\cT_B, Y_{B,k})$ and since the uniform set $\cT_{B,k}''$ is also a refinement of the uniform set $\cT_B$, we know that
\begin{enumerate}
    \item For each $J\in\cT_{B,k}''$, the one-dimensional $r^{-1}$-dilate of $Y_{B,k}''(J)$ along the tube segment $J$ is a $(\de/r, s, (\de/r)^{-2\eta_1})$-set.
    \item For each $v\in[\de,r]$,  there are $\gtrsim (r/v)^{s}(\de/r)^{4\eta_1}$ many distinct $v\times r$-tubes in $\cT_v(B)$ containing at least a tube $J\in\cT_{B,k}''$. 
    \item We have $\#\cT_{B,k}''(p)\leq \#\cT_{B,k}'(p)\lessapprox|E_{\cT_{B,k}',Y_{B,k}'}|^{-1}\sum_{J\in\cT_{B,k}'}|Y_{B,k}'(J)|\lessapprox\\|E_{\cT_{B,k}'',Y_{B,k}''}|^{-1}\sum_{J\in\cT_{B,k}''}|Y_{B,k}''(J)|$ for all $ p\subset E_{\cT_{B,k}'',Y_{B,k}''}$.
\end{enumerate}

Now we can apply Proposition \ref{prop-induction} to the $r^{-1}$-dilate of $(\cT_{B,k}'', Y_{B,k}'')$ (recall that $4\eta_1\leq\eta$). 
After rescaling back, there are three possible outcomes: 
\begin{enumerate}
    \item[{\bf A.}]  Define $\mu_{B,k}:=\max_{p\subset E_{\cT_{B,k}'', Y_{B,k}''}} \# \cT_{B,k}''(p) $, then $\mu_{B,k}\lessapprox (\de/r)^{(s-1)/2}$.
    \item[{\bf B.}] We have
    \begin{equation}
        |E_{\cT_{B,k}'', Y_{B,k}''}|\gtrapprox (\de/r)^{\e^2+3\eta}(\de/r)^{3(1-s)/2}|B|.
    \end{equation}
    \item[{\bf C.}] There exists a scale $\De_B\geq r(\de/r)^{1-\sqrt{\eta}}$ such that $|E_{\cT_{B,k}'', Y_{B,k}''}\cap N_{\De_B}(J)\cap B |\gtrapprox (\de/r)^{1-s+2\eta}|N_{\De_B}(J)\cap B|\sim (\de/r)^{1-s+2\eta}(\De_B r)$ for each $J\in \cT_{B,k}''$.
\end{enumerate}

For each $k\leq K_B$, one of the above three outcomes must happen for the pair $(\cT_B,Y_{B,k})$. 
For $\bf X\in\{\bf A, B, C\}$, let $\ck_B({\bf X})$ be the set of $k\leq K_B$ such that outcome $\bf X$ holds for $(\cT_B,Y_{B,k})$.  
Since $(\cT_{B,k}'', Y_{B,k}'')$ is a refinement of $(\cT_B, Y_{B,k})$ for each $k$, we know from \eqref{a-lot-of-J} that
\begin{equation}
    \int_{B}\#L_3(x)\lessapprox\sum_{{k\leq K_B}}\sum_{p\subset E_{\cT_B, Y_{B,k}}}\#\cT_B(p)\lessapprox\sum_{k\leq K_B}\sum_{p\subset E_{\cT_{B,k}'', Y_{B,k}''}}\#\cT_{B,k}''(p).
\end{equation}
By pigeonholing, there exists an $\bf X\in\{\bf A, B, C\}$ such that 
\begin{equation}
\label{case-D}
    \int_{B}\#L_3(x)\lessapprox\sum_{k\in\ck_B(\bf X)}\sum_{p\subset E_{\cT_{B,k}'', Y_{B,k}''}}\#\cT_{B,k}''(p).
\end{equation}
We remark that \eqref{case-D} applies to each $B\in\cb_{r,5}$. 

\medskip

\noindent {\bf Step 6.2: Analyze $E_L\cap B$ for each $r$-ball $B$ in the broad case (II).}

Since $(L_5, Y_5)_\de$ is a refinement of $(L_3,Y_3)_\de$ and since for each $B\in\cb_{r,5}$, $B\cap Y_5(\ell)\not=\varnothing$ for some $\ell\in L_5$, we have 
\begin{equation}
    \int_{E_{L_3}}\# L_3(x)\lessapprox\int_{E_{L_5}}\# L_5(x)=\sum_{B\subset\in\cb_{r,5}}\int_{B}\#L_5(x)\leq\sum_{B\in\cb_{r,5}}\int_{B}\#L_3(x) .
\end{equation}
For each $\bf X\in\{\bf A, B, C\}$, let $\cb_{r,5}({\bf X})\subset \cb_{r,5}$ be the set of $r$-balls that \eqref{case-D} is true for this ${\bf X}$.
By pigeonholing, there exists a fixed $\bf X\in\{\bf A, B, C\}$ such that
\begin{equation}
\label{Y-6-refinement}
    \int_{E_{L_3}}\# L_3(x)\lessapprox\sum_{B\in\cb_{r,5}(\bf X)}\sum_{k\in\ck_B(\bf X)}\sum_{p\subset E_{\cT_{B,k}'', Y_{B,k}''}}\#\cT_{B,k}''(p) .
\end{equation}

The configurations $\{(\cT_{B,k}'', Y_{B,k}''):B\in\cb_{r,5}({\bf X}), k\in \ck_B(\bf X)\}$ defines a new shading $Y_6$ for each $\ell\in L_3$: 
Let $\cb_{r,5}(\ell)\subset\cb_{r,5}({\bf X})$ be that for each $B\in\cb_{r,5}(\ell)$, $Y_3(\ell)\cap B\not=\varnothing$. 
For each $B\in\cb_{r,5}(\ell)$, we know that 
\begin{enumerate}
    \item There is a $J\in\cT_B$ such that $J\subset (Y_3(\ell))_r$.
    \item There is a $k\leq K_J$ such that $\ell\in L_{3,k}(B)$.
    \item For this $(J,k)$, $Y_{B,k}(J)\supset Y_3(\ell)\cap B$.
\end{enumerate}
Now for this $J\in\cT_B$, we first define
\begin{equation}
\label{Y-6-local}
    Y_6(\ell)\cap B:=\left\{\begin{array}{cc}
    Y_{B,k}''(J)\cap B, & \text{ if } k\in\ck_B(X);\\
    \varnothing,     & \text{otherwise}.
    \end{array}\right.
\end{equation}
Then, we define
\begin{equation}
\label{Y-6}
    Y_6(\ell)=\bigcup_{B\in\cb_{r,5}(\ell)}Y_6(\ell)\cap B.
\end{equation}
Let $L_6=L_3$. By \eqref{Y-6-refinement}, we know that $(L_6, Y_6)_\de$ is a refinement of $(L_3, Y_3)_\de$.

\smallskip

The definition of  $(L_6, Y_6)_\de$ is subject to the outcome $\bf X\in\{A,B,C\}$ from pigeonholing.  
Let us discuss what happens for different possibilities of $\bf X$. 
Define $\cb_{r,6}=\cb_{r,5}(\bf X)$, so $E_{L_6}\subset \cup_{\cb_{r,6}}$.

When $\bf X=A$: For each $B\in\cb_{r,6}$, we have $\mu_{B,k} \lessapprox (\de/r)^{(s-1)/2}$ for each $k\in\ck_B(\bf X)$. 
Consequently, since $\#\ck_B({\bf X})\leq K_B\lessapprox \mu_3$ and by \eqref{Y-6-local} and \eqref{Y-6}, for each $x\in E_{L_6} \cap B$,  we have 
\begin{equation}
    \# L_6(x)\leq \sum_{k\in\ck_B(\bf X)}\mu_{B,k}\lessapprox \#\ck_B({\bf X})\cdot \mu_{B,k}\lessapprox \mu_3\cdot (\de/r)^{(s-1)/2}.
\end{equation}

When $\bf X=B$: For each $B\in\cb_{r,6}$, since $E_{\cT_{B,k}'', Y_{B,k}''}\subset E_{L_6}\cap B$ for all $k\in\ck(\bf X)$, 
\begin{equation}
    |E_{L_6}\cap B|\geq\max_{k\in\ck_B(\bf X)}|E_{\cT_{B,k}'', Y_{B,k}''}|\gtrapprox (\de/r)^{\e^2+3\eta}(\de/r)^{3(1-s)/2}|B|.
\end{equation} 

When $\bf X=C$: For each $B\in\cb_{r,6}$, there exists a scale $\De_B\geq r(\de/r)^{1-\sqrt{\eta}}$ such that $|E_{\cT_{B,k}'', Y_{B,k}''}\cap N_{\De_B}(J)\cap B |\gtrapprox (\de/r)^{1-s+2\eta}|N_{\De_B}(J)\cap B|\sim (\de/r)^{1-s+2\eta}(\De_B r)$ for all $J\in \cT_{B,k}''$ and all $k\in\ck_B(\bf X)$. 
Since  $E_{\cT_{B,k}'', Y_{B,k}''}\subset E_{L_6}$ for all $k\in\ck_B(\bf X)$, we know that for all $\ell\in L_6$ and all $B\in\cb_{r,6}$, if $Y_6(\ell)\cap B\not=\varnothing$, then
\begin{equation}
    |E_{L_6}\cap N_{\De_B}(\ell)\cap B|\gtrapprox  (\de/r)^{1-s+\eta}(\De_B r).
\end{equation}

\smallskip

Let us summarize what we have so far: $(L_6, Y_6)_\de$ is a refinement of $(L_2, Y_2)_\de$, and for $(L_6, Y_6)_\de$, either one of the following must happen:
\begin{enumerate}
    \item[{\bf A.}]  For each $x\in E_{L_6}$, 
    \begin{equation}
        \# L_6(x)\lessapprox \mu_3\cdot (\de/r)^{(s-1)/2}.
    \end{equation}
    \item[{\bf B.}] For each $B\in\cb_{r,6}$,
    \begin{equation}
        |E_{L_6}\cap B|\gtrapprox (\de/r)^{\e^2+3\eta}(\de/r)^{3(1-s)/2}|B|.
    \end{equation} 
    \item[{\bf C.}] For all $B\in\cb_{r,6}$, there exists a scale $\De_B\geq r(\de/r)^{1-\sqrt{\eta}}$ such that for all $\ell\in L_6$, if $Y_6(\ell)\cap B\not=\varnothing$, then
    \begin{equation}
        |E_{L_6}\cap N_{\De_B}(\ell)\cap B|\gtrapprox  (\de/r)^{1-s+2\eta}(\De_B r).
    \end{equation}
\end{enumerate}
We proceed to the next step.

\medskip

\noindent {\bf Step 7: Estimate $|E_L|$: Three cases.}

\smallskip

{\bf Case A.} Suppose item {\bf A} happens. 
By \eqref{tilde-lambda} and \eqref{small-multi-assumption}, we have 
\begin{align}
    \#L_6(x)&\lessapprox c_\e^{-1} r^{-\e}(C\tilde C_2)^{\eta^{-2}}   \la^{-1/2} (r/\de)^{-(1-s-\eta_1)/2}\cdot (\de/r)^{(s-1)/2}\\
    &= c_\e^{-1}r^{-\e}(r/\de)^{\eta_1/2}(C\tilde C_2)^{\eta^{-2}} \la^{-1/2}.
\end{align}
Recall that $(L_6, Y_6)_\de$ is a refinement of $(L,Y)_\de$. 
Since $\tilde C_2\lessapprox1$ and since $\eta_1\leq\e/10$, 
\begin{align}
    |E_L&|\geq|E_{L_6}|\gtrapprox c_\e r^{\e}(\de/r)^{\eta_1/2}(C\tilde C_2)^{-\eta^2}  \la^{1/2}\sum_{\ell\in L_6}|Y_6(\ell)|\\
    &\gtrapprox c_\e \de^{\e}(\de/r)^{\eta_1/2-\e}(C\tilde C_2)^{-\eta^{-2}}\la^{1/2}\sum_{\ell\in L}|Y(\ell)|\geq c_\e\de^{\e}C^{-\eta^{-2}}\la^{1/2}\sum_{\ell\in L}|Y(\ell)|.
\end{align}
This concludes Proposition \ref{two-ends-furstenberg-prop}.

\smallskip

{\bf Case B.}  Suppose item {\bf B} happens. Then for each $B\in\cb_{r,6}$,
\begin{equation}
\label{high-density-in-r-ball}
    |E_{L}\cap B|\geq |E_{L_6}\cap B|\gtrapprox (\de/r)^{\e^2+3\eta}(\de/r)^{3(1-s)/2}|B|.
\end{equation}
Let $L_7$ be the set of lines $\ell\in L_6$ such that $|Y_6(\ell)| \gtrapprox |Y_2(\ell)|\approx\la$.
Since $(L_6,Y_6)_\de$ is a refinement of $(L_2,Y_2)_\de$,
$\#L_7 \gtrapprox \#L_2$. 
For each $\ell\in L_7$, let  $Y_7(\ell)$ be a uniform refinement of $ Y_6(\ell)$, so $|Y_7(\ell)|\approx \la$.
Note that $(L_7,Y_7)_\de$ is a refinement of $(L_2,Y_2)_\de$.

For each $\ell\in L_7$, let $\cb_{r,6}(\ell)\subset\cb_{r,6}$ be so that $Y_7(\ell)\cap B\not=\varnothing$ for all $B\in\cb_{r,6}(\ell)$.
Now define a new shading 
\begin{equation}
    \tilde Y_8(\ell)=\cup_{\cb_{r,6}(\ell)}\cap N_{r}(\ell)
\end{equation}
by $r$-balls. 
Since $|Y_7(\ell)\cap B|\lesssim (\de/r)^{1-s-\eta_1}|N_\de(\ell)\cap B|$ for each $B\in \cb_{r, 6}(\ell)$ (by item (2) in the outcome of Proposition~\ref{multiscale-prop} in Step 1), we have (recall \eqref{tilde-lambda})
\begin{equation}
  |\tilde Y_8(\ell)|/ |N_r(\ell)| \gtrsim \de^{-1}|Y_7(\ell)|\cdot (r/\de)^{1-s-\eta_1}\gtrapprox \tilde \la.
\end{equation}

By Lemma \ref{katz-tao-set-lem}, there exists $\tilde L_8\subset L_7$ such that $(r\# \tilde L_8)\gtrapprox (\de \# L_7)$, and $\tilde L_8$ is a Katz-Tao $(r,1)$-set. 
For all $\ell\in L_7$, since $Y_7(\ell)$ is a refinement of $Y_2(\ell)$ and since $Y_2(\ell)$ is a $(\de,\e^2,CC_2;\rho)$-set for some $C_2\lessapprox1$, $Y_7(\ell)$ is a $(\de,\e^2,CC_7;\rho)$-set for some $C_7\lessapprox1$.
Since $Y_7(\ell)$ is uniform, $\tilde Y_8(\ell)$ is a $(r,\e^2,CC_8;\max\{\rho,r\})$-set for some $C_8\lessapprox1$.
Clearly, $\max\{\rho,r\}\in[r,r^\eta]$, since $\rho\in[\de,\de^\eta]$.
Apply Proposition \ref{two-ends-furstenberg-prop} at scale $r$ to $(\tilde L_8, \tilde Y_8)_{r}$ so that
\begin{align}
    |E_{\tilde L_8}|&\geq c_\e r^{\e}(CC_8)^{-\eta^{-2}}  \tilde{\lambda}^{1/2} \sum_{\ell\in \tilde L_8}|\tilde Y_8(\ell)| \gtrapprox c_\e r^{\e}(CC_8)^{-\eta^{-2}} \tilde{\lambda}^{3/2}  \, (r \# \tilde L_8) \\
    &\gtrapprox c_\e r^{\e}(CC_8)^{-\eta^{-2}} \tilde{\lambda}^{3/2}  \, (\de \# L_7)\gtrapprox  c_\e r^{\e}(CC_8)^{-\eta^{-2}} \tilde{\lambda}^{3/2}  \lambda^{-1} \sum_{\ell \in L_7} |Y_7(\ell)|.
\end{align}
Recall \eqref{tilde-lambda}. As a result, we have
\begin{align}
    |E_{\tilde L_8}|\gtrapprox c_\e r^{\e}(CC_8)^{-\eta^{-2}}\la^{1/2}(r/\de)^{3(1-s-\eta_1)/2}\sum_{\ell\in L_7}|Y_7(\ell)|.
\end{align}
By \eqref{high-density-in-r-ball} and since $(L_7,Y_7)_\de$ is a refinement of $(L,Y)_\de$, we thus have
\begin{align}
    &|E_L| \gtrsim \big(\min_{B\in\cb_{r,6}}  \frac{  |E_L\cap B| }{|B|}\big) \cdot |E_{L_8}|  \gtrapprox (\de/r)^{\e^2+3\eta}(\de/r)^{3(1-s)/2}|E_{L_8}|\\[1ex]
    &\gtrapprox (\de/r)^{\e^2+3\eta}(\de/r)^{3(1-s)/2}\cdot c_\e r^{\e}(CC_8)^{-\eta^{-2}}\la^{1/2}(r/\de)^{3(1-s-\eta_1)/2}\sum_{\ell\in L_7}|Y_7(\ell)|\\
    &\gtrapprox (\de/r)^{\e^2+3\eta+3\eta_1/2-\e}C_8^{-\eta^{-2}}\cdot c_\e \de^{\e}C^{-\eta^{-2}}\la^{1/2}\sum_{\ell\in L}|Y(\ell)|.
\end{align}
Since $\eta_1\leq\eta\leq\e^2$, we get $(\de/r)^{\e^2+3\eta+3\eta_1/2-\e}C_8^{-\eta^{-2}}\geq (\de/r)^{-\e/2}$. 
Therefore, we obtain $|E_L|\geq c_\e \de^{\e}C^{-\eta^{-2}}\la^{1/2}\sum_{\ell\in L}|Y(\ell)|$, as desired.

\medskip

{\bf Case C.} Suppose item {\bf C} happens. Then for each $B\in\cb_{r,6}$,  there exists a scale $\De_B\geq r(\de/r)^{1-\sqrt{\eta}}$ such that $|E_{L_6}\cap N_{\De_B}(\ell)\cap B|\gtrapprox  (\de/r)^{1-s+\eta}(\De_B r)$ for all $\ell\in L_6$. 
Note that $Y_6(\ell)\subset\cup_{\cb_{r,6}}$ for any $\ell\in L_6$. 
Since $(L_6,Y_6)_\de$ is a refinement of $(L_2,Y_2)_\de$, 
\begin{equation}
    \int_{E_{L_2}}\# L_2(x)\lessapprox\int_{\cup_{\cb_{r,6}}}\# L_6(x).
\end{equation}

By dyadic pigeonholing on $\{\De_B: B\in\cb_{r,6}\}$, there exists a subset $\cb_{r,6}'\subset\cb_{r,6}$ and a uniform scale $\De\geq r(\de/r)^{1-\sqrt{\eta}}$ such that 
\begin{enumerate}
    \item For all $B\in\cb_{r,6}'$ and all $\ell\in L_6$, $|E_{L_6}\cap N_{\De}(\ell)\cap B|\gtrapprox  (\de/r)^{1-s+2\eta}(\De r)$.
    \item $ \int_{E_{L_2}}\# L_2(x)\lessapprox \int_{\cup_{\cb_{r,6}'}}\# L_6(x)$.
\end{enumerate}
Thus, if denoting $Y_6'(\ell)=Y_6(\ell)\cap\cup_{\cb_{r,6}'}$ for all $\ell\in L_6$, then $(L_6,Y_6')_\de$ is a refinement of $(L_2,Y_2)_\de$. 
Let $(L_7,Y_7)_\de$ be a refinement of $(L_6,Y_6')_\de$ such that $|Y_7(\ell)|\gtrapprox |Y_2(\ell)|$ and $Y_7(\ell)$ is uniform for all $\ell\in L_7$. 
Hence $|Y_7(\ell)|\approx \la$ for all $\ell\in L_7$.

\smallskip

For each $\ell\in L_7$, let $\cb_{r,6}'(\ell)\subset\cb_{r,6}'$ be so that $Y_7(\ell)\cap B\not=\varnothing$ for all $B\in\cb_{r,6}'(\ell)$.
By dyadic pigeonholing, there exists a refinement $\cb_{r,6}''(\ell)$ of $\cb_{r,6}'(\ell)$ so that $|E_{L_6}\cap N_{\De}(\ell)\cap B|$ are the same up to a constant multiple for all $B\in\cb_{r,6}''(\ell)$. 
For each $B\in\cb_{r,6}''(\ell)$, consider the family of finite-overlapping $\De$-balls $\cq_B(\ell)$ that covers $B\cap N_\De(\ell)$.
By pigeonholing on $\{|E_{L_6}\cap N_{\De}(\ell)\cap Q|:Q\in\cq_B(\ell)\}$, there exists a $\kappa_B(\ell)\leq 1$ and a subset $\cq_B'(\ell)\subset\cq_B(\ell)$ such that 
\begin{enumerate}
    \item $\#\cq_B'(\ell)\cdot\kappa_B(\ell)\gtrapprox \De^{-2}|E_{L_6}\cap N_{\De}(\ell)\cap B|\gtrapprox(\de/r)^{1-s+2\eta}(r/\De)$.
    \item For each $Q\in\cq_B'(\ell)$, $|E_{L_6}\cap Q|\geq\kappa_B(\ell)|Q|$.
\end{enumerate}
Pigeonholing on $\{\kappa_B(\ell),\#Q_B(\ell): B\in\cb_{r,6}''(\ell)\}$, there exists two numbers $\kappa(\ell),d(\ell)$, and a refinement $\cb_{r,6}'''(\ell)$ of $\cb_{r,6}''(\ell)$ such that $\kappa_B(\ell)\sim \kappa(\ell)$ and $\#\cq_B'(\ell)\sim d(\ell)$ for all $B\in \cb_{r,6}'''(\ell)$.
By dyadic pigeonholing on $\{\kappa(\ell),d(\ell): \ell\in L_7\}$, there exists a refinement $L_8$ of $L_7$ and two uniform numbers $\kappa,d$ such that $\kappa(\ell)\sim \kappa$ and $d(\ell)\sim d$ for all $\ell\in L_8$.
Now for each $\ell\in L_8$,  define two new shadings
\begin{equation}
\label{caseC-Y-8}
    Y_8(\ell)=Y_7(\ell)\bigcap\cup_{\cb_{r,6}'''(\ell)},\,\,\,\,\,\tilde Y_8(\ell)=\bigcup_{B\in\cb_{r,6}'''(\ell)}\cup_{\cq_B(\ell)}.
\end{equation}
For all $\ell\in L_8$, since $Y_7(\ell)$ is uniform, $|Y_7(\ell)\cap B|$ are about the same for all $B\in\cb_{r,6}'(\ell)$.
Since $\cb_{r,6}'''(\ell)$ is a refinement of $\cb_{r,6}'(\ell)$, $Y_8(\ell)$ is a refinement of $Y_7(\ell)$, so it is also a refinement of $Y_2(\ell)$.
Moreover, the parameters $d,\kappa$ satisfy the following
\begin{enumerate}
    \item For each $\De$-ball $Q\subset E_{L_8, \tilde Y_8}$, we have $|E_{L_6}\cap Q|\geq \kappa|Q|$.
    \item $d\kappa\gtrapprox(\de/r)^{1-s+2\eta}(r/\De)$, implying $d\gtrapprox \kappa^{-1}(\de/r)^{1-s+2\eta}(r/\De)$.
    \item $|\tilde Y_8(\ell)\cap J|_\De\sim d$ for all $J\subset (\tilde Y_8(\ell))_r$.
\end{enumerate}

Since $|Y_7(\ell)|\gtrapprox |Y_2(\ell)|$ and since 
$|Y_7(\ell)\cap B|\lesssim (\de/r)^{1-s-\eta_1}|N_\de(\ell)\cap B|$ for all $\ell\in L_7$ and all $B\in \cb_{r,6}'$  (by item (2) in the outcome of Proposition~\ref{multiscale-prop} in Step 1), we have $\#\cb_{r,6}'(\ell)\gtrsim (\de r)^{-1}|Y_7(\ell)| (r/\de)^{1-s-\eta_1}$. 
Thus, for all $\ell\in L_8$,
\begin{align}
    |\tilde Y_8(\ell)|&\geq \#\cb_{r,6}'''(\ell)\cdot d\De^2\gtrapprox d\De^2(\de r)^{-1}|Y_7(\ell)| (r/\de)^{1-s-\eta_1}\\
    &\gtrapprox \kappa^{-1}(\de/r)^{1-s+2\eta}(r/\De)\cdot \De^2(\de r)^{-1}|Y_7(\ell)| (r/\de)^{1-s-\eta_1}\\
    &\gtrapprox \kappa^{-1}(\de/r)^{2\eta+\eta_1}(\De/\de)|Y_7(\ell)|.
\end{align}
In particular, $\tilde Y_8$ is $\gtrapprox \kappa^{-1}(\de/r)^{2\eta+\eta_1}\la$-dense.

By Lemma \ref{katz-tao-set-lem}, there exists $\tilde L_8\subset L_8$ such that $(\De\# \tilde L_8)\gtrapprox (\de \# L_8)$, and $\tilde L_8$ is a Katz-Tao $(\De,1)$-set.
Since $L_8$ is a refinement of $L_7$, $(\De\# \tilde L_8)\gtrapprox (\de \# L_7)$.
Now we are going to use the assumption $r\leq\de^{\eta}$.
For all $\ell\in L_8$, since $Y_8(\ell)$ is a refinement of $Y_2(\ell)$ and since $Y_2(\ell)$ is a $(\de,\e^2,CC_2;\rho)$-set for some $C_2\lessapprox1$, $Y_8(\ell)$ is a $(\de,\e^2,CC_8;\rho)$-set for some $C_8\lessapprox1$.
Since $Y_7(\ell)$ is uniform, by the definition of $Y_8$ in \eqref{caseC-Y-8}, $(Y_8(\ell))_r$ is a $(r,\e^2,C\tilde C_8;\max\{\rho,r\})$-set for some $\tilde C_8\lessapprox1$.
Finally, since $|\tilde Y_8(\ell)\cap J|_\De\sim d$ for all $J\subset (\tilde Y_8(\ell))_r$ and by \eqref{caseC-Y-8}, 
$\tilde Y_8(\ell)$ is a $(\De,\e^2,C\tilde C_8;\max\{\rho,r\})$-set.
Note that $\max\{\rho,r\}\in[\De,\De^\eta]$, since $\rho\in[\de,\de^\eta]$ and since $r\leq\de^{\eta}$.
Apply induction at scale $\De$ to $(L_8, Y_8)_{\De}$ so that 
\begin{align}
    |E_{\tilde L_8}|&\gtrapprox  c_\e \De^{\e}(C\tilde C_8)^{-\eta^{-2}}\kappa^{-1/2}\la^{1/2}(\de/r)^{(2\eta+\eta_1)/2}\sum_{\ell\in \tilde L_8}|\tilde Y_8(\ell)|\\
    &\gtrapprox c_\e \De^{\e}(C\tilde C_8)^{-\eta^{-2}}\kappa^{-3/2}\la^{1/2}(\de/r)^{(2\eta+ \eta_1)/2}\sum_{\ell\in L_7}|Y_7(\ell)|.
\end{align}
Note that $|E_L|\geq(\sup_{Q\subset E_{\tilde L_8}}|E_{L_6}\cap Q|/|Q|) \cdot |E_{\tilde L_8}|\geq \kappa  |E_{\tilde L_8}|$. Hence
\begin{align}
    |E_L|\gtrapprox \big((\de/r)^{(2\eta+\eta_1)/2}(\De/\de)^{-\e}\kappa^{-1/2}\tilde C_8^{-\eta^{-2}}\big)\cdot c_\e \de^{\e}C^{-\eta^{-2}}\la^{1/2}\sum_{\ell\in L}|Y(\ell)|.
\end{align}
Since $\De\geq r(\de/r)^{1-\sqrt{\eta}}$, we have $\De/\de\geq (r/\de)^{\sqrt{\eta}}$. 
Also, since $\eta_1\leq\eta\leq\e^{3}$ and since $\kappa\leq1$,  $(\de/r)^{(2\eta+\eta_1)/2}(\De/\de)^{-\e}\kappa^{-1/2}\tilde C_8^{-\eta^{-2}}\geq (r/\de)^{\e\sqrt{\eta}-(2\eta+\eta_1)/2}\tilde C_8^{-\eta^{-2}}\geq (r/\de)^{\eta}$. 
Consequently, we obtain $|E_L|\geq c_\e \de^{\e}C^{-\eta^{-2}}\la^{1/2}\sum_{\ell\in L}|Y(\ell)|$, as desired. \qedhere

\end{proof}

\bigskip


\section{Some two-ends Furstenberg inequalities in higher dimensions} \label{section: twoends Kakeyand}

\subsection{Three dimensions: Two-ends hairbrush} We will use the hairbrush structure and Theorem \ref{two-ends-furstenberg} to prove the following result in $\ZR^3$.

\begin{lemma}
\label{two-ends-hairbrush-lem}
Let $\de\in(0,1)$.
Let $(L,Y)_\de$ be a set of $1$-parallel, $\de$-separated lines in $\ZR^3$ with an $(\e_1, \e_2)$-two-ends, $\la$-dense shading. Then for any $\e>0$,
\begin{equation}
    |E_L|\geq c_\e\de^\e \de^{3\e_1/4}\la^{3/4}\de^{1/2} \sum_{\ell\in L} |Y(\ell)|.
\end{equation}
\end{lemma}

\begin{proof}
By dyadic pigeonholing, there is a refinement $(L_0,Y_0)_\de$ of $(L,Y)_\de$ such that $|Y(\ell)|$ are the same up to a constant multiple for all $\ell\in L_0$. 
By Lemma \ref{rich-point-refinement} and Lemma \ref{two-ends-shading-lem}, there exists a dyadic number $\mu$, a set $E^\mu\subset E_L$, and an $(\e_1, \e_2, C_1)$-two-ends (for some $C_1\lessapprox 1$), $\la$-dense refinement $(L_1,Y_1)_\de$ of $(L_0,Y_0)_\de$ such that 
\begin{enumerate}
    \item $\#L_0(x)\sim\mu$ for all $x\in E^\mu$.
    \item $Y_1(\ell)\subset E^\mu$ for all $\ell\in L_1$.
    \item $Y_1(\ell)$ is a refinement of $Y_0(\ell)$ for all $\ell\in L_1$.
\end{enumerate}
Apply Lemma \ref{broad-narrow-lem} to $(L_1, Y_1)_\de$ and by pigeonholing, there is a $\rho\in[\de,1]$, a set $E_\rho$, and a refinement $(L_2, Y_2)_\de$ of $(L_1, Y_1)_\de$ such that
\begin{enumerate}
    \item $L_2=L_1$ and $Y_2(\ell)=Y_1(\ell)\cap E_\rho$ for all $\ell\in L_2$.
    \item For each $x\in E_\rho$, there is a $10\rho$-cap $\si_x$ such that the direction $V(\ell)\in \si_x$ for all $\ell\in L_2$.
    \item $L_2(x)$ is a refinement of $L_1(x)$ for any $x\in E_{L_2}$.
    \item There are two disjoints subsets $L', L''\subset L_2(x)$ of lines such that $\# L', \# L''\gtrsim \# L_2(x)$, and $\rho\geq\text{dist}(\ell',\ell'')\gtrapprox \rho$ for all $\ell'\in L',\ell''\in L''$.
\end{enumerate}

\smallskip

Suppose $\rho\leq \de^{\e/10}$.
Let $\cT$ be a collection of $10\rho$-separated $\rho\times\rho\times 1$-tubes that containing at least one $\ell\in L_2$.
Recall Definition \ref{L[T]-def} for $L_2[T]$.
As a result, similar to Step 3 in the proof of Theorem \ref{two-ends-furstenberg}, we have
\begin{equation}
    |E_{L_2}|\gtrsim \sum_{T\in\cT}|E_{L_2[T]}|\geq c_\e\de^{9\e/10} \de^{3\e_1/4}\la^{3/4}\de^{1/2}\sum_{\ell\in L_2}|Y_2(\ell)|
\end{equation}
by induction. 
Since $(L_2, Y_2)_\de$ is a refinement of $(L, Y)_\de$,
\begin{equation}
    |E_{L}|\geq|E_{L_2}|\geq c_\e\de^\e \de^{3\e_1/4}\la^{3/4}\de^{1/2}\sum_{\ell\in L}|Y(\ell)|.
\end{equation}

\smallskip

Suppose $\rho\geq\de^{\e/10}$.
Since $(L_2, Y_2)_\de$ is a refinement of $(L_0, Y_0)_\de$, 
\begin{equation}
    \int_{E_{L_2}}\#L_2(x)\gtrapprox\int_{E_{L_0}}\#L_0(x)\geq\int_{E^\mu}\#L_0(x).
\end{equation}
In addition, since $E_{L_2}\subset E^\mu$, by Lemma \ref{rich-point-refinement}, we can find a set $E'\subset E_{L_2}$ so that $\# L_2(x)\gtrapprox\mu$ for all $x\in E'$, and $(L_3, Y_3)_\de$ is a $\gtrapprox \la $-dense refinement of $(L_2, Y_2)_\de$, where $L_3=L_2$ and $Y_3(\ell)=Y_2(\ell)\cap E'$ for all $\ell\in L_3$.
By pigeonholing, there is a line $\ell\in L_3$ with a $\gtrapprox \la $-dense shading.
Let
\begin{equation}
    \ch(\ell)=\{\ell'\in L_1: Y_1(\ell')\cap Y_3(\ell)\not=\varnothing, \ang(\ell',\ell)\gtrapprox\de^{\e/10}\}
\end{equation}
be the hairbrush of $\ell$. 
Since $Y_3(\ell)\subset E'\cap E_\rho$, $\#\ch(\ell)\gtrapprox \mu\la\de^{-1}$.

Let $H=\bigcup_{\ell'\in\ch(\ell)} Y_1(\ell')\setminus N_{\de^{\e_1+\e/10}(\ell)}$ and let $Y_4(\ell')=Y_1(\ell')\setminus N_{\de^{\e_1+\e/10}(\ell)} \subset H$ for all $\ell'\in \ch(\ell)$, so $|Y_4(\ell')|\gtrsim|Y_1(\ell')|$.
Since $Y_1(\ell)$ is $(\e_1, \e_2, C_1)$-two-ends and $\gtrapprox \la $-dense, $(\ch(\ell), Y_4)_\de$ is $(\e_1, \e_2, C_4)$-two-ends (for some $C_4\lessapprox 1$) and $\gtrapprox \la $-dense.
Note that there is a collection of $\de\times1\times1$-slabs $\{P\}$ so that $\ell\cap B^3(0,1)\subset P$ and $P\setminus N_{\de^{\e_1+\e/10}(\ell)}$ forms a $O(\de^{-(\e_1+\e/10)})$-overlapping covering of $H$. 
Moreover, for each $\ell'\in\ch(\ell)$, $\ell'\cap B^3(0,1)$ belongs to $\lesssim 1$ slabs in $\{P\}$. 
Let $ P(\ell)=\{\ell'\in \ch(\ell):  \ell'\cap B^3(0,1)\subset P\}$. 

Apply Theorem \ref{two-ends-furstenberg} with $\e/10$ in the place of $\e$ so that ($C_4\lessapprox1$ is negligible)
\begin{equation}
    \big|\bigcup_{\ell'\in P(\ell)}Y_4(\ell')\big|\gtrapprox \delta^{\e/10} \de^{\e_1/2}\la^{1/2}\sum_{\ell'\in P(\ell)}|Y_4(\ell')|\gtrapprox \delta^{\e/10}\de^{\e_1/2} \la^{1/2}\sum_{\ell'\in P(\ell)}|Y_1(\ell')|.
\end{equation}
Summing up all $P$ and noting that $\{P\}$ are $O(\de^{-(\e_1+\e/10)})$-overlapping, we have
\begin{equation}
    |E_L|\geq \de^{\e_1+\e/10}\sum_{P}\big|\bigcup_{\ell'\in P(\ell)}Y_4(\ell')\big|\gtrapprox \de^{3\e_1/2+\e/5}\la^{1/2}\sum_{\ell'\in \ch(\ell)}|Y_1(\ell')|. 
\end{equation}
Recall that $|Y_0(\ell)|$ are  the same up to a constant multiple for all $\ell\in L_0$, and $Y_1(\ell)$ is a refinement of $Y_0(\ell)$ for all $\ell\in L_1$.
Since $\# L_1\lesssim\de^{-2}$, the above gives
\begin{equation}
    |E_L|\gtrapprox \de^{\e/5} \de^{3\e_1/2}\la^{1/2}\frac{\#\ch(\ell)}{\# L_1}\sum_{\ell\in L_1}|Y_1(\ell)|\gtrapprox\de^{\e/5} \de^{3\e_1/2}\la^{1/2}(\mu\la\de)\sum_{\ell\in L}|Y(\ell)|.
\end{equation}

On the other hand,
\begin{equation}
    |E_L|\geq |E_{L_1}|\gtrapprox \mu^{-1}\sum_{\ell\in L}|Y(\ell)|.
\end{equation}
These two estimates give
\begin{equation}
    |E_L|\gtrapprox \de^{\e/10} \de^{3\e_1/4}\la^{3/4}\de^{1/2} \sum_{\ell\in L} |Y(\ell)|\geq c_\e\de^\e \de^{3\e_1/4}\la^{3/4}\de^{1/2} \sum_{\ell\in L} |Y(\ell)|. \qedhere
\end{equation}
\end{proof}

\smallskip

As a corollary, we have

\begin{proposition}
\label{kakeya-prop}
Let $\de\in(0,1)$.
Let $(L,Y)_\de$ be a set of $m$-parallel, $\de$-separated lines with an $(\e_1, \e_2)$-two-ends, $\la$-dense shading. Take $\mu=\de^{-2\e_1}m\la^{-3/4}\de^{-1/2}$. Then there exists a set $E_\mu\subset E_L$ such that $\# L(x)\lessapprox \mu$ for all $x\in E_\mu$, and
\begin{equation}
    |E_L\setminus E_\mu|\leq \de^{\e_1}|E_L|.
\end{equation}
\end{proposition}
\begin{proof}
Let $\rho\in[\mu,\de^{-2}]$ be a dyadic number, and let $E_\rho=\{x\in E_L:\# L(x)\sim\rho\}$. 
Clearly $\rho|E_\rho|\leq\sum_{\ell\in L} |Y(\ell)|$. 
Let $E_\mu'=\bigcup_{\rho\gtrapprox\mu}E_\rho$ and let $E_\mu=E_L\setminus E_\mu'$. Then
\begin{equation}
    |E_\mu'|=\sum_{\rho\gtrapprox\mu}\rho^{-1} \sum_{\ell\in L} |Y(\ell)|\lessapprox\mu^{-1} \sum_{\ell\in L} |Y(\ell)|\leq \de^{2\e_1}m^{-1}\la^{3/4}\de^{1/2} \sum_{\ell\in L} |Y(\ell)|. 
\end{equation}

On the other hand, let $L'\subset L$ be a maximal directional $\de$-separated lines such that $|Y(\ell)|\geq |Y(\ell')|$ for all $\ell\in L'$ and $\ell'\in L$ with $\ang(\ell, \ell')\leq\de$. 
As a result, $\sum_{\ell\in L'}|Y(\ell)|\gtrsim m^{-1}\sum_{\ell\in L}|Y(\ell)|$. 
Apply Lemma \ref{two-ends-hairbrush-lem} with $\e=\e_1/8$ to $(L',Y)_\de$ to get
\begin{equation}
    |E_L|\geq |E_{L'}|\gtrapprox \de^{7\e_1/8}\la^{3/4}\de^{1/2} \sum_{\ell\in L'} |Y(\ell)|\gtrapprox\de^{7\e_1/8}m^{-1}\la^{3/4}\de^{1/2} \sum_{\ell\in L} |Y(\ell)|.
\end{equation}
This shows $|E_L\setminus E_\mu|\leq \de^{\e_1}|E_L|$. 
\end{proof}

\medskip

\subsection{Higher dimensions.}

We first prove a two-ends brush estimate.
\begin{lemma}
\label{bush-lem-high-d}
Let $\de\in(0,1)$.
Let $(L,Y)_\de$ be a set of $\de$-separated lines in $\ZR^n$ with an $(\e_1, \e_2)$-two-ends, $\la$-dense shading. Then $|E_L|\gtrapprox \de^{\e_1/2}\la\de^{\frac{n-1}{2}}(\de^{n-1}\#L)^{1/2}$.
\end{lemma}
\begin{proof}
By Lemma \ref{rich-point-refinement}, there exists a dyadic number $\mu$, a set $E^\mu\subset E_L$, and a refinement $(L',Y')_\de$ of $(L,Y)_\de$ such that $\#L(x)\sim\mu$ for all $x\in E^\mu$, and $Y'(\ell)\subset E^\mu$ for all $\ell\in L'$.
Therefore, on the one hand,
\begin{equation}
    |E_L|\geq|E^\mu|\gtrsim \mu^{-1}\sum_{\ell\in L'}|Y'(\ell)|\gtrapprox \mu^{-1}\sum_{\ell\in L}|Y(\ell)|\geq \mu^{-1}\la(\de^{n-1}\# L).
\end{equation}
On the other hand, consider a single bush rooted at a point $x\in E^\mu$. 
Since the shading $Y(\ell)$ is $(\e_1,\e_2)-$two-ends and $\la$-dense, 
\begin{equation}
    |E_L|\geq\big|\bigcup_{\ell\in L(x)}Y(\ell)\big|\gtrsim \de^{\e_1} \mu\la\de^{n-1}.
\end{equation}
The two estimates together give
\begin{equation}
    |E_L|\gtrapprox \de^{\e_1/2}\la\de^\frac{n-1}{2}(\de^{n-1}\#L)^{1/2}. \qedhere
\end{equation}
\end{proof}

Note that $\de^{n-1}\#L\lesssim m$ when $L$ is $m$-parallel. 
Similar to Proposition \ref{kakeya-prop}, we get

\begin{corollary}
\label{cor-3}
Let $\de\in(0,1)$.
Let $(L,Y)_\de$ be a set of $m$-parallel, $\de$-separated lines in $\ZR^n$ with an $(\e_1, \e_2)$-two-ends, $\la$-dense shading. Take $\mu=\de^{-2\e_1}m\de^{-\frac{n-1}{2}}$.
Then there exists a set $E_\mu\subset E_L$ such that $\# L(x)\lessapprox \mu$ for all $x\in E_\mu$, and
\begin{equation}
    |E_L\setminus E_\mu|\leq \de^{3\e_1/2}|E_L|.
\end{equation}
\end{corollary}

\smallskip

Lemma \ref{bush-lem-high-d} is useful when the density $\la$ is small. 
When $\la$ is large, we need a two-ends inequality from \cite{Katz-Tao-Kakeya-maximal}.
\begin{theorem}[c.f. \cite{Katz-Tao-Kakeya-maximal}, Page 18]
\label{katz-tao-thm}
Let $\de\in(0,1)$.
Let $(L,Y)_\de$ be a set of $1$-parallel, $\de$-separated lines in $\ZR^n$ with an $(\e_1, \e_2)$-two-ends, $\la$-dense shading. 
Suppose $\la\geq \de^{1/4}$.
Then
\begin{equation}
\label{katz-tao-esti}
    |E_L|\gtrapprox \de^{\e_1} \la^\frac{2n+10}{7}\de^{\frac{3n-3}{7}}\Big(\sum_{\ell\in L}|Y(\ell)|\Big)^{4/7}.
\end{equation}
\end{theorem}
\begin{remark}
\rm
\eqref{katz-tao-esti} was stated in \cite{Katz-Tao-Kakeya-maximal} under the stronger assumption $\la\geq \de^{1/8}$.
However, the exact same proof indeed gives \eqref{katz-tao-esti} under a weaker assumption $\la\gtrapprox\de^{1/2}$. 
In Section \ref{sketch-KT}, we will quickly go through the main idea of the proof in \cite{Katz-Tao-Kakeya-maximal} and focus mostly on the steps that require a lower bound on $\la$.
We also remark that we have not attempted to optimize the dependence on $\e_1$ in \eqref{katz-tao-esti}
\end{remark}

Similarly to Corollary \ref{cor-3}, we have
\begin{corollary}
\label{cor-4}
Let $\de\in(0,1)$.
Let $(L,Y)_\de$ be a set of $m$-parallel, $\de$-separated lines in $\ZR^n$ with an $(\e_1, \e_2)$-two-ends, $\la$-dense shading. 
Take $\mu=\de^{-2\e_1}m\la^{-\frac{2n+7}{7}}\de^{-\frac{3n-3}{7}}$. 
Suppose $\la\geq \de^{1/4}$. 
Then there exists a set  $E_\mu\subset E_L$ such that $\# L(x)\lessapprox \mu$ for all $x\in E_\mu$, and
\begin{equation}
    |E_L\setminus E_\mu|\leq \de^{\e_1}|E_L|.
\end{equation}
\end{corollary}

\smallskip

Putting Corollary \ref{cor-3} and \ref{cor-4} together, we have
\begin{proposition}
\label{kakeya-prop-high-d}
Let $\de\in(0,1)$.
Let $(L,Y)_\de$ be a set of $m$-parallel, $\de$-separated lines in $\ZR^n$ with an $(\e_1, \e_2)$-two-ends, $\la$-dense shading. 
The following is true:
\begin{enumerate}
    \item If $\la\geq \de^{1/4}$. Take $\mu=\de^{-2\e_1}m\la^{-\frac{2n+7}{7}}\de^{-\frac{3n-3}{7}}$.
    Then there exists a set $E_\mu\subset E_L$ such that $\# L(x)\lessapprox \mu$ for all $x\in E_\mu$, and
    \begin{equation}
        |E_L\setminus E_\mu|\leq \de^{\e_1}|E_L|.
    \end{equation}
    \item If $\la\leq \de^{1/4}$. Take $\mu=\de^{-2\e_1}m\de^{-\frac{n-1}{2}}$. 
    Then there exists a set $E_\mu\subset E_L$ such that $\# L(x)\lessapprox \mu$ for all $x\in E_\mu$, and
    \begin{equation}
        |E_L\setminus E_\mu|\leq \de^{\e_1}|E_L|.
    \end{equation}
\end{enumerate}

\end{proposition}

\bigskip


\section{Preliminaries in Fourier analysis}\label{section: fourier analysis}

By an epsilon removal argument by Tao \cite{Tao-BR-restriction} (see also \cite{Bourgain-Besicovitch}), Conjecture \ref{restriction-conj} is a consequence of  the following local version:

\begin{conjecture}
\label{restriction-conj-local-1}
Suppose that $S\subset\ZR^n$ is a compact $C^2$ hypersurface (maybe with boundary) with a strictly positive second fundamental form. 
Then when $p\geq\frac{2n}{n-1}$, for any $\e>0$ and $R>1$,
\begin{equation}
    \|E_Sf\|_{L^p(B_R)}\leq C_\e R^\e \|f\|_{L^p(d\si_S)}.
\end{equation}
\end{conjecture}

By parabolic rescaling (see, for example, \cite{Guth-R3}), we can assume all the principle curvatures of $S$ are $\sim1$.
By dividing $S$ into pieces and affine transformations, we can also assume that $S$ is the graph of a function $\Phi:B^{n-1}(0,1)\to\ZR$ obeying that $\Phi(0)=0$, $\nabla\Phi(0)=0$, and the eigenvalues of $\nabla^2\Phi$ are $\sim1$.

\smallskip

Next, we construct a (standard) wave packet decomposition of $f$ and hence $E_Sf$.
Let $x=(\bar x,x_n)\in\ZR^n$, where $\bar x$ is the first $n-1$ entries $x$. 
In order to simplify the construction and further calculations, we consider another extension operator (associated with $\Phi$ and hence also $S$)
\begin{equation}
\label{simplified-extension-operator}
    E f(x):=\int_{B^{n-1}(0,1)}e^{i\bar x\cdot \bar\xi}e^{ix_n\Phi(\bar\xi)}f(\xi)d\xi.
\end{equation}
By expressing the surface measure $d\si_S$ as a function of $\bar x$, it is straightforward to check that Conjecture \ref{restriction-conj-local-1} is equivalent to the following.
\begin{conjecture}
\label{restriction-conj-local-2}
Suppose $E$ is the extension operator defined in \eqref{simplified-extension-operator}. 
Then when $p\geq\frac{2n}{n-1}$, for any $\e>0$ and $R>1$,
\begin{equation}
\label{reduced-restriction-esti}
    \|E f\|_{L^p(B_R)}\leq C_\e R^\e \|f\|_{p}.
\end{equation}
\end{conjecture}

Let $C=1000n$. For the $\e$ in Conjecture \ref{restriction-conj-local-2}, we fix a small constant $\e_0=\e^{C}$. 
In the frequency space, let $\Theta$ be a finite-overlapping partition of $B^{n-1}(0,1)$ by $R^{-1/2}$-balls, and let $\{\vp_\theta\}_{\theta\in \Theta}$ be a smooth partition of unity so that $\supp(\vp)\subset 2\theta$ and $\sum_{\theta\in\Theta} \vp_\theta=1$ on $B^{n-1}(0,1)$. 
Therefore, $|\widecheck\vp_\theta(x)|\lesssim R^{-C}$ when $x\not \in B^{n-1}(0,R^{1/2+\e_0})$.

In the physical space, let $\cv$ be a finite-overlapping partition of $\ZR^{n-1}$ by $R^{1/2}$-balls, and let $\{\psi_v\}_{v\in \cv}$ be a smooth partition of unity of $\ZR^{n-1}$ so that $\supp(\wh\psi_v)\subset B^{n-1}(0,R^{-1/2})$ and $\sum_{v\in\cv}\psi_v=1$ in $\ZR^{n-1}$. 
Therefore, $|\psi_v|\lesssim R^{-C}$ when $x\not\in R^{\e_0}v$ (here $v$ is an $R^{1/2}$-ball).

\smallskip

The above frequency-space partition gives the wave packet decomposition for any function $f$ supported on $B^{n-1}(0,1)$:
\begin{equation}
\label{wpt-decomposition}
    f=\sum_{\theta\in \Theta}\sum_{v\in\cv}(f\vp_\theta)\ast\hat\psi_v=:\sum_{(\theta,v)\in\Theta\times\cv}f_{\theta,v}.
\end{equation}
For each $\theta\in \Theta$ and each $v\in\cv$, let $T_{\theta,v}=\{(\bar x,x_n)\in B^n(0,R):|\bar x-c_v+x_n\nabla\Phi(c_\theta)|\leq R^{1/2+\e_0} \}$ be a tube of dimensions $R^{1/2+\e_0}\times\cdots\times R^{1/2+\e_0}\times R$, where $c_\theta,c_v$ are the center of $\theta,v$ respectively. 
Denote by $V(\theta)$ the vector $(1, \nabla\Phi(c_\theta))$.
Let $\bar\ZT(\theta)=\{T_{\theta,v}:v\in\cv\text{ and }T_{\theta,v}\cap B_R\not=\varnothing\}$ be a family of $R$-tubes with direction $V(\theta)$, and let $\bar\ZT=\bigcup_\theta\bar\ZT(\theta)$. 
We denote by $f_T=f_{\theta,v}$ for an $R$-tube $T\in\bar\ZT$ if $T=T_{\theta,v}$.

\begin{lemma}
\label{wpt}
The wave packet decomposition satisfies the following properties.
\begin{enumerate}
    \item $|Ef_{T}(x)|\lesssim R^{-C}$ when $x\not\in T$.
    \item $\supp f_{T}\subset 3\theta$ when $T$ has direction $V(\theta)$.
    \item $|Ef-\sum_{T\in\bar\ZT}Ef_{T}|\leq R^{-C}$ on $B_R$.
    \item $\{V(\theta)\}_{\theta\in\Theta}$ are $\gtrsim R^{-1/2}$-separated.
    \item $\ZT_\theta$ is $R^{O(\e_0)}$-overlappling.
    \item $\|Ef_T\|_{L^p(w_{B_R})}\lesssim R^{(\frac{1}{p}-\frac{1}{2})\frac{n+1}{2}} \|Ef_T\|_{L^2(w_{B_R})}$ for all $T\in\bar\ZT $, where $w_{B_R}$ is a weight that is $\sim1$ on $B_R$ and decreases rapidly outside $B_R$. 
\end{enumerate}    
\end{lemma}
\begin{proof}
The first five items are apparent from the construction.
The sixth item follows from Young's convolution inequality since $\supp (\wh{Ef_Tw_{B_R}})$ is contained in a $100n$-dilate of an $R^{-1}\times R^{-1/2}\times\cdots\times R^{-1/2}$-cap.
\end{proof}

\smallskip

With the wave packet decomposition, we state the refined decoupling theorem.

\begin{theorem}[\cite{GIOW} Theorem 4.2]
\label{refined-decoupling-thm}
Let $E$ be the extension operator \eqref{simplified-extension-operator}, and let $p=\frac{2(n+1)}{n-1}$. Suppose $f$ is a sum of wave packets $f=\sum_{T\in\ZT}f_T$ so that $\|Ef_T\|_{L^p(w_{B_R})}^2$ are the same up to a constant multiple for all $T\in\ZT$. 
Let $X$ be a union of $R^{1/2}$-balls in $B_R$ such that each $R^{1/2}$-ball $Q\subset X$ intersects to at most $M$ tubes from $\ZT$. 
Then 
\begin{equation}
\label{refined-decoupling}
    \|Ef\|_{L^p(X)}^p\lessapprox R^{\e_0} M^{\frac{2}{n-1}}\sum_{T\in\ZT}\|Ef_T\|_{L^p(w_{B_R})}^p.
\end{equation}
Here $w_{B_R}$ is a weight that is $\sim1$ on $B_R$ and decreases rapidly outside $B_R$. 
\end{theorem}

\smallskip

There is a general statement of the refined decoupling theorem for all $2\leq p\leq 2(n+1)/(n-1)$.
However, for $p$ smaller than the critical exponent $2(n+1)/(n-1)$, the corresponding decoupling inequality is usually inefficient, since it is derived by a simple interpolation between $p=2$ and $p=2(n+1)/(n-1)$.
In many cases, after several steps of reductions, extra information can be found in the $L^2$ space, where stronger orthogonality occurs. 
What follows is an $L^2$ inequality regarding a special scenario for the restriction operator.

\begin{lemma}
\label{lem: l2}
Let $X=\cup_{\cq}$ be a union of $R^{1/2}$-balls, and let $f=\sum_{T\in\ZT}f_T$ be a sum of wave packets.
Suppose for each $T\in \ZT$, there is a shading $Y(T)\subset T$ by $R^{1/2}$-balls in $\cq$ such that the number of $R^{1/2}$-balls contained in $Y(T)$ is $\lesssim \lambda R^{1/2}$.
Then
\begin{equation}
\label{l2}
    \int_{X}\big|\sum_{T\in\ZT}Ef_{T}\Id_{Y(T)}\big|^2\lesssim (\la R)\|f\|_2^2.
\end{equation}
\end{lemma}
\begin{proof}

For each $R^{1/2}$-ball $Q\subset \cQ$, let $\ZT(Q)=\{ T\in\ZT:   Q\cap  Y(T)\neq \varnothing \}$. 
Then by $L^2$-orthogonality on $Q$, 
\begin{align}
    \int_Q \big|\sum_{T\in\ZT}Ef_{T}\Id_{Y(T)}\big|^2 & \leq \int_{2Q} \big|\sum_{T\in \ZT(Q)}Ef_{T}\big|^2 +R^{-C}\|f\|_{2}^2\\
    & \lesssim \int \sum_{T\in\ZT(Q)}\big| Ef_{T} \, \phi_{2Q}\big|^2 +R^{-C}\|f\|_{2}^2  \\
    &\lesssim  R^{1/2} \sum_{T\in \ZT(Q)} \|f_T\|_{2}^2+R^{-C}\|f\|_{2}^2 . 
\end{align}
In the first inequality, we enlarge $Q$ to $2Q$ to make sure that $Y(T) \cap 2Q$ contains an $R^{1/2}$-ball.
$\phi_{2Q}$ is a bump function of $2Q$ so that $\{ \widehat{Ef_T } \ast \widehat{\phi}_{2Q}: T\in \ZT(Q)\}$ have finite-overlapping support. 

Since each $T$ appears in $\lesssim \lambda R^{1/2}$ many sets $\ZT(Q)$, sum over all $Q\in \cQ$ to have
\begin{equation}
    \int_{X}\big|\sum_{T\in\ZT}Ef_{T}\Id_{Y(T)}\big|^2 \lesssim R^{1/2} \sum_{Q\in \cQ} \sum_{T\in \ZT(Q)} \|f_T\|_{2}^2+R^{-C}\|f\|_{2}^2 \lesssim (\lambda R) \|f\|_{2}^2. \qedhere
\end{equation}
\end{proof}

\begin{remark}

\rm

We will only apply Lemma \ref{lem: l2} at the single scale $R^{1/2}$ later in Section \ref{section: restriction estimates}. 
It is intriguing to explore whether an iterative application on Lemma \ref{lem: l2} could lead to further insights into the restriction conjecture in higher dimensions.
In two dimensions, Lemma \ref{lem: l2} was applied iteratively in \cite{Li-Wu} to reprove the reverse square function estimate.
\end{remark}

\smallskip

Finally, we reduce the $L^p$-norm on the right-hand side of \eqref{reduced-restriction-esti} to a mixed norm, which is more suitable for induction. 
This mixed norm was first established in \cite{Guth-R3}.
For each $\theta\in\Theta$, we define 
\begin{equation}
    \|f_\theta\|_{L^2_{ave}(\theta)}^2:=|\theta|^{-1}\|f_\theta\|_2^2.
\end{equation}
In the next section, we will prove
\begin{equation}
\label{guth-trick-pre}
    \|E f\|_{L^p(B_R)}^p\leq C_\e R^\e \|f\|_2^2\sup_{\theta\in\Theta}\|f_\theta\|_{L^2_{ave}(\theta)}^{p-2}.
\end{equation}
for the desired exponent $p=p_0$, where $p_0=22/7$ when $n=3$ and $p_0=p(n)$ for $n\geq4$ (see \eqref{p-n}). 
The next lemma reduces \eqref{reduced-restriction-esti} to \eqref{guth-trick-pre}.

\begin{lemma}
\label{real-interpolation}
If \eqref{guth-trick-pre} is true for $p=p_0$ for any $f\in L^p$, then Conjecture \ref{restriction-conj-local-2} is true when $p>p_0$.
\end{lemma}
\begin{proof}
This basically follows from a real interpolation for functions of restrict-type.

Since $\|f_\theta\|_{L^2_{ave}}\lesssim\|f\|_\infty$, \eqref{reduced-restriction-esti} is a consequence of \eqref{guth-trick-pre} when $f$ is the characteristic function of a measurable set. 
In particular, \eqref{guth-trick-pre} implies the restrict-type estimate for \eqref{reduced-restriction-esti} when $p=p_0$. 
Therefore, the range $p>p_0$ for \eqref{reduced-restriction-esti} follows from  a real interpolation between the restrict-type estimate when $p=p_0$ and the trivial bound $\|Ef\|_\infty \lesssim\|f\|_\infty$.
\end{proof}

\bigskip


\section{The restriction estimates} \label{section: restriction estimates}

By Lemma \ref{real-interpolation}, Theorem \ref{main-3d} and \ref{main-high-d} are consequences of the following result.
\begin{theorem}
\label{main-thm-reduction}
Suppose $E$ is the extension operator defined in \eqref{simplified-extension-operator}. 
Then 
\begin{equation}
\label{guth-trick}
    \|E f\|_{L^p(B_R)}^p\leq C_\e R^\e \|f\|_2^2\sup_{\theta\in\Theta}\|f_\theta\|_{L^2_{ave}(\theta)}^{p-2}
\end{equation}
for $p=p_0$, where $p_0=22/7$ when $n=3$ and $p_0=p(n)$ for $n\geq4$ (see \eqref{p(n)}). 
\end{theorem}

\smallskip

Recall the wave packet decomposition \eqref{wpt-decomposition} and the notation $f_T=f_{\theta,v}$ if $T=T_{\theta,v}$ for all $T\in\bar\ZT$.
Thus,
\begin{equation}
    \|Ef\|_{L^p(B_R)}\lesssim\big\|\sum_{T\in\bar\ZT}Ef_T\big\|_{L^p(B_R)}.
\end{equation}
Denote by $p_n=\frac{2(n+1)}{n-1}$. 
We first reduce \eqref{guth-trick} to the scenario when the magnitudes of the wave packets are about the same.

\smallskip

For each dyadic number $\bar\be$, let $\bar\ZT_{\bar\be}=\{T\in\bar\ZT:\|Ef_T\|_{L^{p_n}(w_{B_R})}\sim \bar\be\|f\|_2\}$. 
By H\"older's inequality and Plancherel, we have $\|Ef_T\|_{L^p(w_{B_R})}\leq R^{n}\|Ef_T\|_{L^2(w_{B_R})}\leq R^{2n}\|f_T\|_2\lesssim R^{2n}\|f\|_2$. 
This shows
$\bar\ZT_{\bar\be}=\varnothing$ when $\bar\be\geq R^{10n}$.
Partition $\bar\ZT=\bar\ZT_s\sqcup\bar\ZT_l$, where $\bar\ZT_s=\bigcup_{\bar\be\leq R^{-10n}}\bar\ZT_{\bar\be}$ and $\bar\ZT_l=\bigcup_{\bar\be> R^{-10n}}\bar\ZT_{\bar\be}$. 
By the triangle inequality, 
\begin{equation}
    \|Ef\|_{L^p(B_R)}\lesssim\big\|\sum_{T\in\bar\ZT_s}Ef_T\big\|_{L^p(B_R)}+\big\|\sum_{T\in\bar\ZT_l}Ef_T\big\|_{L^p(B_R)}.
\end{equation}
If the first term dominates the left-hand side, then by the triangle inequality, 
\begin{equation}
    \|Ef\|_{L^p(B_R)}\lesssim\big\|\sum_{T\in\bar\ZT_s}Ef_T\big\|_{L^p(B_R)}\leq\sum_{T\in\bar\ZT_s}\big\|Ef_T\big\|_{L^p(B_R)}\leq R^{-10n}(\#\bar\ZT_s)\|f\|_2.
\end{equation}
Since $\#\bar\ZT_s\lesssim R^{2n}$, we have $\|Ef\|_{L^p(B_R)}\lesssim R^{-8n}\|f\|_2\lesssim \|f\|_p$, which proves \eqref{guth-trick}.

\smallskip

Now let us suppose $\|Ef\|_{L^p(B_R)}\lesssim\big\|\sum_{T\in\bar\ZT_l}Ef_T\big\|_{L^p(B_R)}$. 
By pigeonholing, there exists a $\bar\be\in(R^{-10n},R^{10n})$ such that
\begin{equation}
    \|Ef\|_{L^p(B_R)}\lessapprox\big\|\sum_{T\in\bar\ZT_{\bar\be}}Ef_T\big\|_{L^p(B_R)}.
\end{equation}
For each $\theta\in\Theta$, let $\bar\ZT_{\bar\be}(\theta)=\bar\ZT_{\bar\be}\cap\bar\ZT(\theta)$. 
By dyadic pigeonholing, there exists a dyadic $m\geq1$ and a subset $\Theta_m\subset\Theta$ such that $\#\bar\ZT_{\bar\be}(\theta)\sim m$ for all $\theta\in\Theta_m$, and, if denoting by $\ZT=\bigcup_{\theta\in\Theta_m}\bar\ZT_{\bar\be}(\theta)$, we have
\begin{equation}
\label{before-two-ends}
    \|Ef\|_{L^p(B_R)}\lessapprox\big\|\sum_{T\in\ZT}Ef_T\big\|_{L^p(B_R)}.
\end{equation}
For simplicity, we denote $\ZT(\theta)=\bar\ZT_{\bar\be}(\theta)$ and $\Theta=\Theta_m$. 
Thus, the set of tubes $\ZT=\cup_{\theta\in\Theta}\ZT_\theta$ satisfies the following properties:
\begin{enumerate}
    \item $\#\ZT(\theta)\sim m$ for all $\theta\in\Theta$.
    \item $\|Ef_T\|_{L^{p_n}(w_{B_R})}$ are  the same up to a constant multiple for all $T\in\ZT$.
\end{enumerate}

\smallskip

By dyadic pigeonholing, there is a family of disjoint $R^{1/2}$-balls $\cq$ such that $\|\sum_{T\in\ZT}Ef_T\|_{L^p(Q)}$ are the same up to a constant multiple, and, if denoting $X=\cup_\cq$,
\begin{equation}
\label{uniform-on-R1/2-ball}
    \big\|\sum_{T\in\ZT}Ef_T\big\|_{L^p(B_R)}\lessapprox \big\|\sum_{T\in\ZT}Ef_T\big\|_{L^p(X)}.
\end{equation}
From now on, we focus on $\ZT$, $\cq$, and $X=\cup_\cq$.

\medskip

\subsection{A two-ends reduction}

For each $T\in\ZT$, we first partition $T$ into sub-tubes $\cj(T)=\{J\}$ of length $R^{1-\e^2}$. 
Then, partition the set $\cj(T)=\bigcup_\la\cj_\la(T)$, where $\la\leq 1$ is a dyadic number and $|J\cap X|\sim \la R$ for any $J\in\cj_\la(T)$. 
Hence
\begin{equation}
    \sum_{T\in\ZT}Ef_T=\sum_{\la}\sum_{T\in\ZT}\sum_{J\in\cj_\la(T)}Ef_{T}\Id_J.    
\end{equation}
For each $Q\in\cq$, by pigeonholing, there is a $\la(Q)$ such that 
\begin{equation}
    \big\|\sum_{T\in\ZT}Ef_T\big\|_{L^p(Q)}^p\lessapprox\int_Q\big|\sum_{T\in\ZT}\sum_{J\in\cj_{\la(Q)}(T)}Ef_{T}\Id_J\big|^p.
\end{equation}
Recall that $\|\sum_{T\in\ZT}Ef_T\|_{L^p(Q)}$ are  the same up to a constant multiple for all $Q\in\cq$. 
By dyadic pigeonholing on $\{(\la(Q), \|\sum_{T\in\ZT}\sum_{J\in\cj_{\la(Q)}(T)}Ef_{T}\Id_J\|_{L^p(Q)}):Q\in\cq\}$, there is a uniform $\la$ and a refinement $\cq_1$ of $\cq$ such that for all $Q\in\cq_1$, $\la(Q)=\la$, and
$\|\sum_{T\in\ZT}\sum_{J\in\cj_\la(T)}Ef_{T}\Id_J\|_{L^p(Q)}$ are  the same up to a constant multiple.
Denote by $X_1=\cup_{\cq_1}$.
Thus, 
\begin{equation}
\label{first-refinement}
    \big\|\sum_{T\in\ZT}Ef_T\big\|_{L^p(X)}^p\lessapprox\big\|\sum_{T\in\ZT}Ef_T\big\|_{L^p(X_1)}^p\lessapprox\int_{X_1}\big|\sum_{T\in\ZT}\sum_{J\in\cj_\la(T)}Ef_{T}\Id_J\big|^p.
\end{equation}
Fix this $\la$ from now on.

\smallskip

Consider the partition $\ZT=\bigcup_\be\ZT_\be$, where $\be\in[1,R^{\e^2}]$ is a dyadic number and $\#\cj_\la(T)\sim\be$ for all $T\in\ZT_\be$. 
As a result, 
\begin{equation}
    \sum_{T\in\ZT}\sum_{J\in\cj_\la(T)}Ef_{T}\Id_J=\sum_\be\sum_{T\in\ZT_\be}\sum_{J\in\cj_\la(T)}Ef_{T}\Id_J.    
\end{equation}
For each $Q\in\cq_1$, by pigeonholing, there is a $\be(Q)$ such that 
\begin{equation}
    \int_{Q}\big|\sum_{T\in\ZT}\sum_{J\in\cj_\la(T)}Ef_{T}\Id_J\big|^p\lessapprox\int_{Q}\big|\sum_{T\in\ZT_{\be(Q)}}\sum_{J\in\cj_\la(T)}Ef_{T}\Id_J\big|^p.
\end{equation}
Since $\|\sum_{T\in\ZT}\sum_{J\in\cj_{\la(Q)}(T)}Ef_{T}\Id_J\|_{L^p(Q)}$ are the same up to a constant multiple for all $Q\in\cq_q$, by similar dyadic pigeonholing as before, there exists a uniform $\be$ and a refinement $\cq_2\subset\cq_1$ such that the following is true for all $Q\in\cq_2$:
\begin{enumerate}
    \item $\be(Q)=\be$.
    \item $\|\sum_{T\in\ZT_\be}\sum_{J\in\cj_{\la}(T)}Ef_{T}\Id_J\|_{L^p(Q)}$ are the same up to a constant multiple.
\end{enumerate}
Moreover, by \eqref{first-refinement} and \eqref{uniform-on-R1/2-ball}, we have
\begin{equation}
\label{two-ends-reduction-1}
    \big\|\sum_{T\in\ZT}Ef_T\big\|_{L^p(B_R)}^p\lessapprox\int_{X_2}\big|\sum_{T\in\ZT_\be}\sum_{J\in\cj_\la(T)}Ef_{T}\Id_J\big|^p.
\end{equation}
We also remark that $|X_2|\gtrapprox|X|$, as $\cq_2$ is a refinement of $\cq$.

\medskip

\subsection{The non-two-ends scenario} Suppose $\be\leq R^{\e^{4}}$. 
Let $\{B_k\}$ be a family of $R^{1-\e^2}$-balls that covers $B_R$. 
Then,
\begin{equation}
\label{related}
    \int_{X_2}\big|\sum_{T\in\ZT_\be}\sum_{J\in\cj_\la(T)}Ef_{T}\Id_J\big|^p\lesssim \sum_{k}\int_{X_2\cap B_k}\big|\sum_{T\in\ZT_\be}\sum_{J\in\cj_\la(T)}Ef_{T}\Id_J\big|^p.
\end{equation}
For each $B_k$, define
\begin{equation}
    f_{k}=\sum_{\substack{T\in\ZT_\be \text{ such that}\\ \exists J\in\cj_\la(T),\, J\cap B_k\not=\varnothing}} f_{T}. 
\end{equation}
Thus, we have
\begin{align}
    \int_{X_2\cap B_k}\big|\sum_{T\in\ZT_\be}\sum_{J\in\cj_\la(T)}Ef_{T}\Id_J\big|^p\lesssim\int_{B_k}\big|\sum_{T\in\ZT_\be}\sum_{J\in\cj_\la(T)}Ef_{T}\Id_J\big|^p\lesssim\int_{B_k}\big|Ef_{k}\big|^p,
\end{align}
Note that for each $T$, there are $\lessapprox R^{\e^{4}}$ many $B_k$ such that $\exists J\in\cj_\la(T), J\cap B_k\not=\varnothing$.
As a consequence, 
\begin{equation}
    \sum_{k}\|f_{k}\|_2^2\lesssim R^{\e^{4}}\|f\|_2^2. 
\end{equation}
Apply \eqref{guth-trick} when $p=p_0$ as an induction hypothesis on each $R^{1-\e^2}$-ball $B_k$ to have
\begin{equation}
    \|Ef_k\|_{L^p(B_k)}^p\leq C_\e R^{(1-\e^2)\e}\|f_k\|_2^2\sup_{\om}\|f_{k,\om}\|_{L_{ave}^2(\om)}^{p-2},
\end{equation}
where $\om$ is an $R^{\e^2-1}\times R^{(\e^2-1)/2}\times\cdots\times R^{(\e^2-1)/2}$-cap.
Since $\sup_{\om}\|f_{k,\om}\|_{L^2_{ave}(\om)}^{p-2}\lesssim\sup_{\theta}\|f_{\theta}\|_{L^2_{ave}(\theta)}^{p-2}$ by $L^2$-orthogonality, sum up all $B_k$ in \eqref{related} to have for $p=p_0$,
\begin{align}
    \int_{X_2}\big|\sum_{T\in\ZT_\be}\sum_{J\in\cj_\la(T)}Ef_{T}\Id_J\big|^p&\lesssim\sum_{B_k}C_\e R^{(1-\e^2)\e} \|f_k\|_2^2\sup_{\om}\|f_{k,\om}\|_{L^2_{ave}(\om)}^{p-2}\\
    &\lesssim R^{-\e^3+\e^4}C_\e R^\e\|f\|_2^2\sup_{\theta}\|f_{\theta}\|_{L^2_{ave}(\theta)}^{p-2}.
\end{align}
Plug this back to \eqref{two-ends-reduction-1} and then \eqref{uniform-on-R1/2-ball} to get \eqref{guth-trick}.

\medskip

\subsection{The two-ends scenario: Three dimensions}

Suppose $\be\in[R^{\e^{4}},R^{\e^2}]$. 
For each $T\in\ZT_\be$, consider the shading $Y(T)=\bigcup_{J\in\cj_\la(T)}J\cap X$.
Then $Y$ is an $(\e^2, \e^{4})$-two-ends, $\la\be$-dense shading.

On the one hand, take
\begin{equation}
\label{mu}
    \mu=R^{2\e^{2}}m(\la\be)^{-3/4}R^{1/4}.
\end{equation}
Note that the configuration $(\ZT_\be, Y)$ is $(\e^2,\e^4)$-two-ends, and the radius of $T$ is $R^{1/2+\e_0}$.
In order to apply Proposition \ref{kakeya-prop}, we may partition $T$ into $\lesssim R^{O(\e_0)}$ many tubes of radius $R^{1/2}$.
Thus, after giving up a loss of $R^{O(\e_0)}$, we can apply the key incidence estimate  Proposition \ref{kakeya-prop} to the $R^{-1}$-dilate of $(\ZT_\be, Y)$ to obtain a set $X_3\subset X$ with
\begin{equation}
\label{multi-R-half}
    \sup_{Q\subset X_3}\#\{T\in\ZT_\be:Y(T)\cap Q\not=\varnothing\}\lessapprox R^{O(\e_0)}\mu
\end{equation}
such that $|X\setminus X_3|\leq R^{-\e^2}|X|$. 
Since $X_2$ is a refinement of $X$,  $|X_2\setminus X_3|\leq R^{-\e^2}|X|\lessapprox R^{-\e^2}|X_2|$. 
Denote by $X_4=X_2\cap X_3$, so that $|X_4|\gtrapprox|X_2|$ and $X_4\subset X_3$. 
Since $\|\sum_{T\in\ZT_\be}\sum_{J\in\cj_{\la}(T)}Ef_{T}\Id_J\|_{L^p(Q)}$ are about the same for all $Q\subset X_2$, 
\begin{equation}
\label{X-4-1}
    \int_{X_2}\big|\sum_{T\in\ZT_\be}\sum_{J\in\cj_\la(T)}Ef_{T}\Id_J\big|^p\lessapprox\int_{X_4}\big|\sum_{T\in\ZT_\be}\sum_{J\in\cj_\la(T)}Ef_{T}\Id_J\big|^p.
\end{equation}

Recall that $\{B_k\}$ is a partition of $B_R$ into $R^{1-\e^2}$-balls. 
For each $B_k$, let $\ZT_{\be, k}=\{T\in\ZT_\be: \exists J\in\cj_{\la}(T), J\cap B_k\not=\varnothing\}$.
Thus,
\begin{align}
    \int_{X_4}&\big|\sum_{T\in\ZT_\be}\sum_{J\in\cj_\la(T)}Ef_{T}\Id_J\big|^4\sim\sum_k\int_{X_4\cap B_k}\big|\sum_{T\in\ZT_\be}\sum_{J\in\cj_\la(T)}Ef_{T}\Id_J\big|^4\\ \label{each-k}
    &\sim \sum_k\int_{X_4\cap B_k}\big|\sum_{T\in\ZT_{\be, k}}Ef_{T}\big|^4\lesssim R^{10\e^2}\sup_k\int_{X_4\cap B_k}\big|\sum_{T\in\ZT_{\be, k}}Ef_{T}\big|^4.
\end{align}
Notice that for each $Q\subset X_4\cap B_k\subset X_3\cap B_k$,
\begin{equation}
    \#\{T\in\ZT_\be:Y(T)\cap Q\not=\varnothing\}=\#\{T\in\ZT_{\be,k}:T\cap Q\not=\varnothing\}.
\end{equation}
Invoke Theorem \ref{refined-decoupling-thm} with $n=3$ and by \eqref{multi-R-half}, we have
\begin{equation}
    \int_{X_4\cap B_k}\big|\sum_{T\in\ZT_{\be,k}}Ef_{T}\big|^4\lessapprox  R^{O(\e_0)} \mu\sum_{T\in\ZT}\big\|Ef_T\big\|_{L^4(w_{B_R})}^4.
\end{equation}
Recall that $\|Ef_T\|_{L^4(w_{B_R})}$ are  the same up to a constant multiple for all $T\in\ZT$ and $\#\ZT(\theta)\sim m$ for all $\theta\in\Theta$.
By Lemma \ref{wpt}, we have
\begin{align}
    \sum_{T\in\ZT(\theta)}&\big\|Ef_T\big\|_{L^4(w_{B_R})}^4\lesssim m^{1-\frac{4}{2}}\Big(\sum_{T\in\ZT(\theta)}\big\|Ef_T\big\|_{L^4(w_{B_R})}^2\Big)^{\frac{4}{2}}\\\label{after-decoupling}
    &\lesssim m^{-1}R^{-2} \Big(\sum_{T\in\ZT(\theta)}\big\|Ef_T\big\|_{L^2(w_{B_R})}^2\Big)^2\lesssim R^{O(\e_0)} m^{-1}R^{-2} \big\|Ef_\theta\big\|_{L^2(w_{B_R})}^4.
\end{align}
Note that, by Plancherel, $\|Ef_\theta\|_{L^2(w_{B_R})}^4\lesssim R^2\|f_\theta\|_2^4\lesssim R\|f_\theta\|_2^2\sup_\theta\|f_\theta\|_{L^2_{ave}}^2$. 
Recall \eqref{mu}, $O(\e_1)\leq\e^2$, and $\be\leq R^{\e^2}$.
Thus, summing up all $\theta$ in \eqref{after-decoupling} and plugging it back to \eqref{each-k}, we have
\begin{align}
\label{l4}
    \int_{X_4}\big|\sum_{T\in\ZT_\be}&\sum_{J\in\cj_\la(T)}Ef_{T}\Id_J\big|^4\lessapprox R^{O(\e^2)}\mu (m R)^{-1}\|f\|_2^2\sup_\theta\|f_\theta\|_{L^2_{ave}}^2\\
    &\lessapprox R^{O(\e^{2})}(\la R)^{-3/4}\|f\|_2^2\sup_\theta\|f_\theta\|_{L^2_{ave}}^2.
\end{align}

On the other hand, since $\be\leq R^{\e^2}$ and by Lemma~\ref{lem: l2},  
\begin{equation}
\label{l2-1}
    \int_{X_4}\big|\sum_{T\in\ZT_\be}\sum_{J\in\cj_\la(T)}Ef_{T}\Id_J\big|^2\lesssim (\la\be R)\|f\|_2^2\leq R^{\e^2}(\la R)\|f\|_2^2.
\end{equation}
Therefore, by \eqref{X-4-1} and \eqref{two-ends-reduction-1}, $\eqref{l4}^{4/7}\cdot\eqref{l2-1}^{3/7}$ gives when $p=22/7$,
\begin{align}
    \big\|\sum_{T\in\ZT}Ef_T\big\|_{L^p(X)}^p\lessapprox R^{O(\e^2)} \|f\|_2^2\sup_\theta\|f_\theta\|_{L^2_{ave}}^{p-2}\leq C_\e R^{\e}\|f\|_2^2\sup_\theta\|f_\theta\|_{L^2_{ave}}^{p-2}.
\end{align}
This proves Theorem \ref{main-thm-reduction} when $n=3$.

\medskip

\subsection{The two-ends scenario: Higher dimensions}

The numerology here is almost identical to that of the three dimensions. 

On the one hand, since $\be\leq R^{\e^2}$ and by Lemma~\ref{lem: l2},  
\begin{equation}
\label{l2-2}
    \int_{X_2}\big|\sum_{T\in\ZT_\be}\sum_{J\in\cj_\la(T)}Ef_{T}\Id_J\big|^2\lesssim (\la \be R)\leq R^{\e^2}(\la R)\|f\|_2^2.
\end{equation}

On the other hand, take 
\begin{equation}
\label{mu-high-d}
    \mu=\Big\{\begin{array}{lc}
    R^{2\e^2}m(\la\be)^{-\frac{2n+7}{7}}R^{-\frac{3n-3}{14}} ,    &  \la\geq R^{-1/8},\\[.5ex]
    R^{2\e^2}m R^{\frac{n-1}{4}},      &  \la\leq R^{-1/8}.
    \end{array}
\end{equation}
Similar to the argument in the previous subsection, we can apply Proposition \ref{kakeya-prop} to the $R^{-1}$-dilate of $(\ZT_\be, Y)$ to obtain a set $X_3\subset X$ with
\begin{equation}
    \sup_{Q\subset X_3}\#\{T\in\ZT_\be:Y(T)\cap Q\not=\varnothing\}\lessapprox R^{O(\e_0)}\mu
\end{equation}
such that $|X\setminus X_3|\leq R^{-\e^2}|X|$. 
Since $X_2$ is a refinement of $X$, $|X_2\setminus X_3|\leq R^{-\e^2}|X|\lessapprox R^{-\e^2}|X_2|$. 
Denote by $X_4=X_2\cap X_3$, so that $|X_4|\gtrapprox|X_2|$ and $X_4\subset X_3$. 
Since $\|\sum_{T\in\ZT_\be}\sum_{J\in\cj_{\la}(T)}Ef_{T}\Id_J\|_{L^p(Q)}$ are about the same for all $Q\subset X_2$,
\begin{equation}
\label{X-4-2}
    \int_{X_2}\big|\sum_{T\in\ZT_\be}\sum_{J\in\cj_\la(T)}Ef_{T}\Id_J\big|^p\lessapprox\int_{X_4}\big|\sum_{T\in\ZT_\be}\sum_{J\in\cj_\la(T)}Ef_{T}\Id_J\big|^p.
\end{equation}

At the decoupling endpoint $p_n=\frac{2(n+1)}{n-1}$, apply Theorem~\ref{refined-decoupling-thm} as we did in the previous subsection so that
\begin{equation}
\label{decoupling-2}
    \int_{X_4}\big|\sum_{T\in\ZT_\be}\sum_{J\in\cj_\la(T)}Ef_{T}\Id_J\big|^{p_n}\lessapprox R^{O(\e^2)}\mu^{\frac{2}{n-1}}\sum_{T\in\ZT}\big\|Ef_T\big\|_{L^{p_n}(w_{B_R})}^{p_n}.
\end{equation}
Also, by Lemma \ref{wpt},
\begin{align}
\label{uncertainty-principle-2}
    \sum_{T\in\ZT(\theta)}\big\|Ef_T\big\|_{L^{p_n}(w_{B_R})}^{p_n}&\lesssim (m R^{\frac{n+1}{2}})^{-\frac{2}{n-1}}\|Ef_\theta\|_2^{p_n}\lesssim m^{-\frac{2}{n-1}} \|f_\theta\|_2^{p_n}\\
    &\lesssim m^{-\frac{2}{n-1}}R^{-1}\|f_\theta\|_2^2\sup_\theta\|f_\theta\|_{L^2_{ave}}^{p_n-2}.
\end{align}
Therefore, we end up with
\begin{equation}
\label{lpn}
    \int_{X_4}\big|\sum_{T\in\ZT_\be}\sum_{J\in\cj_\la(T)}Ef_{T}\Id_J\big|^{p_n}\lessapprox R^{O(\e^2)}(\mu m^{-1})^\frac{2}{n-1}R^{-1}\|f_\theta\|_2^2\sup_\theta\|f_\theta\|_{L^2_{ave}}^{p_n-2}.
\end{equation}
Note that $O(\e_1)\leq \e^2$. Consider the following two separate cases (in both cases,  we prove Theorem \ref{main-thm-reduction} and hence Theorem \ref{main-high-d} for $n\geq4$).

\smallskip

{\bf Case 1: $\la\geq R^{-1/8}$.}  
Apply Proposition \ref{kakeya-prop-high-d} to \eqref{lpn} so that (recall $\be\leq R^{\e^2}$)
\begin{align}
\label{lpn-1}
    \int_{X_4}\big|\sum_{T\in\ZT_\be}\sum_{J\in\cj_\la(T)}Ef_{T}\Id_J\big|^{p_n}&\lessapprox R^{O(\e^2)}(\la^{-\frac{2n+7}{7}}R^{\frac{3n-3}{14}})^\frac{2}{n-1}R^{-1}\|f\|_2^2\sup_\theta\|f_\theta\|_{L^2_{ave}}^{p_n-2}\\
    &\lessapprox R^{O(\e^2)}\la^{-\frac{2(2n+7)}{7(n-1)}}R^{-\frac{4}{7}}\|f\|_2^2\sup_\theta\|f_\theta\|_{L^2_{ave}}^{p_n-2}.
\end{align}
Take $t=\frac{49(n-1)}{77n-95}$. Then $\eqref{l2-2}^{1-t}\cdot \eqref{lpn-1}^{t}$ (in \eqref{l2-2}, we can freely replace $X_2$ by $X_4$, as $X_4\subset X_2$) gives
\begin{equation}
    \int_{X_4}\big|\sum_{T\in\ZT_\be}\sum_{J\in\cj_\la(T)}Ef_{T}\Id_J\big|^{p}\lessapprox M(R,\la)\|f\|_2^2\sup_\theta\|f_\theta\|_{L^2_{ave}}^{p-2},
\end{equation}
where 
\begin{equation}
\label{p-n}
    p=p(n)=\frac{154n+6}{77n-95}=2+\frac{196}{77n-95}=2+\frac{28}{11n}+O(n^{-2}),
 \end{equation}
and, since $\la\geq R^{-1/8}$,
\begin{equation}
    M(R,\la)= R^{O(\e^2)}\la^{-\frac{144}{77n-95}}R^{-\frac{18}{77n-95}}\leq R^{O(\e^2)}.
\end{equation}
This, recalling \eqref{X-4-2} and \eqref{two-ends-reduction-1}, proves Theorem \ref{main-thm-reduction} when $p\geq4$.

\smallskip

{\bf Case 2: $\la\leq R^{-1/8}$.}
Apply Proposition \ref{kakeya-prop-high-d} to \eqref{lpn} so that (recall $\be\leq R^{\e^2}$)
\begin{align}
\label{lpn-2}
    \int_{X_4}\big|\sum_{T\in\ZT_\be}\sum_{J\in\cj_\la(T)}Ef_{T}\Id_J\big|^{p_n}\lessapprox R^{O(\e^2)}R^{-\frac{1}{2}}\|f\|_2^2\sup_\theta\|f_\theta\|_{L^2_{ave}}^{p_n-2}.
\end{align}
Then $\eqref{l2-2}^{4/11}\cdot\eqref{lpn-2}^{7/11}$ (we replace $X_2$ by $X_4$ in \eqref{l2-2}) gives
\begin{equation}
    \int_{X_4}\big|\sum_{T\in\ZT_\be}\sum_{J\in\cj_\la(T)}Ef_{T}\Id_J\big|^{p}\lessapprox M(R,\la)\|f\|_2^2\sup_\theta\|f_\theta\|_{L^2_{ave}}^{p-2},
\end{equation}
where 
\begin{equation}
    p=p(n)=\frac{22n+6}{11(n-1)}=2+\frac{28}{11(n-1)}=2+\frac{28}{11n}+O(n^{-2}),
 \end{equation}
and, since $\la\leq     R^{-1/8}$, $M(R,\la)= R^{O(\e^2)}\la^{\frac{4}{11}}R^{\frac{1}{22}}\leq R^{O(\e^2)}$.
This, recalling \eqref{X-4-2} and \eqref{two-ends-reduction-1}, also proves Theorem \ref{main-thm-reduction} when $p\geq4$.

\begin{remark}
\rm

We can prove Theorem \ref{proof-of-restriction} by taking $\mu=R^{O(\e^2)}\la^{-\frac{n-1}{2}}$ in \eqref{mu-high-d} and following the same proof given in this subsection.

\end{remark}

\bigskip


\section{Appendix}
\label{ending-remark}

\subsection{Numerology on exponents between restriction and Kakeya}

Here we sketch a proof of the following result.
\begin{lemma}
Suppose Conjecture \ref{restriction-conj} is true when $p>p_0$ and $n=3$. 
Then a Kakeya set in $\ZR^3$ must have Hausdorff dimensions at least $s(p_0)=\frac{6-p_0}{p_0-2}$. 
In particular, $s(22/7)=5/2$, and $s(3.2)=7/3$.
\end{lemma}
\begin{proof}
Let $\cT$ be an arbitrary family of directional $R^{-1/2}$-separated $R^{1/2}\times R^{1/2}\times R$-tubes in $\ZR^3$. 
We are going to prove that Conjecture \ref{restriction-conj-local-2} implies 
\begin{equation}
\label{restriction-implies-kakeya}
    \int\big|\sum_{T\in\cT}\Id_T\big|^{p_0/2}\lessapprox R^{p_0}.
\end{equation}
In fact, for any $T\in\cT$ with direction $V(\theta)$, we can choose $f_\theta$ such that $|Ef_{\theta}|\sim R^{-1}$ on $T$ and $|f_\theta|\sim1$.
Let $f=\sum_\theta a_\theta f_\theta$, where $\{a_\theta\}_{\theta}$ are i.i.d. Rademacher random variables.
Thus, by Khintchine's inequality,
\begin{equation}
    \int \big(\sum_{\theta}|Ef_\theta|^2\big)^{p_0/2}\sim\int\big|\sum_{\theta}a_\theta Ef_\theta\big|^{p_0}.
\end{equation}
Apply \eqref{restriction-esti} with this choice of $f$ to get
\begin{equation}
    \int\big|\sum_{\theta}a_\theta Ef_\theta\big|^{p_0}\lesssim\|f\|_{p_0}^{p_0}\lesssim1,
\end{equation}
which, noting $|Ef_\theta|\gtrsim R^{-1}\Id_T$, yields \eqref{restriction-implies-kakeya}.

\smallskip

By H\"older's inequality, \eqref{restriction-implies-kakeya} implies for any $\e>0$  and any $R^{-\e}$-dense shading,
\begin{equation}
    \big|\bigcup_{T\in \cT} Y(T) \big|\gtrapprox  R^{\frac{3}{2}+\frac{s(p_0)}{2} -\e \frac{p_0}{p_0-2}} .
\end{equation}
This proves the lemma.
\end{proof}

\medskip

\subsection{A sketch of Theorem \ref{katz-tao-thm}}
\label{sketch-KT}

Via appropriate rotations, we can assume that all lines in $L$ are quantitatively transverse to the horizontal hyperplane. 
Therefore, by a vertical $\lessapprox\de^{-\e_1}$-stretching, Theorem \ref{katz-tao-thm} is a consequence of the following result.
\begin{theorem}[\cite{Katz-Tao-Kakeya-maximal}, Page 18]
\label{katz-tao-thm-2}
Let $C=10n$ and let $(L,Y)_\de$ be a set of directional $\de$-separated lines in $\ZR^n$ with a two-ends, $\la$-dense shading. 
Suppose that for each $\ell\in L$ and any line segment $J\subset\ell$ of length $\gtrapprox1$, we have $|Y(\ell)\cap N_\de(J)|\leq|\log\de|^{-C}|Y(\ell)|$.
Then when $\la\geq\de^{1/2-\e}$ for any $\e>0$,
\begin{equation}
\label{katz-tao-esti-2}
    |E_L|\gtrapprox \la^\frac{2n+14}{7}\de^{\frac{3n-3}{7}}(\de^{n-1}\#L)^{4/7}.
\end{equation}
\end{theorem}

\smallskip

To illustrate the main idea of \cite{Katz-Tao-Kakeya-maximal}, we first recall a parameterization of non-horizontal lines in $\ZR^n$. Let $p_{(a,b)}(t)=\frac{t}{1+t}(b,0)+\frac{1}{1+t}(a,1)$. Then a non-horizontal line in $\ZR^n$ can be parameterized as $\ell_{(a,b)}=\{p_{(a,b)}(t):t>-1\}$, where $a,b\in\ZR^{n-1}$.
As a result, if the lines in $L$ are all quantitatively transverse to the horizontal plane, the assumption that lines in $L$ are directional $\de$-separated can be reinterpreted as the following: 
The family of $\de$-balls $(E_L)_{-1}:=\{B^{n-1}(a-b,\de): \ell_{(a,b)}\in L\}$ in $\ZR^{n-1}$ are finitely overlapping.
Note that $p_{(a,b)}(-\infty)=a-b$.

Now let us define $(E_L)_{t}:=\{B^{n-1}(p_{(a,b)}(t),\de): \ell_{(a,b)}\in L,\, p_{(a,b)}(t)\in Y(\ell_{a,b})\}$, which is a horizontal slice of $E_L$ at height $1/(1+t)$.
Therefore, to prove \eqref{katz-tao-esti-2}, it suffices to establish a lower bound for $\#(E_L)_t$ for a generic slice $(E_L)_t$, and \cite{Katz-Tao-Kakeya-maximal} uses six slices.
This is essentially a problem about projection, since the horizontal coordinates of the centers $\{p_{(a,b)}(t):\ell_{(a,b)\in L}\}$ are $\{\frac{1}{1+t}(a+tb):\ell_{(a,b)}\in L\}$.
It is referred to as the ``sums and differences" problem in \cite{Katz-Tao-Kakeya-maximal}.

Let $Z$ be a real vector space ($Z=\ZR^{n-1}$ for our purpose) and let $G\subset Z\times Z$ be a finite set.
Define a projection $\pi_t:G\to Z$, $\pi_t(a,b)=a+tb$.
It was shown in \cite{Katz-Tao-Kakeya-maximal} Theorem 3.3 that
\begin{equation}
\label{sum-difference}
    \#\pi_{-1}(G)\lesssim \sup_{t\in\{0,r_1,r_1',r_2,r_2',\infty\}}\pi_t(G)^{7/4}.
\end{equation}
Here $r_1, r_2$ are two arbitrary choices of slopes satisfying $s=r_j+r_j/r_j'$ for some real $s\not=0$, $j=1,2$.

\smallskip

Let us return to the two-ends problem. 
For simplicity, we define $(E_L)_\ZR$ be the family of $\de$-balls contained in $E_L$.
Let $N=\de^{-1}$ to align with the notation in \cite{Katz-Tao-Kakeya-maximal}.
Without loss of generality, we assume that the $\de$-balls in $(E_L)_\ZR$ is contained in the horizontal strip $\{x\in\ZR^n:x_n\in[0,1/2]\}$.
By dyadic pigeonholing on $\#(E_L)_t$, we can find a set of heights $\La$ and further assume that the $N^{-1}$-balls in $E_L$ can be partitioned into $2^{-k}N$ many $N^{-1}$-separated slices $\{(E_L)_t\}_{t\in \La}$ such that $\#(E_L)_t\approx 2^kN^{-1}\#(E_L)_\ZR$.

Let $\ga_n:\ZR^n\to\ZR$ be the projection $\ga_n(x)=x_n$. Since $Y$ is two-ends, by several dyadic pigeonholing, we can obtain the following:
\begin{enumerate}
    \item There are two heights $t_1,t_2\in \La$ such that for a generic $\ell\in L$, we have $\# Q_{t_1,t_2}(\ell)\approx \la^2 N^2$, where
    \begin{equation}
        Q_{t_1,t_2}(\ell):=\{(t_3,t_4)\in\ga_n(Y(\ell))^2: |t_j-t_k|\approx1\text{ for all }1\leq j<k\leq4\}.
    \end{equation}
    \item Let $r(t)=(t-t_1)/(t_2-t)$ and let $s$ be a function $s(t,t')=r(t)+r(t)/r(t')$. Denote $(t,t')\sim_s(t'',t''')$ if $|s(t,t')-s(t'',t''')|=O(N^{-1})$. Then for a generic $\ell\in L$, we have
    \begin{equation}
    \label{lower-bound-lambda}
        \#\{(t_3,t_4,t_5,t_6)\in Q_{t_1,t_2}(\ell)^2:(t_3,t_4)\sim_s (t_5,t_6),|t_3-t_5|\gtrapprox\la^2\}\gtrapprox\la^4 N^3.
    \end{equation}
\end{enumerate}
In \eqref{lower-bound-lambda}, we need the assumption $\la\gtrapprox N^{-1/2}$ to guarantee $|t_3-t_5|\geq N^{-1}$, where $N^{-1}$ is the minimal resolution required to distinguish the $N^{-1}$-balls in $(E_L)_\ZR$.
This is also the place where we use the assumption "$\la\geq\de^{1/2-\e}$" in the statement of Theorem \ref{katz-tao-thm-2}.

\smallskip

After several more dyadic pigeonholing, we can find a number $d\gtrsim\la^2$, a set of lines $L'\subset L$ with $\# L'\gtrapprox\la^6d^{-1}2^{4k}(\#L)$, and four numbers $t_3,t_4,t_5,t_6$ with $|t_3-t_5|\sim d$ so that for all $\ell\in L'$, $(t_3,t_4), (t_5,t_6)\in Q_{t_1,t_2}(\ell)$.

Fix $L'$ and the numbers $t_j, 1\leq j\leq 6$.
Now we want to apply the idea of \eqref{sum-difference}.
Define the produce set
\begin{equation}
    G:=\{(a,b)\in\ga_n^{-1}(t_1)\times\ga_n^{-1}(t_2):a,b\in \ell\text{ for some }\ell\in L'\}.
\end{equation}
Let $(r_1,r_1',r_2,r_2')=(r(t_3), r(t_5), r(t_4), r(t_6))$ and let $s:=s(r_3,r_4)$, so $|s(r_5,r_6)-s|=O(N^{-1})$.
Observe that $|r_1-r_2|\approx d$ and $|r_1|, |r_1'|, |r_2|, |r_2'|\approx1$.

\smallskip

We want to apply \eqref{sum-difference} to the configuration $(G;r_1,r_1',r_2,r_2')$.
However, the fact $|r_1-r_2|\approx d$ will result in a factor depending on $d$ in the upper of $\# G$.
It was shown in \cite{Katz-Tao-Kakeya-maximal} Page 18 that eventually we have 
\begin{equation}
    \# G\lessapprox d^{(1-n)/4}\Big(\frac{2^k\#(E_L)_\ZR}{N}\Big)^{7/4}.
\end{equation}
This implies \eqref{katz-tao-thm-2} since $d\gtrsim\la^2$, $N=\de^{-1}$, $2^k\geq1$, and since $\#(E_L)_\ZR=\de^{-n}|E_L|$.

\bigskip

\bibliographystyle{alpha}
\bibliography{bibli}

\end{document}